\documentclass [12pt]{amsart}
\usepackage{inputenc}

\usepackage{amsmath}
\usepackage{amsthm}
\usepackage{amssymb,amsfonts}
\usepackage{mathtools}

\usepackage[english]{babel}
\usepackage{comment}
\usepackage{hyperref}
\usepackage{bbm}
\usepackage{mathrsfs}

\usepackage{tikz,graphicx,color}
\usepackage{tikz-cd}
\usepackage{tikz-3dplot}

\usetikzlibrary{calc}
\usetikzlibrary{arrows}
\usetikzlibrary{shapes}
\usetikzlibrary{patterns}
\usetikzlibrary{positioning}
\usetikzlibrary{arrows.meta}
\usetikzlibrary{decorations.markings}

\def\positn{0.5}
  \tikzset{decoration={
    markings,
    mark=at position \positn with {\arrow{>}}}}

\usepackage{epstopdf}

\usepackage{pgfplotstable}
\usepackage{pgfplots}

\usepackage{enumerate}

\usepackage[arrow]{xy}
\usepackage{diagbox}
\usepackage{subfig}
\usepackage{chngcntr} 
\usepackage{arcs}
\usepackage{xcolor}
\usepackage{tabu}
\usepackage{booktabs}

\usepackage[margin=1in]{geometry}

\usepackage{fancyhdr}

\newtheorem{theorem}{Theorem}[section]

\newtheorem{proposition}[theorem]{Proposition}
\newtheorem{corollary}[theorem]{Corollary}
\theoremstyle{definition}
\newtheorem{remark}[theorem]{Remark}
\theoremstyle{definition}
\newtheorem{definition}[theorem]{Definition}
\newtheorem{conjecture}[theorem]{Conjecture}

\theoremstyle{definition}

\theoremstyle{definition}
\newtheorem{example}[theorem]{Example}

\def\Tcal{\mathcal{T}}

% ...

\def\C{\mathbb{C}}
\def\R{\mathbb{R}}

\def\Z{\mathbb{Z}}
\def\Q{\mathbb{Q}}
\def\P{\mathbb{P}}

% ...

\newcommand\parr[1]{{({#1})}}
% ...

\def\<{{\langle}}
\def\>{{\rangle}}

\def\l{{\lambda}}
\def\m{{\mu}}

\def\op{{ \operatorname{op}}}

\def\Conv{ \operatorname{Conv}}

\def\proj{ \operatorname{proj}}

\def\wt{\operatorname{wt}}

\begin{document}
\numberwithin{equation}{section}

\title{$R$-systems}

\author{Pavel Galashin}
\address{Department of Mathematics, Massachusetts Institute of Technology,
Cambridge, MA 02139, USA}
\email{{\href{mailto:galashin@mit.edu}{galashin@mit.edu}}}

\author{Pavlo Pylyavskyy}
\address{Department of Mathematics, University of Minnesota,
Minneapolis, MN 55414, USA}
\email{{\href{mailto:ppylyavs@umn.edu}{ppylyavs@umn.edu}}}

\date{\today}



\begin{abstract}
Birational toggling on Gelfand-Tsetlin patterns appeared first in the study of geometric crystals and geometric Robinson-Schensted-Knuth correspondence. Based on these birational toggle relations, Einstein and Propp introduced a discrete dynamical system called \emph{birational rowmotion} associated with a partially ordered set. We generalize birational rowmotion to the class of arbitrary strongly connected directed graphs, calling the resulting discrete dynamical system \emph{the $R$-system}.
We study its integrability from the points of view of singularity confinement and algebraic entropy.  We show that in many cases, singularity confinement in an $R$-system reduces to the Laurent phenomenon either in a cluster algebra, or in a Laurent phenomenon algebra, or beyond both of those generalities, giving rise to many new sequences with the Laurent property possessing rich groups of symmetries. Some special cases of $R$-systems reduce to Somos and Gale-Robinson sequences.
\end{abstract}

\maketitle

\setcounter{tocdepth}{1}
\tableofcontents

\newgeometry{margin=1.35in}

\newcommand{\arxiv}[1]{arXiv preprint \href{https://arxiv.org/abs/#1}{\textup{\texttt{arXiv:#1}}}}

\def\ring{\mathbb{S}}
\def\field{\mathbb{K}} 
\def\fast{\field^\ast}
\def\Init{\mathcal{I}}
\def\X{X}
\def\x{{\mathbf x}}
\newcommand{\RP}[1]{\mathbb{P}^{#1}(\field)}
\newcommand{\AFFF}[2]{\left(#1\right)^{#2}}
\newcommand{\AFF}[1]{\AFFF{\field}{#1}}
\newcommand{\wtt}[2]{\wt(#1,#2)}
\newcommand{\edge}[2]{(#1,#2)}
\def\proj{\pi}

  \def\z{\mathbf{z}}
  \def\X{X}
  \def\y{y}
  \def\Y{Y}
  \def\x{\mathbf{x}}
  \def\P{P}
  \def\QQQ{Q}
  
\def\ta{u}
  
\def\proj{\pi}
\def\rowm{\phi}
\def\bto{\dashrightarrow}
\def\sp{\mathcal{F}_G}

\def\arb{T}
\def\Arb{\Tcal}
  

\def\is{\mathbf{y}}

\allowdisplaybreaks

\section{Introduction}

Recall that a {\it {Gelfand-Tsetlin pattern}} is a collection of nonnegative integers $\{t_{i,j}\}$ for $n \geq i \geq j \geq 1$ satisfying inequalities
\[t_{i,j} \leq t_{i+1,j}\quad \text{and} \quad t_{i,j} \leq t_{i-1,j-1}.\]
A (piecewise-linear) {\it {toggling}} of $t_{i,j}$ was defined by Berenstein and Kirillov in~\cite{KB} to be an operation of replacing $t_{i,j}$ with $t'_{i,j}$ defined by 
\[t'_{i,j}+t_{i,j} = \min(t_{i+1,j}, t_{i-1,j-1}) + \max(t_{i+1,j+1}, t_{i-1,j}).\]
A successive application of piecewise-linear toggle operations  to the vertices of Gelfand-Tsetlin patterns yields an alternative description of the Robinson-Schensted-Knuth correspondence~\cite{Knuth,Schensted}. 
The birational version of the same operation is defined as follows.
$$t'_{i,j}t_{i,j} = \frac{t_{i+1,j}+t_{i-1,j-1}}{t_{i+1,j+1}^{-1} + t_{i-1,j}^{-1}}.$$
It first appeared in the work of Kirillov \cite[Eq.~(4.1)]{Kirillov}. Both versions have been studied extensively ever since, see for example the works of Noumi--Yamada and O'Connell--Sepp\"al\"ainen--Zygouras \cite{NY,OCSZ}. 

Separately, a  combinatorial action of \emph{rowmotion} on order ideals of arbitrary partially ordered sets (\emph{posets} for short) has been studied. Rowmotion consists of {\it {combinatorial toggling}} of poset elements in and out of the ideal, performed at each element 
once in a natural order. Rowmotion was implicitly studied by Fon-Der-Flaass~\cite{FDF}, Sch\"utzenberger~\cite{Schutznbrgr}, Brouwer--Schrijver~\cite{BS},
Panyushev~\cite{Pan}, and Stanley~\cite{StanleyPE}. It was a recent work by Striker and Williams \cite{SW} however that coined the term and systematized the study of rowmotion. 

The two notions have been joined by Einstein and Propp in \cite{PEabstr,PEfull}, where they defined a discrete dynamical system called the \emph{birational rowmotion}. 
It is realized by applying {\it {birational toggling}} operators to parameters associated with poset elements according to the formula
\[t'(v)t(v) = \frac{\sum_{v \lessdot u} t(u)}{\sum_{w \lessdot v} t^{-1}(w)}.\]
Here $u,v,w$ are the elements of the poset and $\lessdot$ denotes its covering relation. The toggle operations are applied once to each element of the poset  in a natural order.
Clearly, in the special case of Gelfand-Tsetlin patterns, this coincides with Kirillov's birational toggling operator.

For certain posets, birational rowmotion is periodic. It was shown by Grinberg and Roby \cite{GR1,GR2} that this is the case for rectangular posets. As it was observed by Max Glick (private communication), their result can be deduced from Zamolodchikov periodicity~\cite{Keller,Volkov,GP1} for rectangular $Y$-systems by relating the values of birational rowmotion to the values of the $Y$-system via a monomial transformation. 
In essence, this can be seen as an application of the \emph{Laurentification} technique introduced by Hone~\cite{HoneSuper} (the term ``Laurentification'' is taken from~\cite{HoneQRT}). Hone uses Laurentification to study the integrability of certain discrete dynamical systems such as the \emph{discrete Painlev\'e equation} of Ramani, Grammaticos, and Hietarinta~\cite{RGH}. In a lot of discrete dynamical systems of interest, Laurentification allows one to translate the \emph{singularity confinement} of Ohta--Tamizhmani--Grammaticos--Ramani~\cite{OTGR} and \emph{algebraic entropy} of Bellon--Viallet~\cite{BV} in terms of  a certain sequence (called the \emph{$\tau$-sequence} in~\cite{HoneSuper}), all of whose entries are Laurent polynomials in their original values. In a lot of cases, the $\tau$-sequence arises from a \emph{cluster algebra} of Fomin-Zelevinsky~\cite{FZ,FZ2,FZ3,FZ4}. 

In this paper, we introduce a new discrete dynamical system which we call \emph{the $R$-system}. It directly  generalizes birational rowmotion of Einstein-Propp from the class of posets to the class of arbitrary strongly connected directed graphs, and its definition is primarily based on the birational toggle relation of~\cite{Kirillov}. Namely, given a directed graph $G=(V,E)$, the $R$-system consists of iterating the map $X\mapsto X'$, where $X=(X_v)_{v\in V}$ and $X'=(X'_v)_{v\in V}$ are assignments of rational functions to the vertices of $G$ satisfying the \emph{toggle relations}
\begin{equation}\label{eq:toggle_intro}
\X_v\X'_v=\left(\sum_{\edge v w\in E} X_w\right)
    \left(\sum_{\edge u v\in E}\frac{1}{X'_u}\right)^{-1},\quad \text{for all $v\in V$}.
  \end{equation}
For example, the values of $X$ and $X'$ satisfying~\eqref{eq:toggle_intro} are given in Figure~\ref{fig:somos_5}. It turns out that if $G$ is not \emph{strongly connected} then the system~\eqref{eq:toggle_intro} either does not make sense or has no solutions $X'$ for generic $X$. When $G$ is strongly connected, $X'$ is uniquely determined as an element of the projective space, and we give an explicit combinatorial formula for $X'$ in terms of $X$ as a certain weighted sum over \emph{arborescences} of $G$ in Theorem~\ref{thm:arb}. This sum has appeared earlier in the context of the \emph{Abelian Sandpile Model} as we discuss in Remark~\ref{rmk:sand}.

\begin{figure}
 \def\scl{0.3}
    \def\lw{0.25}
    \def\tikzscl{2}
    \def\arrwidth{0.4}
    \def\arrht{0.5}
    \def\arrlw{1}
\scalebox{1}{
\begin{tabular}{ccc}
\begin{tikzpicture}[scale=\tikzscl,baseline=(a.base)]
\node[label=below left:{$a$},draw,circle,fill=black,scale=\scl] (a) at (0,0) {};
\node[label=above left:{$b$},draw,circle,fill=black,scale=\scl] (b) at (0,1) {};
\node[label=above right:{$c$},draw,circle,fill=black,scale=\scl] (c) at (1,1) {};
\node[label=below right:{$d$},draw,circle,fill=black,scale=\scl] (d) at (1,0) {};
\draw[postaction={decorate}] (a)--(b);
\draw[postaction={decorate}] (b)--(c);
\draw[postaction={decorate}] (c)--(d);
\draw[postaction={decorate}] (d)--(a);
\draw[postaction={decorate}] (b)--(d);
\end{tikzpicture}&
\begin{tikzpicture}[scale=\tikzscl,baseline=(a.base)]
  \draw[white] (0,-\arrht)--(0,\arrht);
  \node[white] (a) at (0,-\arrht) {};
  \draw[->,line width=\arrlw] (0,0) -- (\arrwidth,0);
\end{tikzpicture}&
\begin{tikzpicture}[scale=\tikzscl,baseline=(a.base)]
\node[label=below:{$bc$},draw,circle,fill=black,scale=\scl] (a) at (0,0) {};
\node[label=above:{$c(c+d)$},draw,circle,fill=black,scale=\scl] (b) at (0,1) {};
\node[label=above:{$d(c+d)$},draw,circle,fill=black,scale=\scl] (c) at (1,1) {};
\node[label=below:{$ac$},draw,circle,fill=black,scale=\scl] (d) at (1,0) {};
\draw[postaction={decorate}] (a)--(b);
\draw[postaction={decorate}] (b)--(c);
\draw[postaction={decorate}] (c)--(d);
\draw[postaction={decorate}] (d)--(a);
\draw[postaction={decorate}] (b)--(d);
\end{tikzpicture}
  \\
  $X$ & & $X'$
\end{tabular}
}
  \caption{\label{fig:somos_5}One iteration of the $R$-system.}
\end{figure}
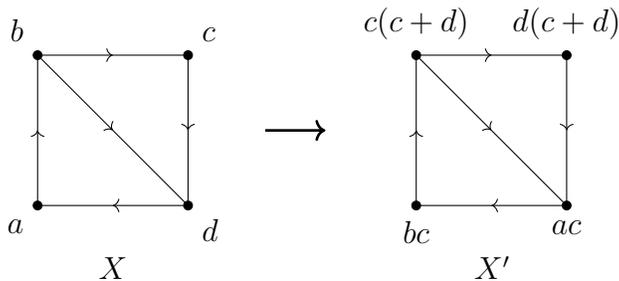

The values of the $R$-system naturally lie in the projective space so one cannot directly talk about the Laurent phenomenon. However, we observe in many cases that the $R$-system possesses a closely related property of \emph{singularity confinement}, which reduces via the same Laurentification technique to the Laurent property of a suitable $\tau$-sequence. Just as the Laurent property, singularity confinement can only occur if a lot of unexpected cancellations happen as the dynamical system proceeds. We show how in many cases, the $\tau$-sequence is a special case of a cluster algebra dynamics, or, more generally, a \emph{Laurent phenomenon algebra}~\cite{LP} dynamics. However, for the most symmetric directed graphs (such as \emph{toric} or \emph{bidirected graphs} that we consider in Sections~\ref{sec:toric-graphs} and~\ref{sec:bidirected-graphs} respectively), the $\tau$-sequence still conjecturally has the Laurent property, but is not a special case of any other Laurent system studied before. Thus one advantage of generalizing birational rowmotion from posets to strongly connected directed graphs is that it provides access to much more symmetric objects one can start with to produce a $\tau$-sequence all of whose entries are Laurent polynomials.

The first part of the paper is devoted to studying some general properties of $R$-systems. We define $R$-systems in Section~\ref{sec:definition-r-system}, and then in Section~\ref{sec:sing-conf-algebr} we discuss the singularity confinement and algebraic entropy properties, as well as how they can be proved using $\tau$-sequences. We then demonstrate Hone's Laurentification procedure in the case of the discrete Painlev\'e equation and compare it to a similar construction for an $R$-system in Section~\ref{sec:laurentification}. We introduce a conserved quantity 
\[\sp(\X)=\sum_{(u,w)\in E}\frac{\X_w}{\X_u}\]
of the $R$-system which we call the \emph{superpotential} in Section~\ref{sec:mirr-symm-superp}. 
It generalizes the superpotential for Type $A$ partial flag varieties studied in the mirror symmetry literature, see \cite{Givental, BCKS, Rietsch1, MR,  RW,LT}. We also discuss the critical points of $\sp$ and return to the $R$-systems associated with  Gelfand-Tsetlin patterns in the same section. 

The second part of the paper is concerned with various special classes of directed graphs. We start by considering several simple sporadic examples in Section~\ref{sec:small-examples}, and then pass to a more systematic approach. In Section~\ref{sec:somos-gale-robinson}, we show how the Somos and Gale-Robinson sequences~\cite{Gale,Robinson} arise as $\tau$-sequences for $R$-systems associated to some directed graphs. We then prove in Section~\ref{sec:circles-with-doubled} that the $R$-system associated with any subgraph of a bidirected cycle admits a $\tau$-sequence coming from a $T$-system in the sense of Nakanishi~\cite{Nakanishi} in a cluster algebra associated with a certain quiver. Just like for rectangular posets the $R$-system reduces to a periodic $Y$-system dynamics, we show in Section~\ref{sec:cylindr-posets-octag} that the $R$-system for a cylindric poset naturally reduces in a similar way to a Laurent phenomenon algebra dynamics. We then continue this series of examples of graphs on surfaces by studying toric directed graphs in Section~\ref{sec:toric-graphs} for which the $\tau$-sequence still appears to have the Laurent property but is not a special case of either a cluster algebra or a Laurent phenomenon algebra. Finally, in Section~\ref{sec:bidirected-graphs}, we study the $R$-system dynamics on bidirected graphs. We show that the \emph{coefficient-free} $R$-system is periodic for all such graphs, however, the \emph{$R$-system with coefficients} conjecturally produces even more symmetric and mysterious $\tau$-sequences with Laurent property. In particular, the $R$-system with coefficients for the complete bidirected graph is universal in the sense that any other $R$-system (in particular, those reducing to Somos and Gale Robinson sequences) can be obtained from it by a suitable specialization of the coefficients, and thus some questions regarding arbitrary $R$-systems can be reduced to this universal case, which, however, appears to be very hard.

{\large\part{General $R$-systems}}

\section{The definition of the $R$-system}\label{sec:definition-r-system}
Let $G=(V,E)$ be a directed simple graph (\emph{digraph} for short), that is, $G$ is not allowed to have loops or multiple edges with the same start and end. We say that $G$ is \emph{strongly connected} if for any two vertices $u,v\in V$, there exists a directed path from $u$ to $v$ in $G$. Let $\ring$ be a ring and $\field$ be its field of fractions.\footnote{We always assume that $\ring$ is a unique factorization domain (UFD) with unity and that it is a ring of characteristic zero. The only rings that we consider in this text are (Laurent) polynomial rings over $\Z$ or $\Z$ itself.}
A \emph{weighted digraph} $G=(V,E,\wt)$ is a digraph with a \emph{weight function} $\wt:E\to \field^\ast$ taking non-zero values in $\field$.  
Every digraph has a canonical weight function that assigns a weight of one to every edge.

Let $U$ be a non-empty finite set. We denote by $\RP{U}$ the \emph{$(|U|-1)$-dimensional projective space over $\field$}, that is, the set of all vectors $\X=(\X_u)_{u\in U}\in \field^U\setminus\{0\}$ modulo simultaneous rescalings by non-zero scalars $\l\in\fast$. If $U=[n]:=\{1,2,\dots,n\}$ then we also write $\X=(\X_1:\X_2:\dots:\X_n)$. 

We now introduce our main object of study. Let $G=(V,E,\wt)$ be a weighted digraph. We  consider the following system of equations in the variables $\X=(\X_v)_{v\in V}$ and $\X'=(\X'_v)_{v\in V}$:
\begin{equation}\label{eq:toggle}
\X_v\X'_v=\left(\sum_{\edge v w\in E}\wtt v w X_w\right)
    \left(\sum_{\edge u v\in E}\frac{\wtt u v}{X'_u}\right)^{-1},\quad \text{for all $v\in V$}.
\end{equation}

A more symmetric way of writing down these equations is
\begin{equation}\label{eq:toggle_sym}
  \sum_{\edge u v\in E}\wtt u v\frac{X'_v}{X'_u}
  =\sum_{\edge v w\in E}\wtt v w \frac{X_w}{X_v}
,\quad \text{for all $v\in V$}.
\end{equation}
It is clear  that $\X$ and $\X'$ give a solution to~\eqref{eq:toggle} if and only if $\l\X$ and $\m\X'$ give a solution to~\eqref{eq:toggle}. Here $\l,\m\in\fast$ are non-zero scalars and $\l\X:=(\l\X_v)_{v\in V}$. Thus~\eqref{eq:toggle} can be considered as a system of equations on $\RP V$ (where we treat $X$ as the given input and $X'$ as the output that we need to find).

To explain our first main result, we need to introduce the notion of an~\emph{arborescence}.
\begin{definition}
  Given a strongly connected digraph $G$ and a vertex $v\in V$, an \emph{arborescence rooted at $v$} is a map: $\arb:V\setminus \{v\}\to V$ such that 
\begin{itemize}
 \item for any $u\in V\setminus \{v\}$, we have $(u,\arb(u))\in E$;
 \item for any $u\in V\setminus \{v\}$, there exists $k \in \mathbb Z_{>0}$ such that $\arb^k(u)  = v$. 
\end{itemize}
\end{definition}
In other words, an arborescence rooted at $v$ is a collection of edges of $G$ that together form a spanning tree oriented towards $v$. The set of all arborescences rooted at $v$ is denoted by $\Arb(G,v)$. Given a point $\X\in\RP{V}$ with non-zero coordinates, the \emph{weight} $\wt(\arb;\X)\in\field$ of $\arb$ is defined to be
\[\wt(\arb;\X):=\prod_{u\in V\setminus\{v\}}\wtt{u}{\arb(u)} \frac{\X_{\arb(u)}}{\X_u}.\]

\begin{example}
  \label{ex:somos_4}
  Let $\field=\Q(x_1,x_2,x_3)$ be the field of rational functions in three variables. Consider the digraph $G$ with vertex set $V=\{1,2,3\}$ and edge set $E=\{(1,2),(2,3),(3,1),(1,3)\}$ shown in Figure~\ref{fig:somos_4} (left). We consider the \emph{coefficient-free} $R$-system, i.e., we assume that all edge weights are equal to $1$.
  
Suppose that $\X=(x_1:x_2:x_3)$. There are four arborescences $T^\parr1,\dots,T^\parr4$ in $G$, shown in Figure~\ref{fig:somos_4} (right) together with their weights.


    \begin{figure}
      \centering
      \def\tikzsc{1.2}
      \def\sclbx{0.8}
      \begin{tabular}{c|cccc}
\scalebox{\sclbx}{
      \begin{tikzpicture}[scale=\tikzsc] 
        \node[draw,circle] (1) at (-30:1) {$1$};
        \node[draw,circle] (2) at (90:1) {$2$};
        \node[draw,circle] (3) at (210:1) {$3$};
        \draw[postaction={decorate}] (1) -- (2);
        \draw[postaction={decorate}] (2) -- (3);
        \draw[postaction={decorate}] (1) to[bend right=10] (3);
        \draw[postaction={decorate}] (3) to[bend right=10] (1);
      \end{tikzpicture}} & 
\scalebox{\sclbx}{
      \begin{tikzpicture}[scale=\tikzsc] 
        \node[draw,circle] (1) at (-30:1) {$1$};
        \node[draw,circle] (2) at (90:1) {$2$};
        \node[draw,circle] (3) at (210:1) {$3$};
        \draw[postaction={decorate}] (2) -- (3);
        \draw[postaction={decorate}] (3) to[bend right=10] (1);
      \end{tikzpicture}} &
\scalebox{\sclbx}{
      \begin{tikzpicture}[scale=\tikzsc] 
        \node[draw,circle] (1) at (-30:1) {$1$};
        \node[draw,circle] (2) at (90:1) {$2$};
        \node[draw,circle] (3) at (210:1) {$3$};
        \draw[postaction={decorate}] (1) -- (2);
        \draw[postaction={decorate}] (3) to[bend right=10] (1);
      \end{tikzpicture}} &
\scalebox{\sclbx}{
      \begin{tikzpicture}[scale=\tikzsc] 
        \node[draw,circle] (1) at (-30:1) {$1$};
        \node[draw,circle] (2) at (90:1) {$2$};
        \node[draw,circle] (3) at (210:1) {$3$};
        \draw[postaction={decorate}] (1) -- (2);
        \draw[postaction={decorate}] (2) -- (3);
      \end{tikzpicture}} &
\scalebox{\sclbx}{
      \begin{tikzpicture}[scale=\tikzsc] 
        \node[draw,circle] (1) at (-30:1) {$1$};
        \node[draw,circle] (2) at (90:1) {$2$};
        \node[draw,circle] (3) at (210:1) {$3$};
        \draw[postaction={decorate}] (2) -- (3);
        \draw[postaction={decorate}] (1) to[bend right=10] (3);
      \end{tikzpicture}} \\
        $G$ & $\arb^\parr1$ & $\arb^\parr2$& $\arb^\parr3$ & $\arb^\parr4$\\
            &  $\wt=\frac{x_3x_1}{x_2x_3}$ & $\wt=\frac{x_1x_2}{x_3x_1}$& $\wt=\frac{x_2x_3}{x_1x_2}$ & $\wt=\frac{x_3^2}{x_1x_2}$\\
      \end{tabular}
      \caption{\label{fig:somos_4}A strongly connected digraph $G$ from Example~\ref{ex:somos_4} (left). Its four arborescences and their weights (right).}
    \end{figure}
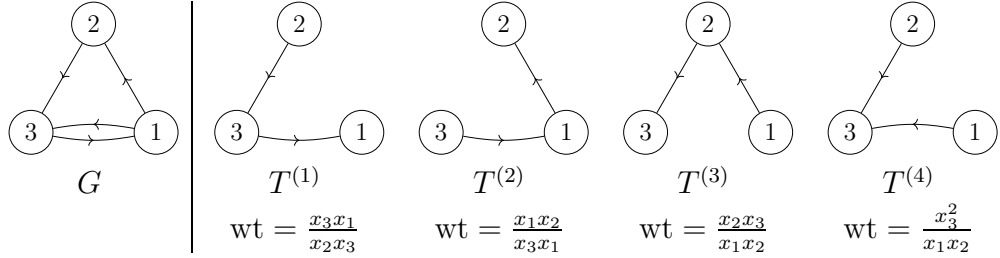

  \end{example}

\def\DOM{O}
\begin{theorem}
  \label{thm:arb}
  Let $G=(V,E,\wt)$ be a strongly connected weighted digraph. Then there exists a birational map $\rowm:\RP V\bto \RP V$ defined on some Zariski open subset $\DOM\subset\RP V$ such that for each $\X\in\DOM$, there exists a unique $\X'\in\RP V$ that gives a solution to~\eqref{eq:toggle}, and in this case we have $\rowm(\X)=\X'$. Explicitly, $\X'=(\X'_v)_{v\in V}$ is given by
    \begin{equation}
      \label{eq:arb} 
      \X'_v=\frac{\X_v}{\sum_{\arb\in\Arb(G,v)}\wt(\arb;\X)}.
    \end{equation}
  \end{theorem}
  
\begin{example}\label{ex:somos_4_cont}
For the digraph from Example~\ref{ex:somos_4}, according to Theorem~\ref{thm:arb}, the unique solution to~\eqref{eq:toggle} is given by
    \begin{equation}
     \label{eq:solution_somos} 
    x_1'=\frac{x_1x_2x_3}{x_1x_3}=x_2;\quad
      x_2'=\frac{x_1x_2x_3}{x_1x_2}=x_3;\quad
      x_3'=\frac{x_1x_2x_3}{x_2x_3+x_3^2}=\frac{x_1x_2}{x_2+x_3}.
    \end{equation}
    Let us check that setting $\X':=(x_1':x_2':x_3')$ indeed gives a solution to~\eqref{eq:toggle}:
    \begin{equation}
      \begin{split}
        x_1x_1'&=(x_2+x_3) \left(\frac1{x_3'}\right)^{-1}= x_1x_2;\\
        x_2x_2'&=(x_3) \left(\frac1{x_1'}\right)^{-1}= x_2x_3;\\
        x_3x_3'&=(x_1) \left(\frac1{x_1'}+\frac1{x_2'}\right)^{-1}= \frac{x_1x_2x_3}{x_2+x_3}.\\
      \end{split}
    \end{equation}
    One easily verifies directly that this solution is unique as an element of $\RP{V}$. Thus for the digraph in Figure~\ref{fig:somos_4}, the map $\rowm$ sends $(x_1:x_2:x_3)$ to $\left(x_2:x_3:\frac{x_1x_2}{x_2+x_3}\right)$.
\end{example}

\begin{definition}\label{dfn:R}
  Let $G=(V,E,\wt)$ be a strongly connected weighted digraph. The \emph{$R$-system} associated with $G$ is a discrete dynamical system consisting of iterative application of the map $\rowm$. More precisely, for  $\Init\in\RP V$, the $R$-system is a family $(R(t))_{t\geq 0}$ of elements of $\RP V$ defined for each nonnegative integer $t$ via $R(t)=\rowm^t(\Init)$.
\end{definition}

\begin{remark}\label{rmk:subtr_free}
For each $t\geq 0$, $R(t)$ is defined for all $\Init$ belonging to some Zariski open set in $\RP V$. However, if $\field$ is a rational field in some variables and $\Init$ can be written as a subtraction-free rational expression in these variables then the same is true for $\rowm(\Init)$ and thus in this case $R(t)$ is defined for all $t\geq 0$.
\end{remark}

\def\tb{s}
\begin{remark}\label{rmk:poset}

  Let us explain how our notion of an $R$-system is a direct generalization of birational rowmotion of~\cite{PEabstr}. Given a finite poset $(P,\leq)$, denote by $\hat P$ the poset obtained from $P$ by attaching a minimum $\hat0$ and a maximum $\hat1$. One then constructs a digraph $G=G(P)$ as follows: first let $G'$ be obtained from the Hasse diagram of $\hat P$ by orienting every edge upwards. In other words, $G'$ contains an edge $(u,v)$ if and only if $u\lessdot v$ in $\hat P$. Then, $G$ is obtained from $G'$ by identifying the vertices $\hat0$ and $\hat1$ into a new vertex $\tb$. Thus the vertex set of $G$ is equal to $V=P\cup\{\tb\}$. See Figure~\ref{fig:poset} for an example. We consider the canonical weight function assigning weight $1$ to every edge.
  
  Clearly $G$ is a strongly connected digraph, and it is a non-trivial consequence of Theorem~\ref{thm:arb} that the $R$-system dynamics associated with $G$ is identical to the birational rowmotion dynamics associated with $P$. To see that, note that one iteration of birational rowmotion gives the unique solution to the system of equations obtained from~\eqref{eq:toggle} by removing the equation corresponding to $\tb$. However, by Theorem~\ref{thm:arb}, the solution to the whole system~\eqref{eq:toggle} exists and thus it has to coincide with the output of birational rowmotion. We also note that the $R$-system operates on $\RP V$ while birational rowmotion operates on $\field^P$ so in order to perform the reduction one needs to rescale the entries of $R(t)$ so that $R_\tb(t)=1$ for all $t\geq 0$.

      \def\nodesc{0.3}
      \def\tikzsc{0.5}
  \begin{figure}
    \begin{tabular}{ccccc}

  \begin{tikzpicture}[scale=\tikzsc]
    \node[draw=white,circle,fill=white,scale=\nodesc] (Z) at (0,-2) {$ $};
    \node[draw=white,circle,fill=white,scale=\nodesc] (O) at (0,6) {$ $};
    \node[anchor=north,text=white] (ZZ) at (Z.south) {$\hat0$};
    \node[anchor=south,text=white] (OO) at (O.north) {$\hat1$};
    \node[draw,circle,fill=black,scale=\nodesc] (A) at (0,0) {$ $};
    \node[draw,circle,fill=black,scale=\nodesc] (B) at (-1,2) {$ $};
    \node[draw,circle,fill=black,scale=\nodesc] (C) at (1,2) {$ $};
    \node[draw,circle,fill=black,scale=\nodesc] (D) at (0,4) {$ $};
    \node[draw,circle,fill=black,scale=\nodesc] (E) at (2,4) {$ $};
    \node[draw,circle,fill=black,scale=\nodesc] (F) at (1,4) {$ $};
    \node[draw,circle,fill=black,scale=\nodesc] (G) at (2,0) {$ $};
    \node[draw,circle,fill=black,scale=\nodesc] (H) at (3,2) {$ $};
    \draw (G)--(H);
    \draw (H)--(E);
    \draw (A)--(B);
    \draw (A)--(C)--(D);
    \draw (C)--(E);
    \draw (C)--(F);
    \draw (G)--(C);
  \end{tikzpicture}
      &
&
  \begin{tikzpicture}[scale=\tikzsc]
    \node[draw,circle,fill=black,scale=\nodesc] (Z) at (0,-2) {$ $};
    \node[draw,circle,fill=black,scale=\nodesc] (O) at (0,6) {$ $};
    \node[anchor=north] (ZZ) at (Z.south) {$\hat0$};
    \node[anchor=south] (OO) at (O.north) {$\hat1$};
    \node[draw,circle,fill=black,scale=\nodesc] (A) at (0,0) {$ $};
    \node[draw,circle,fill=black,scale=\nodesc] (B) at (-1,2) {$ $};
    \node[draw,circle,fill=black,scale=\nodesc] (C) at (1,2) {$ $};
    \node[draw,circle,fill=black,scale=\nodesc] (D) at (0,4) {$ $};
    \node[draw,circle,fill=black,scale=\nodesc] (E) at (2,4) {$ $};
    \node[draw,circle,fill=black,scale=\nodesc] (F) at (1,4) {$ $};
    \node[draw,circle,fill=black,scale=\nodesc] (G) at (2,0) {$ $};
    \node[draw,circle,fill=black,scale=\nodesc] (H) at (3,2) {$ $};
    \draw (G)--(H);
    \draw (H)--(E);
    \draw (A)--(B);
    \draw (A)--(C)--(D);
    \draw (C)--(E);
    \draw (C)--(F);
    \draw (G)--(C);
    \draw (Z)--(A);
    \draw (Z)--(G);
    \draw (B)--(O);
    \draw (D)--(O);
    \draw (E)--(O);
    \draw (F)--(O);
  \end{tikzpicture}
      & &

  \begin{tikzpicture}[scale=\tikzsc    ] 
    \node[draw=white,circle,fill=white,scale=\nodesc] (Z) at (0,-2) {$ $};
    \node[draw=white,circle,fill=white,scale=\nodesc] (O) at (0,6) {$ $};
    \node[anchor=north,text=white] (ZZ) at (Z.south) {$\hat0$};
    \node[anchor=south,text=white] (OO) at (O.north) {$\hat1$};
    \node[draw,circle,fill=black,scale=\nodesc] (A) at (0,0) {$ $};
    \node[draw,circle,fill=black,scale=\nodesc] (B) at (-1,2) {$ $};
    \node[draw,circle,fill=black,scale=\nodesc] (C) at (1,2) {$ $};
    \node[draw,circle,fill=black,scale=\nodesc] (D) at (0,4) {$ $};
    \node[draw,circle,fill=black,scale=\nodesc] (E) at (2,4) {$ $};
    \node[draw,circle,fill=black,scale=\nodesc] (F) at (1,4) {$ $};
    \node[draw,circle,fill=black,scale=\nodesc] (G) at (2,0) {$ $};
    \node[draw,circle,fill=black,scale=\nodesc] (H) at (3,2) {$ $};
    \draw[postaction={decorate}] (G)--(H);
    \draw[postaction={decorate}] (H)--(E);
    
    \node[draw,circle,fill=black,scale=\nodesc] (TB) at (-3,2) {$ $};
    \node[anchor=east] (TBTB) at (TB.west) {$\tb$};
    
    \draw[postaction={decorate}] (A)--(B);
    \draw[postaction={decorate}] (A)--(C);
    \draw[postaction={decorate}] (C)--(D);
    \draw[postaction={decorate}] (C)--(E);
    \draw[postaction={decorate}] (C)--(F);
    \draw[postaction={decorate}] (G)--(C);
    \draw[postaction={decorate}] (B)--(TB);
    \draw[postaction={decorate}] (D)--(TB);
    \draw[postaction={decorate}] (E) to[bend right=70] (TB);
    \draw[postaction={decorate}] (F) to[bend right=50] (TB);
    \draw[postaction={decorate}] (TB) to (A);
    \draw[postaction={decorate}] (TB) to[bend right=70] (G);
    
  \end{tikzpicture}\\
      $P$ & & $\hat P$ && $G(P)$

    \end{tabular}
  \caption{\label{fig:poset} Transforming a poset $P$ (left) into a strongly connected digraph $G(P)$ (right). Using this construction, birational rowmotion of Einstein-Propp~\cite{PEabstr} for  $P$ becomes a special case of the $R$-system associated with $G(P)$, see Remark~\ref{rmk:poset}.}
\end{figure}
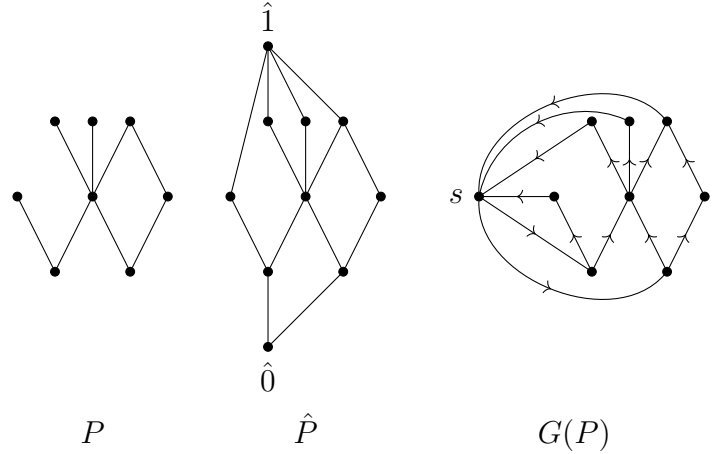
\end{remark}

\begin{example}
  Consider a poset $P$ whose Hasse diagram is shown in Figure~\ref{fig:birational_rowmotion} (left). We see that after applying birational rowmotion to the initial values $a,b,c,d$, the toggle relation at the extra vertex $\tb$ of $G(P)$ shown in Figure~\ref{fig:birational_rowmotion} (right) is automatically satisfied:
  \[1\cdot 1 = (a+b)\left(\left(\frac{c+d}{ac}\right)^{-1}+\left(\frac{c+d}{ad+bc+bd}\right)^{-1}\right)^{-1}.\]
  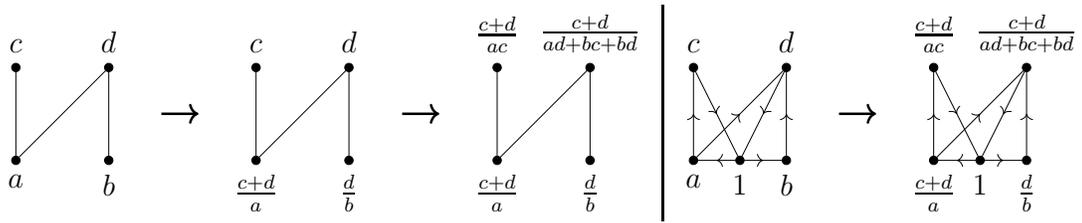
\begin{figure}
    \def\scl{0.3}
    \def\lw{0.25}
    \def\tikzscl{1.3}
    \def\arrwidth{0.4}
    \def\arrht{0.5}
    \def\arrlw{1}
\scalebox{0.95}{
\begin{tabular}{ccccc|ccc}
\begin{tikzpicture}[scale=\tikzscl,baseline=(a.base)]
\node[label=below:{$a$},draw,circle,fill=black,scale=\scl] (a) at (0,0) {};
\node[label=below:{$b$},draw,circle,fill=black,scale=\scl] (b) at (1,0) {};
\node[label=above:{$c$},draw,circle,fill=black,scale=\scl] (c) at (0,1) {};
\node[label=above:{$d$},draw,circle,fill=black,scale=\scl] (d) at (1,1) {};
\draw[line width=\lw] (c)--(a)--(d)--(b);
\end{tikzpicture}&
\begin{tikzpicture}[scale=\tikzscl,baseline=(a.base)]
  \draw[white] (0,-\arrht)--(0,\arrht);
  \node[white] (a) at (0,-\arrht) {};
  \draw[->,line width=\arrlw] (0,0) -- (\arrwidth,0);
\end{tikzpicture}&
\begin{tikzpicture}[scale=\tikzscl,baseline=(a.base)]
\node[label=below:{$\frac{c+d}a$},draw,circle,fill=black,scale=\scl] (a) at (0,0) {};
\node[label=below:{$\frac d b$},draw,circle,fill=black,scale=\scl] (b) at (1,0) {};
\node[label=above:{$c$},draw,circle,fill=black,scale=\scl] (c) at (0,1) {};
\node[label=above:{$d$},draw,circle,fill=black,scale=\scl] (d) at (1,1) {};
\draw[line width=\lw] (c)--(a)--(d)--(b);
\end{tikzpicture}&
\begin{tikzpicture}[scale=\tikzscl,baseline=(a.base)]
  \draw[white] (0,-\arrht)--(0,\arrht);
  \node[white] (a) at (0,-\arrht) {};
  \draw[->,line width=\arrlw] (0,0) -- (\arrwidth,0);
\end{tikzpicture}&
\begin{tikzpicture}[scale=\tikzscl,baseline=(a.base)]
\node[label=below:{$\frac{c+d}a$},draw,circle,fill=black,scale=\scl] (a) at (0,0) {};
\node[label=below:{$\frac d b$},draw,circle,fill=black,scale=\scl] (b) at (1,0) {};
\node[label=above:{$\frac{c+d}{ac}$},draw,circle,fill=black,scale=\scl] (c) at (0,1) {};
\node[label=above:{$\frac{c+d}{ad+bc+bd}$},draw,circle,fill=black,scale=\scl] (d) at (1,1) {};
\draw[line width=\lw] (c)--(a)--(d)--(b);
\end{tikzpicture}&
\begin{tikzpicture}[scale=\tikzscl,baseline=(a.base)]
\node[label=below:{$1$},draw,circle,fill=black,scale=\scl] (tb) at (0.5,0) {};
\node[label=below:{$a$},draw,circle,fill=black,scale=\scl] (a) at (0,0) {};
\node[label=below:{$b$},draw,circle,fill=black,scale=\scl] (b) at (1,0) {};
\node[label=above:{$c$},draw,circle,fill=black,scale=\scl] (c) at (0,1) {};
\node[label=above:{$d$},draw,circle,fill=black,scale=\scl] (d) at (1,1) {};
\draw[postaction={decorate}] (a)--(c);
\draw[postaction={decorate}] (a)--(d);
\draw[postaction={decorate}] (b)--(d);
\draw[postaction={decorate}] (c)--(tb);
\draw[postaction={decorate}] (d)--(tb);
\draw[postaction={decorate}] (tb)--(a);
\draw[postaction={decorate}] (tb)--(b);
\end{tikzpicture}&
\begin{tikzpicture}[scale=\tikzscl,baseline=(a.base)]
  \draw[white] (0,-\arrht)--(0,\arrht);
  \node[white] (a) at (0,-\arrht) {};
  \draw[->,line width=\arrlw] (0,0) -- (\arrwidth,0);
\end{tikzpicture}&
\begin{tikzpicture}[scale=\tikzscl,baseline=(a.base)]
\node[label=below:{$1$},draw,circle,fill=black,scale=\scl] (tb) at (0.5,0) {};
\node[label=below:{$\frac{c+d}a$},draw,circle,fill=black,scale=\scl] (a) at (0,0) {};
\node[label=below:{$\frac d b$},draw,circle,fill=black,scale=\scl] (b) at (1,0) {};
\node[label=above:{$\frac{c+d}{ac}$},draw,circle,fill=black,scale=\scl] (c) at (0,1) {};
\node[label=above:{$\frac{c+d}{ad+bc+bd}$},draw,circle,fill=black,scale=\scl] (d) at (1,1) {};
\draw[postaction={decorate}] (a)--(c);
\draw[postaction={decorate}] (a)--(d);
\draw[postaction={decorate}] (b)--(d);
\draw[postaction={decorate}] (c)--(tb);
\draw[postaction={decorate}] (d)--(tb);
\draw[postaction={decorate}] (tb)--(a);
\draw[postaction={decorate}] (tb)--(b);
\end{tikzpicture}
\end{tabular}
}
  \caption{\label{fig:birational_rowmotion} Applying birational rowmotion to a poset $P$ with four elements (left) produces a solution to the $R$-system associated with $G(P)$ (right).}
\end{figure}
\end{example}

\begin{remark} \label{rmk:sand}
The denominator in the right hand side of~\eqref{eq:arb} has a meaning in the theory of \emph{sandpiles}~\cite{Biggs,BN,HLMPP}. For example, after setting $X_u=1$ for all $u\in V$ and $\wtt u w=1$ for all $u,w\in V$, the value of the denominator equals the size of the \emph{critical group} (also known as the \emph{sandpile group} or the \emph{Jacobian group}) of $G$ with $v$ chosen as the designated sink. Without setting $X_u=1$, this formula gives an expression for the cokernel of the weighted Laplacian matrix. It would be interesting to see if there is a deeper relationship between these two areas.
\end{remark}

We say that the $R$-system is \emph{coefficient-free} if $\wt(e)=1$ for all $e\in E$, otherwise we sometimes call it the \emph{$R$-system with coefficients}.

Let us define two formal sets of variables $\x_V=(x_v)_{v\in V}$ and $\x_E=(x_e)_{e\in E}$.
\begin{definition}
Let $\ring=\Q[\x_V,\x_E]$, and consider a strongly connected digraph $G=(V,E)$. Define $\wt',\wt'':E\to \ring$  by $\wt'(e)=1$, $\wt''(e)=x_e$. The $R$-system associated with $(V,E,\wt'')$ given by initial conditions $R(0)=\x_V$ is called the \emph{universal $R$-system with coefficients associated with $G$}. Similarly, the $R$-system associated with $(V,E,\wt')$ with initial conditions $R(0)=\x_V$ is called the \emph{universal coefficient-free $R$-system associated with $G$}. 
\end{definition}

  \section{Singularity confinement and algebraic entropy}\label{sec:sing-conf-algebr}
  
  Our main motivation for studying $R$-systems comes from various properties that they share with known integrable systems. Let us first give the necessary definitions.

  Recall that $\field$ is the field of fractions of some unique factorization domain $\ring$. We denote by $\ring^\ast$ the set of invertible elements of $\ring$ and we say that an element $r\in\ring\setminus \ring^\ast$ is \emph{irreducible} if whenever $r=r_1r_2$, either $r_1$ or $r_2$ is an invertible element of $\ring$. Two elements $r_1,r_2\in\ring$ are \emph{coprime} if every $r\in\ring$ that divides both $r_1$ and $r_2$ is invertible, i.e., $r\in\ring^\ast$.

\def\multop{\operatorname{mult}}
\newcommand{\mult}[2]{\multop_{#1}(#2)}

  \begin{definition}
Let $x\in \ring$ be an element and let $r\in\ring$ be an irreducible element. Define the \emph{multiplicity} $\mult r x$ to be the maximal integer $m\geq 0$ such that $x$ is divisible by $r^m$.
  \end{definition}

  \def\Xt{\tilde{\X}}
  \def\allones{(1,1,\dots,1)}
  \begin{definition}
Let $V$ be a set. Given an element $\X=(\X_v)_{v\in V}\in \RP V$, its \emph{canonical form} is any vector $\Xt=(\Xt_v)_{v\in V}\in \ring^V$ such that the greatest common divisor of the elements $\{\Xt_v\mid v\in V\}$ is $1$ and $\Xt=\X$ in $\RP V$.
\end{definition}
For instance, the canonical form of $\left(x_2:x_3:\frac{x_1x_2}{x_2+x_3}\right)$ from Example~\ref{ex:somos_4} is
\[\left(x_2(x_2+x_3),x_3(x_2+x_3),x_1x_2\right).\]
Note that the canonical form of an element of $\RP V$ is defined uniquely up to a multiplication by a common unit element $r\in \ring^\ast$.
\begin{definition}
  Given an element $\X=(\X_v)_{v\in V}\in \RP V$ and an irreducible element $r\in \ring$, define $\mult r \X=(\mult r {\Xt_v})_{v\in V}\in\Z^V$, where $\Xt$ is the canonical form of $\X$.
\end{definition}
Thus $\mult r \X$ is a vector of nonnegative integers with at least one zero coordinate.

When $\ring$ is a polynomial ring in some number of variables, we denote by $\deg(r)$ the degree of the polynomial $r\in\ring$, and for $\X\in\RP V$, we let $\deg(\X)$ be the maximal degree of $\Xt_v$ over all $v\in V$, where again  $\Xt\in\ring^V$ is the canonical form of $\X$.

We are now ready to define the two important properties that are widely regarded as integrability tests in the area of integrable systems. 

\begin{definition}
  Let $G=(V,E,\wt)$ be a strongly connected weighted digraph and let $R(t)$ be the associated $R$-system. We say that $R(t)$ has the \emph{singularity confinement property} if for any irreducible element $r\in \ring$, we have $\mult r {R(t)}=0\in \Z^V$ for infinitely many $t\geq 0$. 
\end{definition}

\begin{definition}
   Let $G=(V,E)$ be a strongly connected digraph. The \emph{algebraic entropy} $d(G)$ of the associated universal $R$-system with coefficients is the limit
  \begin{equation}\label{eq:entropy_def}
    d(G):=\lim_{t\to\infty} \frac{\log\deg(R(t))}{t}.
  \end{equation}
\end{definition}

In a lot of cases, the properties of singularity confinement and having zero algebraic entropy indicate that the discrete dynamical system in question is integrable in the sense of possessing a large amount of symmetries, or a large number of conserved quantities. 
Singularity confinement was introduced by Grammaticos, Ramani, and Papageorgiou~\cite{GRP}, and algebraic entropy was introduced by Bellon-Viallet~\cite{BV}.

\begin{remark}
  We note that singularity confinement actually is a surprising property of an $R$-system: indeed, suppose that $r\in \ring$ is an irreducible factor that appears in $\X_v(t)$ for some $v\in V$ but does not appear in $\X_u(t+1)$ for any $u\in V$. Since $r$ cannot appear in $\X_u(t)$ for all $u\in V$ when $\X(t)$ is written in canonical form, let us assume that $r$ does not appear in $\X_u(t)$ for at least one outgoing neighbor $u$ of $v$. Then the left hand side of~\eqref{eq:toggle} is trivially divisible by $r$, and therefore the same should be true for the right hand side, however, none of the two factors in the right hand side is such that every term in it is divisible by $r$. So in order for the right hand side to be divisible by $r$, we need to add up several terms, each of which is not divisible by $r$, so that their sum is divisible by $r$. This is exactly the kind of a ``fortuitous cancellation'' that usually is associated with the Laurent phenomenon.
\end{remark}

  \section{Laurentification}\label{sec:laurentification}

  In a lot of cases, studying singularity confinement and algebraic entropy of a discrete dynamical system reduces to studying the same phenomena for another recursive sequence (called the \emph{$\tau$-sequence} in~\cite{HoneSuper}) that has the \emph{Laurent property}, i.e., all of its entries are Laurent polynomials in the initial variables. In~\cite{HoneQRT}, this technique is called \emph{Laurentification}.   Let us illustrate it via a well-studied example of the \emph{discrete Painlev\'e equation $d-P_I$} of Ramani-Grammaticos-Hietarinta~\cite{RGH}.

\begin{example}
  Let us consider a recursive sequence $(u_n)_{n\geq 0}$ defined by $u_0=x_0$, $u_1=x_1$, $u_2=x_2$, and 
\begin{equation}\label{eq:QRT}
u_{n+1}=-\frac1{u_n}-\frac1{u_{n-1}}+\frac{\alpha}{u_{n}}\quad \text{for $n\geq 2$.}
\end{equation}
Here $\alpha\in\field^\ast$ is an arbitrary non-zero element. The map $(u_{n-2},u_{n-1},u_n)\mapsto (u_{n-1},u_n,u_{n+1})$ is known as the \emph{DTKQ-$2$ map} (cf.~\cite{DTKQ}) and is a special case of both $d-P_I$ and the \emph{additive $QRT$ map}~\cite[Eq.~(3)]{HoneQRT}. Let us now consider the substitution
\[u_n=\frac{\tau_n\tau_{n+3}}{\tau_{n+1}\tau_{n+2}}.\]
It imposes the following recurrence relation on the sequence $(\tau_n)_{n\geq 0}$:
\[\tau_n\tau_{n-3}^2\tau_{n-4}+\tau_{n-1}^2\tau_{n-4}^2+\tau_{n-1}\tau_{n-2}^2\tau_{n-5}=\alpha\tau_{n-2}^2\tau_{n-3}^2.\]
As it is shown in~\cite{HoneQRT}, the entries of the sequence $(\tau_n)$ are Laurent polynomials in the variables $\tau_0,\tau_1,\dots,\tau_4$. The degrees of these polynomials form a sequence $2,4,6,9,12,16,20,25,30,\dots$ which is easily shown to be equal to $\lfloor n^2/4 \rfloor$. This immediately implies that the discrete dynamical system~\eqref{eq:QRT} has zero algebraic entropy: the degree sequence grows quadratically rather than exponentially. Showing singularity confinement in this case is harder: one possible way is to prove that for each $n\geq 0$, $\tau_n$ is an \emph{irreducible} Laurent polynomial in the variables $\tau_0,\tau_1,\dots,\tau_4$ and that $\tau_n$ and $\tau_m$ are coprime for $n\neq m$.
\end{example}

This is a prototypical example of an integrable system of interest. See~\cite{HoneQRT} or~\cite{HoneSuper} for many other instances of applying the Laurentification technique to study integrability. Let us now compare this with our computation in Example~\ref{ex:somos_4}. Recall that we had a map
\[\rowm:(x_1:x_2:x_3) \mapsto \left(x_2:x_3:\frac{x_1x_2}{x_2+x_3}\right).\]
This suggests introducing the sequence $(u_n)_{n\geq 0}$ defined by $u_0=x_1$, $u_1=x_2$, $u_2=x_3$, and
\[u_{n+1}=\frac{u_{n-1}u_{n-2}}{u_{n-1}+u_{n}},\quad \text{for $n\geq 2$}.\]
Thus for all $t\geq 0$, the $R$-system associated with the digraph in Figure~\ref{fig:somos_4} satisfies $R(t)=(u_{t}:u_{t+1}:u_{t+2})$.
After making the substitution
\[u_n=\frac{\tau_n}{\tau_{n+1}},\]
we get that the sequence $(\tau_n)$ is defined by a recurrence relation
\[\tau_n\tau_{n-4}=\tau_{n-1}\tau_{n-3}+\tau_{n-2}^2,\quad \text{for $n\geq 4$,}\]
with initial data $\tau_0=x_1x_2x_3$, $\tau_1=x_2x_3$, $\tau_2=x_3$, $\tau_3=1$. This is the well-studied Somos-$4$ sequence~\cite{Gale} which is a certain reduction of the Hirota-Miwa equation~\cite{Miwa}. In particular, its entries are irreducible Laurent polynomials in $x_1,x_2,x_3$ that are pairwise coprime, see~\cite{FZCube} and~\cite[Theorem~3]{KMMT}. Their degrees, moreover, are well known to grow quadratically (cf. Proposition~\ref{prop:GR_quadratic}), and, as a corollary, we obtain the following result:

\begin{proposition}
The coefficient-free $R$-system associated with the digraph from Figure~\ref{fig:somos_4} has the singularity confinement property and its algebraic entropy is zero.
\end{proposition}

\def\Kbid{K^{\leftrightarrow}}

\def\rowm{\phi}
\def\bto{\dashrightarrow}

\section{Mirror symmetry, superpotential critical points, and GT patterns}\label{sec:mirr-symm-superp}
In this section, for any strongly connected digraph $G$, we define a certain Laurent polynomial $\sp$ called the \emph{superpotential} of $G$. We give two results on the number of fixed points of $\sp$ and then explain how our results are related to the superpotentials that appear in the mirror symmetry literature.

Throughout this section, we assume $G=(V,E,\wt)$ to be a strongly connected weighted digraph.

\begin{definition}
The \emph{superpotential} $\sp: \RP{V}\bto \field$ is defined by
\begin{equation}\label{eq:superp}
\sp(\X)=\sum_{(u,w)\in E} \wtt u w \frac{\X_w}{\X_u}.
\end{equation}

\end{definition}

One immediate observation is that $\sp$ is a \emph{conserved quantity} of the $R$-system associated with $G$:

\begin{proposition}
Let $\X,\X'\in\RP V$ be such that $\rowm(\X)=\X'$. Then we have
  \[\sp(\X)=\sp(\X').\]
\end{proposition}
\begin{proof}
  Using~\eqref{eq:toggle_sym}, we get
  \begin{equation*}\label{eq:}
\begin{split}
\sp(\X)=\sum_{v\in V}\sum_{(v,w)\in E} \wtt v w \frac{\X_w}{\X_v}=\sum_{v\in V}\sum_{(u,v)\in E} \wtt u v \frac{\X_v'}{\X_u'}=\sp(\X'),
\end{split}
\end{equation*}
which finishes the proof of the proposition.
\end{proof}
For the case of Example~\ref{ex:somos_4}, where $\X=(x_1:x_2:x_3)$ and $\X'=\left(x_2:x_3:\frac{x_1x_2}{x_2+x_3}\right)$, we have
\[\sp(\X)=\frac{x_2}{x_1}+\frac{x_3}{x_2}+\frac{x_3}{x_1}+\frac{x_1}{x_3},\]
while 
\[\sp(\X')=\frac{x_3}{x_2}+\frac{x_1x_2}{x_3(x_2+x_3)}+\frac{x_1x_2}{x_2(x_2+x_3)}+\frac{x_2(x_2+x_3)}{x_1x_2}.\]
The reader is encouraged to check that the two expressions coincide.

Another relation between the $R$-system and the superpotential is that the fixed points of the former are the critical points of the latter:

\begin{proposition}
We have $\rowm(\X)=\X$ if and only if the gradient of $\sp$ vanishes on $\X$. 
\end{proposition}
\begin{proof}
This fact is obvious from the definition of $\sp$.
\end{proof}

\def\ff{h}
\def\hh{g}
\def\RPP{\R \mathbb P} 
 
Our next result concerns positive critical points of the superpotential. Let us define $\RPP^V_{>0}$ to be the subset of $\mathbb{P}^V(\R)$ that consists of all points $\X=(\X_v)_{v\in V}$ such that $\X_v>0$ for all $v\in V$. Similarly, let $\R^V_{>0}$ be the subset of $\R^V$ consisting of points with all coordinates positive. When $\field=\R$, we denote by $\rowm_{>0}$ the restriction of $\rowm$ to $\RPP^V_{>0}$. Recall that by Remark~\ref{rmk:subtr_free}, $\rowm_{>0}$ is defined on the whole $\RPP^V_{>0}$ rather than on a Zariski open subset of it.

\def\ex{\exp}
\begin{proposition}\label{prop:fixpts_positive}
Suppose that $\field=\R$ and for all $(u,w)\in E$, the weight $\wtt u w$ is a positive real number. Then there exists a unique fixed point of $\rowm_{>0}:\RPP^{V}_{>0}\to\RPP^{V}_{>0}$ (also the unique positive critical point of $\sp$).
\end{proposition}
\begin{proof}
  Let $W=\R^V/\<\allones\>$ be a $(|V|-1)$-dimensional space and define a diffeomorphism $\ex:W\to \RPP^V_{>0}$ sending $\l=(\l_v)_{v\in V}\in W$ to $\X=(\exp(\l_v))_{v\in V}\in \RPP^V_{>0}$. In the $\l$-coordinates, $\sp$ has a positive definite Hessian by a direct computation, see~\cite[p.~137]{Rietsch1}, and thus the unique critical point of $\sp$ is the one where it attains its minimum. Such a point exists since $\sp$ takes positive values on $\RPP^V_{>0}$ and tends to $+\infty$ at the boundary of $\RPP^V_{>0}$. We are done with the proof.
\end{proof}

We now pass to the case $\field=\C$ and let the weights $\wtt u w$ be generic complex numbers. In this setting, the number of critical points of $\sp$ is given by Kouchnirenko's theorem~\cite{Kouchnirenko}:

\newcommand{\Newton}[1]{\operatorname{N}(#1)}
\begin{theorem}\label{thm:kouch}
  When the weights $\wtt u v$ are generic complex numbers, the number of  critical points of $\sp$ (all of which are non-degenerate) in $\mathbb{P}^V(\C)$ equals the \emph{normalized $(|V|-1)$-dimensional volume} of the \emph{Newton polytope} $\Newton{\sp}\subset \R^V$ of $\sp$ defined by
  \[\Newton{\sp}:=\Conv\left\{e_v-e_u\mid (u,v)\in E\right\}.\]
  For arbitrary weights, the number of critical points counted with multiplicities is at most the volume of $\Newton{\sp}$.
\end{theorem}
Here $\{e_v\}_{v\in V}$ is the standard linear basis of $\R^V$ and $\Conv$ denotes the convex hull. The polytope $\Newton{\sp}$ belongs to the hyperplane with zero sum of coordinates, and the normalized volume of $\Newton{\sp}$ equals $(|V|-1)!$ times the standard $(|V|-1)$-dimensional volume of $\Newton{\sp}$ in $\R^V$. Since $G$ is strongly connected, $\Newton{\sp}$ contains the origin. The polytope $\Newton{\sp}$ is closely related to the \emph{root polytope} of Postnikov~\cite{Postnikov_Perm} which is the convex hull of some subset of positive roots of the Type $A_n$ root system together with the origin. In fact, we can give a simple combinatorial interpretation of this volume by constructing a \emph{central triangulation} of $\Newton{\sp}$ as in~\cite[Section~13]{Postnikov_Perm}.

\def\tree{T}
\begin{definition}
Given a digraph $G=(V,E)$, an \emph{oriented spanning tree} of $G$ is a subset $T\subset E$ such that the undirected graph $(V,T)$ is a tree. In other words, $T$ is an oriented spanning tree if $|T|=|V|-1$ and the graph $(V,T)$ is connected.
\end{definition}

\begin{definition}
We say that an oriented spanning tree $T\subset E$ is \emph{admissible} if there is no path $u_1,u_2,\dots,u_k\in V$ such that $(u_i,u_{i+1})\in T$ for $1\leq i<k$ but $(u_1,u_k)\in E$. We say that two oriented spanning trees $T_1,T_2\subset E$ are \emph{compatible} if all the oriented cycles of the digraph $(V,T_1\cup T_2^\op)$ have length $2$. 
Here we denote $T_2^\op:=\{(v,u)\mid (u,v)\in T_2\}$.
\end{definition}

We obtain the following combinatorial interpretation for the number of critical points of $\sp$ when the weights are generic:
\begin{proposition}
The normalized volume of $\Newton{\sp}$ equals the maximal size of a collection of admissible and pairwise compatible oriented spanning trees of $G$.
\end{proposition}
\begin{proof}
This follows by adapting the proofs of~\cite[Lemmas~12.6 and~13.2]{Postnikov_Perm} to the case when $G$ is strongly connected in a straightforward way.
\end{proof}

For instance, the map $\rowm$ from Example~\ref{ex:somos_4} has $4$ fixed points in $\mathbb{P}^V(\C)$ when the weights are generic complex numbers. The normalized volume of $\Newton{\sp}$ in this case is also equal to $4$, see Figure~\ref{fig:somos_newton}. The digraph $G$ has five oriented spanning trees and all of them are admissible except for $\{(1,2),(2,3)\}$ since $(1,3)$ is contained in $E$. The other four trees are pairwise compatible and thus form a central triangulation of $\Newton{\sp}$ shown in Figure~\ref{fig:somos_newton}.

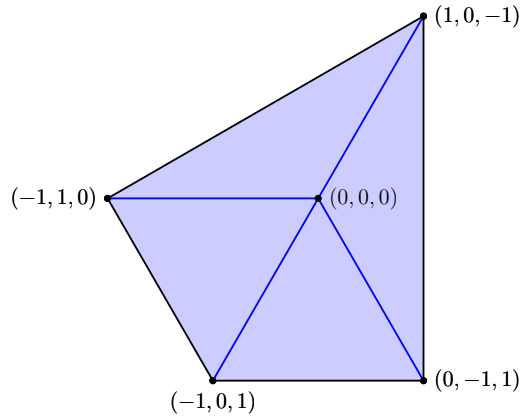
\begin{figure}
  \def\nodesc{0.3}
  \scalebox{0.7}{
\begin{tikzpicture}[scale=4]
\node[draw,circle,fill=black,scale=\nodesc,label=right:{$(0,0,0)$}] (O) at (0,0) { };
\node[draw,circle,fill=black,scale=\nodesc,label=right:{$(1,0,-1)$}] (A) at (60:1) { };
\node[draw,circle,fill=black,scale=\nodesc,label=left:{$(-1,1,0)$}] (B) at (180:1) { };
\node[draw,circle,fill=black,scale=\nodesc,label=below:{$(-1,0,1)$}] (C) at (-120:1) { };
\node[draw,circle,fill=black,scale=\nodesc,label=right:{$(0,-1,1)$}] (D) at (-60:1) { };
\draw[fill=blue!20,line width=1pt] (A.center)--(B.center)--(C.center)--(D.center)--(A.center);
\draw[line width=1pt,blue] (O)--(A);
\draw[line width=1pt,blue] (O)--(B);
\draw[line width=1pt,blue] (O)--(C);
\draw[line width=1pt,blue] (O)--(D);

\node[draw,circle,fill=black,scale=\nodesc,label=right:{$(0,0,0)$}] (O) at (0,0) { };
\node[draw,circle,fill=black,scale=\nodesc,label=right:{$(1,0,-1)$}] (A) at (60:1) { };
\node[draw,circle,fill=black,scale=\nodesc,label=left:{$(-1,1,0)$}] (B) at (180:1) { };
\node[draw,circle,fill=black,scale=\nodesc,label=below:{$(-1,0,1)$}] (C) at (-120:1) { };
\node[draw,circle,fill=black,scale=\nodesc,label=right:{$(0,-1,1)$}] (D) at (-60:1) { };

\end{tikzpicture}}
  \caption{\label{fig:somos_newton} The Newton polytope of the superpotential $\sp$ for the digraph $G$ from Figure~\ref{fig:somos_4}. Its central triangulation into $4$ simplices is shown in blue, thus its normalized volume equals $4$.}
\end{figure}

\subsection{Frozen variables and the connection with mirror symmetry}
Givental~\cite{Givental} gave a mirror construction for the full flag variety of Type $A$ which was later generalized in~\cite{BCKS} to the case of general partial flag varieties of Type $A$. It turns out that the phase function on the mirror model (also called the \emph{superpotential} in the mirror symmetry literature) is a Laurent polynomial which coincides with $\sp$ for a certain strongly connected weighted digraph $G$. Consequently, the fixed points of the $R$-system associated with $G$ get an interpretation in this mirror symmetric context: the number of complex critical points of $\sp$ equals the dimension of the quantum cohomology ring of the corresponding partial flag variety, see~\cite[p.~11]{Givental} or~\cite[Corollary~4.2]{Rietsch2}. In particular, the superpotential for the full flag variety in Type $A_n$ has $(n+1)!$ complex critical points. The existence of a totally positive critical point was shown in~\cite{Rietsch1} (Type $A$) and~\cite{Rietsch3} (general type). Its uniqueness was shown in~\cite{Rietsch1} (Type $A$) and~\cite{LR} (general type). 

\def\nucal{\mathcal{V}}
\def\acal{\mathcal{A}}

Let us briefly explain which digraphs generate superpotentials for Type $A$ partial flag varieties. We follow the exposition in~\cite[\S5]{Rietsch1}. Recall that a partial flag variety in $\C^{n+1}$ is associated to any strictly increasing sequence $0=n_0<n_1<\dots<n_k<n_{k+1}=n+1$ of integers. Define a digraph $G_{n_1,\dots,n_k}=(\nucal,\acal)$ as follows. The vertex set $\nucal$ is a disjoint union of $\nucal_\bullet$ and $\nucal_\star$. We refer to the elements of the latter as \emph{frozen vertices}. Explicitly, these sets are given by
\begin{equation}\label{eq:nucal}
\begin{split}
\nucal_\bullet&=\bigcup_{j=1}^k \{(m,r)\in\Z^2_{\geq0}\mid n_j\leq m<n_{j+1}, 1\leq r\leq n_j\};\\
\nucal_\star&= \{\star_j:=(n_j-1,n_{j-1}+1)\mid j=1,\dots,k+1\};
\end{split}
\end{equation}
The edge set $\acal$ consists of all pairs $((m+1,r),(m,r))$ and $((m,r+1),(m,r))$ whenever both vertices belong to $\nucal$. An example for $k=3$ and $(n_0,n_1,n_2,n_3,n_4)=(0,2,5,6,8)$ is given in Figure~\ref{fig:partial_flag} (left).

\def\spm{\mathcal{F}_{n_1,\dots,n_k}}
\def\qt{{\tilde q}}

\begin{figure}

\begin{tabular}{cc}
\scalebox{1.4}{
\begin{tikzpicture}[scale=0.6]
\node[scale=0.7] (N2x1) at (1.00,-2.00) {$\bullet$};
\node[scale=0.7] (N2x2) at (2.00,-2.00) {$\bullet$};
\node[scale=0.7] (N3x1) at (1.00,-3.00) {$\bullet$};
\node[scale=0.7] (N3x2) at (2.00,-3.00) {$\bullet$};
\node[scale=0.7] (N4x1) at (1.00,-4.00) {$\bullet$};
\node[scale=0.7] (N4x2) at (2.00,-4.00) {$\bullet$};
\node[scale=0.7] (N5x1) at (1.00,-5.00) {$\bullet$};
\node[scale=0.7] (N5x2) at (2.00,-5.00) {$\bullet$};
\node[scale=0.7] (N5x3) at (3.00,-5.00) {$\bullet$};
\node[scale=0.7] (N5x4) at (4.00,-5.00) {$\bullet$};
\node[scale=0.7] (N5x5) at (5.00,-5.00) {$\bullet$};
\node[scale=0.7] (N6x1) at (1.00,-6.00) {$\bullet$};
\node[scale=0.7] (N6x2) at (2.00,-6.00) {$\bullet$};
\node[scale=0.7] (N6x3) at (3.00,-6.00) {$\bullet$};
\node[scale=0.7] (N6x4) at (4.00,-6.00) {$\bullet$};
\node[scale=0.7] (N6x5) at (5.00,-6.00) {$\bullet$};
\node[scale=0.7] (N6x6) at (6.00,-6.00) {$\bullet$};
\node[scale=0.7] (N7x1) at (1.00,-7.00) {$\bullet$};
\node[scale=0.7] (N7x2) at (2.00,-7.00) {$\bullet$};
\node[scale=0.7] (N7x3) at (3.00,-7.00) {$\bullet$};
\node[scale=0.7] (N7x4) at (4.00,-7.00) {$\bullet$};
\node[scale=0.7] (N7x5) at (5.00,-7.00) {$\bullet$};
\node[scale=0.7] (N7x6) at (6.00,-7.00) {$\bullet$};
\node[scale=0.7] (N1x1) at (1.00,-1.00) {$\bigstar$};
\node[anchor=south west,inner sep=5pt,scale=0.7] (QN1x1) at (N1x1.center) {$\qt_1$};
\node[scale=0.7] (N4x3) at (3.00,-4.00) {$\bigstar$};
\node[anchor=south west,inner sep=5pt,scale=0.7] (QN4x3) at (N4x3.center) {$\qt_2$};
\node[scale=0.7] (N5x6) at (6.00,-5.00) {$\bigstar$};
\node[anchor=south west,inner sep=5pt,scale=0.7] (QN5x6) at (N5x6.center) {$\qt_3$};
\node[scale=0.7] (N7x7) at (7.00,-7.00) {$\bigstar$};
\node[anchor=south west,inner sep=5pt,scale=0.7] (QN7x7) at (N7x7.center) {$\qt_4$};
\draw[postaction={decorate}, line width=0.4] (1.00,-2.00) -- (1.00,-1.00);
\draw[postaction={decorate}, line width=0.4] (2.00,-2.00) -- (1.00,-2.00);
\draw[postaction={decorate}, line width=0.4] (1.00,-3.00) -- (1.00,-2.00);
\draw[postaction={decorate}, line width=0.4] (2.00,-3.00) -- (1.00,-3.00);
\draw[postaction={decorate}, line width=0.4] (2.00,-3.00) -- (2.00,-2.00);
\draw[postaction={decorate}, line width=0.4] (1.00,-4.00) -- (1.00,-3.00);
\draw[postaction={decorate}, line width=0.4] (2.00,-4.00) -- (1.00,-4.00);
\draw[postaction={decorate}, line width=0.4] (2.00,-4.00) -- (2.00,-3.00);
\draw[postaction={decorate}, line width=0.4] (3.00,-4.00) -- (2.00,-4.00);
\draw[postaction={decorate}, line width=0.4] (1.00,-5.00) -- (1.00,-4.00);
\draw[postaction={decorate}, line width=0.4] (2.00,-5.00) -- (1.00,-5.00);
\draw[postaction={decorate}, line width=0.4] (2.00,-5.00) -- (2.00,-4.00);
\draw[postaction={decorate}, line width=0.4] (3.00,-5.00) -- (2.00,-5.00);
\draw[postaction={decorate}, line width=0.4] (3.00,-5.00) -- (3.00,-4.00);
\draw[postaction={decorate}, line width=0.4] (4.00,-5.00) -- (3.00,-5.00);
\draw[postaction={decorate}, line width=0.4] (5.00,-5.00) -- (4.00,-5.00);
\draw[postaction={decorate}, line width=0.4] (6.00,-5.00) -- (5.00,-5.00);
\draw[postaction={decorate}, line width=0.4] (1.00,-6.00) -- (1.00,-5.00);
\draw[postaction={decorate}, line width=0.4] (2.00,-6.00) -- (1.00,-6.00);
\draw[postaction={decorate}, line width=0.4] (2.00,-6.00) -- (2.00,-5.00);
\draw[postaction={decorate}, line width=0.4] (3.00,-6.00) -- (2.00,-6.00);
\draw[postaction={decorate}, line width=0.4] (3.00,-6.00) -- (3.00,-5.00);
\draw[postaction={decorate}, line width=0.4] (4.00,-6.00) -- (3.00,-6.00);
\draw[postaction={decorate}, line width=0.4] (4.00,-6.00) -- (4.00,-5.00);
\draw[postaction={decorate}, line width=0.4] (5.00,-6.00) -- (4.00,-6.00);
\draw[postaction={decorate}, line width=0.4] (5.00,-6.00) -- (5.00,-5.00);
\draw[postaction={decorate}, line width=0.4] (6.00,-6.00) -- (5.00,-6.00);
\draw[postaction={decorate}, line width=0.4] (6.00,-6.00) -- (6.00,-5.00);
\draw[postaction={decorate}, line width=0.4] (1.00,-7.00) -- (1.00,-6.00);
\draw[postaction={decorate}, line width=0.4] (2.00,-7.00) -- (1.00,-7.00);
\draw[postaction={decorate}, line width=0.4] (2.00,-7.00) -- (2.00,-6.00);
\draw[postaction={decorate}, line width=0.4] (3.00,-7.00) -- (2.00,-7.00);
\draw[postaction={decorate}, line width=0.4] (3.00,-7.00) -- (3.00,-6.00);
\draw[postaction={decorate}, line width=0.4] (4.00,-7.00) -- (3.00,-7.00);
\draw[postaction={decorate}, line width=0.4] (4.00,-7.00) -- (4.00,-6.00);
\draw[postaction={decorate}, line width=0.4] (5.00,-7.00) -- (4.00,-7.00);
\draw[postaction={decorate}, line width=0.4] (5.00,-7.00) -- (5.00,-6.00);
\draw[postaction={decorate}, line width=0.4] (6.00,-7.00) -- (5.00,-7.00);
\draw[postaction={decorate}, line width=0.4] (6.00,-7.00) -- (6.00,-6.00);
\draw[postaction={decorate}, line width=0.4] (7.00,-7.00) -- (6.00,-7.00);
\end{tikzpicture}}
&
\scalebox{1.4}{
\begin{tikzpicture}[scale=0.6]
\node[scale=0.7] (N2x1) at (1.00,-2.00) {$\bullet$};
\node[scale=0.7] (N2x2) at (2.00,-2.00) {$\bullet$};
\node[scale=0.7] (N3x1) at (1.00,-3.00) {$\bullet$};
\node[scale=0.7] (N3x2) at (2.00,-3.00) {$\bullet$};
\node[scale=0.7] (N4x1) at (1.00,-4.00) {$\bullet$};
\node[scale=0.7] (N4x2) at (2.00,-4.00) {$\bullet$};
\node[scale=0.7] (N5x1) at (1.00,-5.00) {$\bullet$};
\node[scale=0.7] (N5x2) at (2.00,-5.00) {$\bullet$};
\node[scale=0.7] (N5x3) at (3.00,-5.00) {$\bullet$};
\node[scale=0.7] (N5x4) at (4.00,-5.00) {$\bullet$};
\node[scale=0.7] (N5x5) at (5.00,-5.00) {$\bullet$};
\node[scale=0.7] (N6x1) at (1.00,-6.00) {$\bullet$};
\node[scale=0.7] (N6x2) at (2.00,-6.00) {$\bullet$};
\node[scale=0.7] (N6x3) at (3.00,-6.00) {$\bullet$};
\node[scale=0.7] (N6x4) at (4.00,-6.00) {$\bullet$};
\node[scale=0.7] (N6x5) at (5.00,-6.00) {$\bullet$};
\node[scale=0.7] (N6x6) at (6.00,-6.00) {$\bullet$};
\node[scale=0.7] (N7x1) at (1.00,-7.00) {$\bullet$};
\node[scale=0.7] (N7x2) at (2.00,-7.00) {$\bullet$};
\node[scale=0.7] (N7x3) at (3.00,-7.00) {$\bullet$};
\node[scale=0.7] (N7x4) at (4.00,-7.00) {$\bullet$};
\node[scale=0.7] (N7x5) at (5.00,-7.00) {$\bullet$};
\node[scale=0.7] (N7x6) at (6.00,-7.00) {$\bullet$};
\node[scale=0.7] (NST) at (7.00,-1.00) {$\bigstar$};
\draw[postaction={decorate}, line width=0.4] (2.00,-2.00) -- (1.00,-2.00);
\draw[postaction={decorate}, line width=0.4] (1.00,-3.00) -- (1.00,-2.00);
\draw[postaction={decorate}, line width=0.4] (2.00,-3.00) -- (1.00,-3.00);
\draw[postaction={decorate}, line width=0.4] (2.00,-3.00) -- (2.00,-2.00);
\draw[postaction={decorate}, line width=0.4] (1.00,-4.00) -- (1.00,-3.00);
\draw[postaction={decorate}, line width=0.4] (2.00,-4.00) -- (1.00,-4.00);
\draw[postaction={decorate}, line width=0.4] (2.00,-4.00) -- (2.00,-3.00);
\draw[postaction={decorate}, line width=0.4] (1.00,-5.00) -- (1.00,-4.00);
\draw[postaction={decorate}, line width=0.4] (2.00,-5.00) -- (1.00,-5.00);
\draw[postaction={decorate}, line width=0.4] (2.00,-5.00) -- (2.00,-4.00);
\draw[postaction={decorate}, line width=0.4] (3.00,-5.00) -- (2.00,-5.00);
\draw[postaction={decorate}, line width=0.4] (4.00,-5.00) -- (3.00,-5.00);
\draw[postaction={decorate}, line width=0.4] (5.00,-5.00) -- (4.00,-5.00);
\draw[postaction={decorate}, line width=0.4] (1.00,-6.00) -- (1.00,-5.00);
\draw[postaction={decorate}, line width=0.4] (2.00,-6.00) -- (1.00,-6.00);
\draw[postaction={decorate}, line width=0.4] (2.00,-6.00) -- (2.00,-5.00);
\draw[postaction={decorate}, line width=0.4] (3.00,-6.00) -- (2.00,-6.00);
\draw[postaction={decorate}, line width=0.4] (3.00,-6.00) -- (3.00,-5.00);
\draw[postaction={decorate}, line width=0.4] (4.00,-6.00) -- (3.00,-6.00);
\draw[postaction={decorate}, line width=0.4] (4.00,-6.00) -- (4.00,-5.00);
\draw[postaction={decorate}, line width=0.4] (5.00,-6.00) -- (4.00,-6.00);
\draw[postaction={decorate}, line width=0.4] (5.00,-6.00) -- (5.00,-5.00);
\draw[postaction={decorate}, line width=0.4] (6.00,-6.00) -- (5.00,-6.00);
\draw[postaction={decorate}, line width=0.4] (1.00,-7.00) -- (1.00,-6.00);
\draw[postaction={decorate}, line width=0.4] (2.00,-7.00) -- (1.00,-7.00);
\draw[postaction={decorate}, line width=0.4] (2.00,-7.00) -- (2.00,-6.00);
\draw[postaction={decorate}, line width=0.4] (3.00,-7.00) -- (2.00,-7.00);
\draw[postaction={decorate}, line width=0.4] (3.00,-7.00) -- (3.00,-6.00);
\draw[postaction={decorate}, line width=0.4] (4.00,-7.00) -- (3.00,-7.00);
\draw[postaction={decorate}, line width=0.4] (4.00,-7.00) -- (4.00,-6.00);
\draw[postaction={decorate}, line width=0.4] (5.00,-7.00) -- (4.00,-7.00);
\draw[postaction={decorate}, line width=0.4] (5.00,-7.00) -- (5.00,-6.00);
\draw[postaction={decorate}, line width=0.4] (6.00,-7.00) -- (5.00,-7.00);
\draw[postaction={decorate}, line width=0.4] (6.00,-7.00) -- (6.00,-6.00);
\draw[postaction={decorate}, line width=0.4] (1,-2) to[bend right=-25] node[midway,below,scale=0.5, inner sep=7pt] {$\qt_1$} (7,-1);
\draw[postaction={decorate}, line width=0.4] (7,-1) to[bend right=0] node[midway,above,scale=0.5, inner sep=10pt] {$1/\qt_2$} (2,-4);
\draw[postaction={decorate}, line width=0.4] (3,-5) to[bend right=0] node[midway,right,scale=0.5, inner sep=5pt] {$\qt_2$} (7,-1);
\draw[postaction={decorate}, line width=0.4] (7,-1) to[bend right=-25] node[pos=0.65,left,scale=0.5, inner sep=5pt] {$1/\qt_3$} (5,-5);
\draw[postaction={decorate}, line width=0.4] (6,-6) to[bend right=25] node[midway,right,scale=0.5, inner sep=5pt] {$\qt_3$} (7,-1);
\draw[postaction={decorate}, line width=0.4] (7,-1) to[bend right=-50] node[midway,right,scale=0.5, inner sep=5pt] {$1/\qt_4$} (6,-7);
\end{tikzpicture}}\\

\end{tabular}
  \caption{\label{fig:partial_flag}The digraph $G_{2,5,6}$ (left) and the corresponding strongly connected weighted digraph $G$ (right). The unlabeled edges of $G$ have weight $1$. The coordinates are in matrix notation. The left figure is reproduced from~\cite[Figure~1]{Rietsch1}.}
\end{figure}
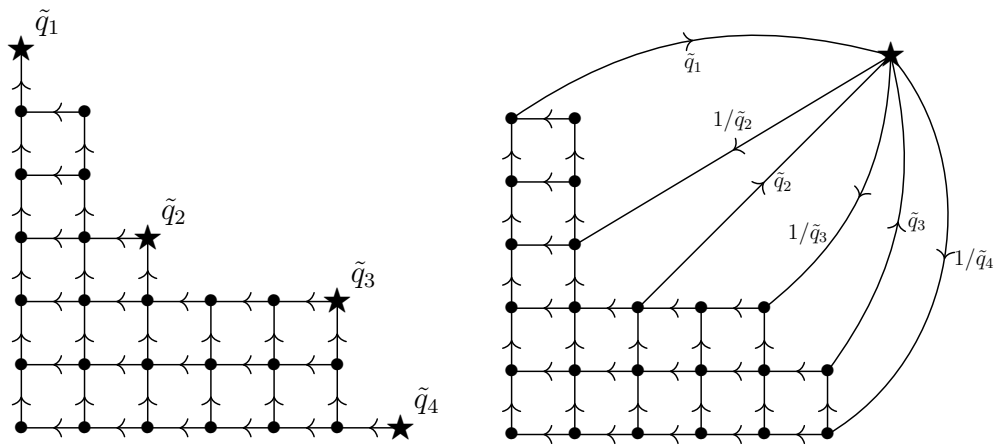

\def\Xb{{\tilde X}}
Additionally, there are $k$ fixed \emph{frozen parameters}\footnote{It is more convenient for us to consider $k+1$ parameters $(\qt_1,\dots,\qt_k,\qt_{k+1})$ representing a point in the projective space, rescaled so that $\qt_{k+1}=1$.} $(\qt_1,\dots,\qt_k)\in\field^k$ and the \emph{superpotential} is a Laurent polynomial $\spm:\field^{\nucal_\bullet}\to\field$ which is, loosely speaking, given by~\eqref{eq:superp}. More precisely, for $\X\in\field^{\nucal_\bullet}$, extend $\X$ to a point $\Xb\in \RP \nucal$ by setting $\Xb_{\star_j}:=\qt_j$ for $1\leq j\leq k$ and $\Xb_{\star_{k+1}}:=1$, and then we define
\[\spm(\X):=\sum_{(u,w)\in \acal} \frac{\Xb_w}{\Xb_u}.\]
To reduce this formula to a special case of~\eqref{eq:superp}, let us create a strongly connected weighted digraph $G=(V,E,\wt)$ from $G_{n_1,\dots,n_k}$. We identify the vertices $\star_1,\dots,\star_{k+1}$ of $G_{n_1,\dots,n_k}$ into one new vertex $\star$ of $G$, and introduce the weights $\wt:E\to \field$ as follows. For every edge $(\star_j, v)\in\acal$ of $G_{n_1,\dots,n_k}$, set $\wtt {\star_j} v:=1/\qt_j$. For every edge $(v,\star_j)\in\acal$ of $G_{n_1,\dots,n_k}$, set $\wtt v{\star_j}:=\qt_j$. Finally, for all the remaining edges of $(u,v)\in\acal$, set $\wtt u v:=1$. See Figure~\ref{fig:partial_flag} (right). It is clear that $\sp:\RP {\nucal_\bullet\cup\{\star\}}\bto\field$ now coincides with the projectivization of $\spm:\field^{\nucal_\bullet}\bto\field$ after setting $\spm(\star):=1$. Note also that one can consider \emph{$R$-systems with frozen variables} in a similar way, and the above reduction procedure shows that this is the same generality as just $R$-systems with coefficients introduced in Definition~\ref{dfn:R}.

\begin{remark}
We note that the equality in Theorem~\ref{thm:kouch} does not directly apply to the superpotentials coming from mirror symmetry since their coefficients are not ``generic enough''. This can already be seen for the full flag variety in $\C^3$. Here we have $k=2$ and $(n_0,n_1,n_2,n_3)=(0,1,2,3)$, so we get the Gelfand-Tsetlin triangle digraph $G_{1,2}$ shown in Figure~\ref{fig:full_flag} (left). It corresponds to a strongly connected weighted digraph $G$ in Figure~\ref{fig:full_flag} (middle). The number of critical points of $\spm$ and therefore of $\sp$ equals to $(n+1)!=6$, however, substituting generic coefficients into $\sp$ yields a Laurent polynomial with $8$ critical points. Equivalently, $\Newton{\sp}$ has volume $8$ as we can see from Figure~\ref{fig:full_flag} (right). Thus Theorem~\ref{thm:kouch} only yields an upper bound that is not optimal already in this simple case.
\end{remark}

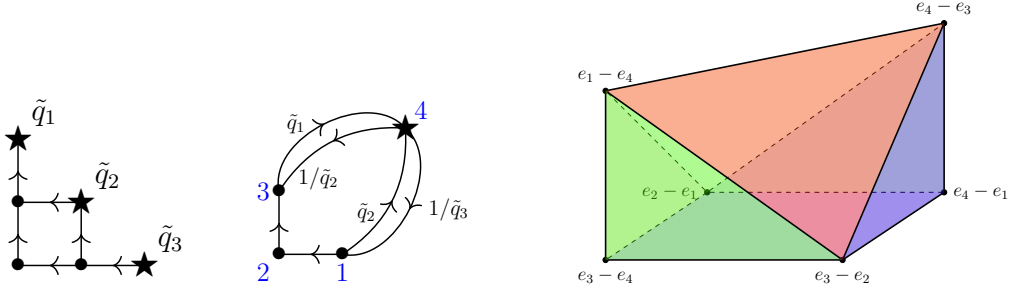
\begin{figure}

\begin{tabular}{ccc}
\scalebox{1.4}{
\begin{tikzpicture}[scale=0.6]
\node[scale=0.7] (N1x1) at (1.00,-1.00) {$\bullet$};
\node[scale=0.7] (N2x1) at (1.00,-2.00) {$\bullet$};
\node[scale=0.7] (N2x2) at (2.00,-2.00) {$\bullet$};
\node[scale=0.7] (N0x1) at (1.00,0.00) {$\bigstar$};
\node[anchor=south west,inner sep=5pt,scale=0.7] (QN0x1) at (N0x1.center) {$\qt_1$};
\node[scale=0.7] (N1x2) at (2.00,-1.00) {$\bigstar$};
\node[anchor=south west,inner sep=5pt,scale=0.7] (QN1x2) at (N1x2.center) {$\qt_2$};
\node[scale=0.7] (N2x3) at (3.00,-2.00) {$\bigstar$};
\node[anchor=south west,inner sep=5pt,scale=0.7] (QN2x3) at (N2x3.center) {$\qt_3$};
\draw[postaction={decorate}, line width=0.4] (1.00,-1.00) -- (1.00,0.00);
\draw[postaction={decorate}, line width=0.4] (2.00,-1.00) -- (1.00,-1.00);
\draw[postaction={decorate}, line width=0.4] (1.00,-2.00) -- (1.00,-1.00);
\draw[postaction={decorate}, line width=0.4] (2.00,-2.00) -- (1.00,-2.00);
\draw[postaction={decorate}, line width=0.4] (2.00,-2.00) -- (2.00,-1.00);
\draw[postaction={decorate}, line width=0.4] (3.00,-2.00) -- (2.00,-2.00);
\end{tikzpicture}}
  &
    \def\textscl{0.6}
\scalebox{1.4}{
\begin{tikzpicture}[scale=0.6]
\node[scale=0.7] (N1x1) at (1.00,-1.00) {$\bullet$};
\node[scale=0.7] (N2x1) at (1.00,-2.00) {$\bullet$};
\node[scale=0.7] (N2x2) at (2.00,-2.00) {$\bullet$};
\node[scale=0.7] (NST) at (3.00,0.00) {$\bigstar$};
\node[text=blue,anchor=north,scale=\textscl] (L22) at (2,-2) {$1$}; 
\node[text=blue,anchor=north east,scale=\textscl] (L12) at (1,-2) {$2$}; 
\node[text=blue,anchor=east,scale=\textscl] (L11) at (1,-1) {$3$}; 
\node[text=blue,anchor=south west,scale=\textscl] (L30) at (3,-0) {$4$}; 
\draw[postaction={decorate}, line width=0.4] (1.00,-2.00) -- (1.00,-1.00);
\draw[postaction={decorate}, line width=0.4] (2.00,-2.00) -- (1.00,-2.00);
\draw[postaction={decorate}, line width=0.4] (1,-1) to[bend right=-75] node[midway,left,scale=0.5, inner sep=10pt] {$\qt_1$} (3,-0);
\draw[postaction={decorate}, line width=0.4] (3,-0) to[bend right=30] node[pos=0.9,right,scale=0.5, inner sep=5pt] {$1/\qt_2$} (1,-1);
\draw[postaction={decorate}, line width=0.4] (2,-2) to[bend right=30] node[pos=0.4,left,scale=0.5, inner sep=5pt] {$\qt_2$} (3,-0);
\draw[postaction={decorate}, line width=0.4] (3,-0) to[bend right=-75] node[midway,right,scale=0.5, inner sep=10pt] {$1/\qt_3$} (2,-2);
\end{tikzpicture}}&

\def\opp{0.2}
                    \def\oppp{0.3}
                    \def\linew{0.6pt}
                    \def\lineww{1pt}
  \def\nodesc{0.3}
  \scalebox{0.6}{
\begin{tikzpicture}[scale=1.5]
\node[draw,circle,fill=black,scale=\nodesc,label=right:{$e_4-e_1$}] (A14) at (2.5,0.5) { };
\node[draw,circle,fill=black,scale=\nodesc,label=above:{$e_4-e_3$}] (A34) at (2.5,3) { };
\node[draw,circle,fill=black,scale=\nodesc,label=above:{$e_1-e_4$}] (A41) at (-2.5,2) { };
\node[draw,circle,fill=black,scale=\nodesc,label=left:{$e_2-e_1$}] (A12) at (-1,0.5) { };
\node[draw,circle,fill=black,scale=\nodesc,label=below:{$e_3-e_4$}] (A43) at (-2.5,-0.5) { };
\node[draw,circle,fill=black,scale=\nodesc,label=below:{$e_3-e_2$}] (A23) at (1,-0.5) { };
\draw[line width=\linew,dashed] (A43)--(A12)--(A14);
\draw[line width=\linew,dashed] (A41)--(A12)--(A34);
\fill[blue!50!red,opacity=\opp] (A12.center)--(A14.center)--(A23.center)--(A43.center)--(A12.center);
\fill[red!50!green,opacity=\opp] (A12.center)--(A34.center)--(A14.center)--(A12.center);
\fill[blue!10!yellow,opacity=\opp] (A12.center)--(A41.center)--(A34.center)--(A12.center);
\fill[red!30!yellow,opacity=\opp] (A12.center)--(A41.center)--(A43.center)--(A12.center);
\fill[blue,opacity=\oppp] (A14.center)--(A34.center)--(A23.center)--(A14.center);
\fill[red,opacity=\oppp] (A23.center)--(A34.center)--(A41.center)--(A23.center);
\fill[green,opacity=\oppp] (A23.center)--(A41.center)--(A43.center)--(A23.center);
\draw[line width=\lineww] (A43)--(A41)--(A34)--(A14)--(A23)--(A43);
\draw[line width=\lineww] (A41)--(A23)--(A34);

\end{tikzpicture}}

\end{tabular}

\caption{\label{fig:full_flag} The digraph $G_{1,2}$ (left), the corresponding strongly connected weighted digraph $G$ (middle), and the Newton polytope $\Newton{\sp}\subset\R^3$ (right). It has normalized volume $8$ (the number of simplices in any triangulation of its boundary). The blue labels in the middle picture indicate the indexing of the vertices of $G$ so that an edge $(i,j)$ of $G$ yields a vertex $e_j-e_i$ of $\Newton{\sp}$.}
\end{figure}

\subsection*{Acknowledgments}
The material in this section is largely based on conversations with Thomas Lam. We are grateful to him for introducing us to this beautiful subject. We also thank Alex Postnikov for suggesting using Kouchnirenko's theorem to compute the number of critical points of the superpotential.

\section{Proof of Theorem~\ref{thm:arb}}\label{sec:proof-theor-arb}
Let us rewrite~\eqref{eq:toggle} even more symmetrically:
\begin{equation}\label{eq:toggle_very_sym}
  \sum_{\edge u v\in E}\wtt u v\frac{\X_v}{\X'_u}
  =\sum_{\edge v w\in E}\wtt v w \frac{\X_w}{\X'_v}
,\quad \text{for all $v\in V$}.
\end{equation}
We would like to represent it as a linear system of equations in the variables $T:=\left(\frac{\X_u}{\X'_u}\right)_{u\in V}$. To do so, we rewrite it one more time as follows:
\begin{equation}\label{eq:toggle_very_very_sym}
  \sum_{\edge u v\in E}\wtt u v\frac{\X_v}{\X_u}\frac{\X_u}{\X'_u}
  =\sum_{\edge v w\in E}\wtt v w \frac{\X_w}{\X_v}\frac{\X_v}{\X'_v}
,\quad \text{for all $v\in V$}.
\end{equation}

The matrix $A=(a_{vu})$ of this system is given by
\[a_{vu}=
  \begin{cases}
    \sum_{\edge v w\in E}\wtt v w  \frac{\X_w}{\X_v} , &\text{if $u=v$;}\\
    -\wtt u v\frac{\X_v}{\X_u}, &\text{if $(u,v)\in E$;}\\
    0,&\text{otherwise.}
 \end{cases}
\]
Clearly $A$ is  a weighted Laplacian matrix of $G$ with edge weights $\wtt u v \frac{\X_v}{\X_u}$, so its \emph{cokernel} (i.e., the unique up to a scalar  solution $T$ to the linear system $AT=0$) is given by the Matrix-Tree theorem, see, e.g.,~\cite{Chaiken}. We obtain the formula~\eqref{eq:arb}.\qed

\begin{remark}
Solving the system~\eqref{eq:toggle_very_sym} in the variables $(X_v)_{v\in V}$ yields a formula analogous to~\eqref{eq:arb} for $\rowm^{-1}$. Thus $\rowm$ is a birational map.
\end{remark}

\def\trop{\operatorname{trop}}
\def\Xtr{\X^\oplus}
\def\rowmt{\phi_\oplus}
\def\prj{\proj}
\def\lt{{\tilde{\l}}}
\section{Tropical dynamics}\label{sec:tropical-dynamics}
Since~\eqref{eq:arb} is a subtraction-free expression, one can apply \emph{tropicalization} to it, i.e., replace multiplication by addition, division by subtraction, and addition by taking the maximum. To give a formal definition, we fix some strongly connected digraph $G=(V,E)$ and consider the $(|V|-1)$-dimensional space $W:=\R^V/\<\allones\>$. Let $\prj:\R^V\to W$ be the natural projection map. For $\l\in W$, we denote by $\lt\in\R^V$ its \emph{canonical form}, that is, the unique vector such that $\prj(\lt)=\l$ and all coordinates of $\lt$ are nonnegative, and at least one of them is zero.
\begin{definition}
  Define the map $\rowmt:W\to W$ as follows. Let $\l\in W$ be a vector and choose any representative $\l'=(\l'_v)_{v\in V}\in\R^V$ of $\l$ (i.e., $\prj(\l')=\l$). Then we set $\rowmt(\l):=\prj(\m')$, where the vector $\m'=(\m'_v)_{v\in V}\in\R^V$ is defined by
  \[\m'_v:=\l'_v-\max_{\arb\in\Arb(G,v)}\left(\sum_{u\in V\setminus\{v\}}(\l'_{\arb(u)}-\l'_u)\right).\]
\end{definition}
Clearly $\rowmt(\l)$ does not depend on the choice of $\l'$.

\begin{definition}
  Let $\l\in W$ be a vector. The \emph{tropical $R$-system associated with $G$ with initial conditions $\l$} is a family $(\Xtr(t))_{t\geq0}$ of elements of $W$ defined by $\Xtr(t)=\rowmt^t(\l)$.
\end{definition}

It is a well known fact that the tropical $R$-system describes the degrees (more generally, the \emph{Newton polytope}) of the values $\X_v(t)$ of the $R$-system associated with $G$, see, e.g.,~\cite[Section~6.1]{GP1}. Thus, for example, one can deduce the information about the algebraic entropy of the (coefficient-free) $R$-system by studying the tropical $R$-system associated with $G$. This relationship allows one  to tropicalize other subtraction-free formulas that are true for $R$-systems, for example:

\begin{corollary}
  The values $\Xtr(t)$ of the tropical $R$-system satisfy the \emph{tropical toggle relation:}
  \begin{equation}\label{eq:toggle_trop}
    \Xtr_v(t)+\Xtr_v(t+1)=\max_{(v,w)\in E}\Xtr_w(t) +\min_{(u,v)\in E}\Xtr_u(t+1).
\end{equation}
\end{corollary}
This tropical toggle relation is the main building block of the Robinson-Schensted-Knuth correspondence for semistandard Young tableaux, see~\cite[Definition~0.1]{KB}.

\begin{remark}\label{rmk:striker_williams}
The \emph{rowmotion} on order ideals of posets consists of basically applying~\eqref{eq:toggle_trop} to the vertices of the poset $P$, under the assumption that $\Xtr_{\hat0}=1$, $\Xtr_{\hat1}=0$, and the rest of the values are equal to either $0$ or $1$ so that we have $\Xtr_v\leq \Xtr_u$ whenever $u\leq v$ in $\hat P$. Recall that to go from posets to strongly connected digraphs, one needs to identify $\hat0$ with $\hat1$, so it is not clear to us how to extend the rowmotion action to our setting.
\end{remark}

\begin{example}
  In the case of the digraph in Figure~\ref{fig:somos_4}, the map $\rowmt$ sends $(x_1,x_{2},x_{3})\in W$ to
  \begin{equation}\label{eq:tropical_somos_4}
(x_1',x_2',x_3'):=(x_{2},x_{3},x_1+x_{2}-\max(x_{2},x_{3})).
  \end{equation}
  
  The tropical toggle relations in this case are
  \begin{equation*}
    \begin{split}
      x_1+x_1'&=\max(x_2,x_3)+x_3';\\
      x_2+x_2'&=x_3+x_1';\\
      x_3+x_3'&=x_1+\min(x_1',x_2').
    \end{split}
  \end{equation*}
  Clearly they are satisfied when $(x_1',x_2',x_3')$ are given by~\eqref{eq:tropical_somos_4}.
\end{example}

Even though tropical dynamics on Gelfand-Tsetlin patterns is important due to its relation to RSK, we leave studying the tropical $R$-system beyond the scope of this paper.

\section{Defining  \texorpdfstring{$\tau$}{tau}-sequences}\label{sec:defin-tau-sequ}
Before we proceed to the next part, where many examples of $R$-systems are studied in detail, we would like to state a precise definition of a $\tau$-sequence. We will also see that if an $R$-system admits a  $\tau$-sequence then it automatically has the singularity confinement property. Thus the existence of a $\tau$-sequence is a rather strong property, and for the vast majority of the examples that we consider, we will either prove or conjecture that the $R$-system admits an explicit $\tau$-sequence.

\def\numberfield{F}
We let $\numberfield$ denote a number field such as $\Q,\R,$ or $\C$.
\begin{definition}
  Let $M,m\geq 0$ be integers and define $\is=(\y_1,\dots,\y_M)$ and $\x=(x_1,\dots,x_m)$ to be the corresponding sets of variables. A \emph{recurrence system} is a pair $(\Y(t),\P)$, where for $t\geq 0$, $\Y(t)=(\Y_1(t),\dots,\Y_M(t))\in\AFFF{\numberfield(\x)}{M}$ is a family of rational functions in $\x$ and $\P=(\P_1,\dots,\P_M)\in\AFFF{\numberfield(\is)}{M}$ is a family of  rational functions in $\is$ such that for all $t\geq 0$ and $1\leq i\leq M$, we have
\[\Y_{i}(t+1)=P_i(\Y_1(t),\Y_2(t),\dots,\Y_M(t)).\]
\end{definition}

\begin{definition}
We say that a recurrence system $(\Y(t),\P)$ is \emph{Laurent} if for all $t\geq 0$, $\Y(t)$ is a \emph{Laurent polynomial} in $\x$, i.e., $\Y_i(t)\in\numberfield[\x^{\pm1}]$ for all $1\leq i\leq M$, where $\x^{\pm1}:=(x_1^{\pm1},x_2^{\pm1},\dots,x_m^{\pm1}).$  We say that $(\Y(t),\P)$ is an \emph{irreducible Laurent recurrence system} if for each $t\geq 0$ and each $1\leq i\leq M$, $\Y_i(t)$ is either a Laurent monomial or an irreducible Laurent polynomial, and, in addition, any two of these Laurent polynomials are coprime elements of $\numberfield[\x^{\pm1}]$.
\end{definition}

We are now ready to state the desired definition.

\begin{definition}\label{dfn:tau}
  Let $G=(V,E,\wt)$ be a strongly connected weighted digraph. A Laurent recurrence system $(\Y(t),\P)$ is said to be a \emph{weak $\tau$-sequence} for the $R$-system associated with $G$ if there exists an $M\times |V|$ matrix $A=(a_{iv})$ such that setting
  
  \[\X_v(t)=\prod_{i=1}^M \Y_i(t)^{a_{iv}}\]
  produces a \emph{solution to the $R$-system}, i.e., we have $\X(t+1)=\rowm(\X(t))$ for all $t\geq 0$. If in addition the Laurent system $(\Y(t),\P)$ is irreducible and the map $\phi:\numberfield^m\to \mathbb{P}^V(\numberfield)$ given by $(x_1,x_2,\dots,x_m)\mapsto (\X_v(0))_{v\in V}$ is \emph{dominant} then $(\Y(t),\P)$ is said to be a \emph{strong $\tau$-sequence} for the $R$-system associated with $G$.
\end{definition}
Here a \emph{dominant map} is a map whose image is Zariski dense in $\numberfield^V$, so, loosely speaking, requiring the map $\phi$ to be dominant means that for (almost) any initial data $\Init$ of the $R$-system, one can find suitable values of $(x_1,\dots,x_m)$ such that the above substitution gives a solution to the $R$-system that coincides with $\Init$ at time $t=0$. We will often have $M=m$ and $\Y_i(0)=x_i$ for all $i=1,2,\dots,m$, in which case clearly $\phi$ is dominant if and only if $\numberfield^V$ is spanned by the columns of $A$ together with the vector $\allones\in\numberfield^V$.

The following is clear from the definitions:
\begin{proposition}\leavevmode
  \begin{itemize}
  \item Suppose that the $R$-system associated with $G$ admits a strong $\tau$-sequence. Then it has the singularity confinement property.
    \item Suppose that the $R$-system associated with $G$ admits a weak $\tau$-sequence with zero algebraic entropy such that the above map $\phi$ is dominant. Then this $R$-system has zero algebraic entropy.
  \end{itemize}
\end{proposition}

\section{Background on cluster algebras, $T$-systems, and LP-algebras}
In Part~\ref{pt:backgr-clust-algebr}, we will consider a lot of examples of $R$-systems and discuss how they admit various $\tau$-sequences consisting of Laurent polynomials. Before we proceed, we briefly review the theory of cluster and Laurent Phenomenon algebras of~\cite{FZ,FZ2,FZ3,FZ4,LP} and how they give rise to \emph{$T$-systems} defined by Nakanishi~\cite{Nakanishi}.

We first describe cluster algebras of Fomin-Zelevinsky. A \emph{quiver} $Q$ is a digraph without directed cycles of length $1$ and $2$. We denote by $I$ its vertex set. A \emph{seed} is a pair $(X,Q)$ consisting of a quiver $Q$ and a family $X=(X_i)_{i\in I}$ of elements of some unique factorization domain attached to the vertices of $Q$. Given a vertex $k\in I$, one can perform a \emph{mutation} at $k$ producing a new seed $(X',Q')$ defined as follows.
\begin{enumerate}[(i)]
\item For each $i\neq k$, let $X'_i:=X_i$, and set
  \begin{equation}\label{eq:cluster_mutation}
  X'_k:=\frac{\prod_{i\to k} X_i+\prod_{k\to j} X_j}{X_k},
  \end{equation}
  where the products are taken over the arrows in $Q$.
\item For each pair $i\to k,k\to j$ of arrows of $Q$, introduce a new arrow $i\to j$.
\item Reverse the direction of all arrows of $Q$ incident to $k$.
\item Remove all directed cycles of length $2$ in $Q$ and denote the resulting quiver by $Q'$.
\end{enumerate}

One of the fundamental results in the theory of cluster algebras is the \emph{Laurent phenomenon}:
\begin{theorem}[\cite{FZ}]\label{thm:FZ_Laurent}
  Let $(X,Q)$ be a seed and let $k_1,k_2,\dots,k_m\in I$ be a sequence of vertices of $Q$. Let
  \[(X',Q'):=\mu_{k_m}\circ\mu_{k_{m-1}}\circ\dots\circ\mu_{k_1}(X,Q).\]
  Then $X'_i$ is a Laurent polynomial in the variables $X$ for each $i\in I$. 
\end{theorem}

Let $F_k$ be the binomial in the numerator in the right hand side of~\eqref{eq:cluster_mutation}. One can reconstruct a cluster algebra seed $(X,Q)$ from the data $(X,F)$, where $F=(F_i)_{i\in I}$. More generally, a seed in a \emph{Laurent Phenomenon algebra} (\emph{LP algebra} for short) of~\cite{LP} is a pair $(X,F)$, where $X=(X_i)_{i\in I}$ is a family of elements of some unique factorization domain and $F=(F_i)_{i\in I}$ is a family of \emph{exchange polynomials} in $X$ such that each $F_i$ is irreducible, not divisible by $X_j$ for any $j\in I$, and does not depend on $X_i$. Let us give a quick explanation for how the seeds in an LP algebra are mutated. We give a simplified version of the mutation algorithm from~\cite{LP}, and in order for it to work, we restrict ourselves to the generality of seeds coming from \emph{double quivers} introduced in~\cite{ACH}.

\begin{definition}\label{dfn:double_quiver}
  We say that a seed $(X,F)$ \emph{arises from a double quiver} if there exists an integer matrix $B=(b_{ij})_{i,j\in I}$ such that for all $i\in I$, we have $b_{ii}=0$,  and
  \[F_i(X)=\prod_{j:b_{ij}>0} X_j^{b_{ij}}+\prod_{j:b_{ij}<0} X_j^{-b_{ij}}.\]
\end{definition}
Note that unlike in the case of a cluster algebra, we allow the situation when $b_{ij}=0$ while $b_{ji} \not = 0$.

\def\FFk{{\tilde F}_k}
We describe the mutation procedure for LP algebra seeds that arise from a double quiver (this property will hold in all examples we consider in Part~\ref{pt:backgr-clust-algebr}). Let $(X,F)$ be such a seed and choose some index $k\in I$.
Then the \emph{mutation of $(X,F)$ at $k$} is a new seed $\mu_k(X,F)=(X',F')$ defined via the following algorithm.
\begin{enumerate}
\item For $i\neq k$, set $X'_i:=X_i$.
\item If $F_j\neq F_k$ for each $j\neq k$ then set
  \begin{equation}\label{eq:LP_mutation_1}
    X'_k:=\frac{F_k(X)}{X_k}.
  \end{equation}
\item More generally, for $J:=\{j\in I\mid F_j=F_k\}$, we put
  \begin{equation}\label{eq:LP_mutation_2}
X'_k:=\frac{F_k(X)}{\prod_{j\in J} X_j }.
  \end{equation}
\item For $i\in I$, if $F_i$ does not depend on $X_k$ then set $F'_i:=F_i$ (in particular, set $F'_k:=F_k$).
\item Otherwise, if $F_i$ depends on $X_k$, let $\FFk$ be obtained from $F_k$ by setting $X_i$ to $0$. Then $F_i'$ is obtained from $F_i$ by substituting $$X_k\to \frac{\FFk(X)}{X'_k \prod_{j\in J, j \not = k} X_j}$$ and then multiplying the result by a unique Laurent monomial $M$ in $X'$ so that the product is a polynomial not divisible by $X_j$ for all $j\in I$.
\end{enumerate}
We refer the reader to the next part for numerous examples of applying the above algorithm.

\begin{remark}
It may happen that $\mu_k(X,F)$ does not arise from a double quiver. However, it does if for all $j\in I$ such that $F_j$ depends on $X_k$, we also get that $F_k$ depends on $X_j$. 
In all cases considered below we only mutate at $k$-s to which this restriction applies, i.e., we always stay within the generality of double quivers. 
\end{remark}

Similarly to cluster algebras, the Laurent phenomenon still holds in this case.
\begin{theorem}[\cite{LP}]
  Let $(X,Q)$ be an LP algebra seed and let $k_1,k_2,\dots,k_m\in I$ be a sequence of elements of $I$. Let
  \begin{equation}\label{eq:mutation_sequence_cluster}
(X',Q'):=\mu_{k_m}\circ\mu_{k_{m-1}}\circ\dots\circ\mu_{k_1}(X,Q).
  \end{equation}
  Then $X'_i$ is a Laurent polynomial in the variables $X$ for each $i\in I$. 
\end{theorem}

\begin{remark}
  The two substantial differences between this LP algebra setting and the above cluster algebra setting are as follows:
\begin{itemize}
\item in a cluster algebra, the matrix $B$ from Definition~\ref{dfn:double_quiver} is required to be skew-symmetric (or, more generally, \emph{skew-symmetrizable}), and
 \item even when $F_j=F_k$ for some $j\neq k$, one still defines $X'_k$ according to~\eqref{eq:LP_mutation_1}, not~\eqref{eq:LP_mutation_2}. Thus even when the matrix $B$ \emph{is} skew-symmetric, the result of an LP algebra seed mutation can differ from the result of a quiver seed mutation defined above.
 \end{itemize}
 Another difference between cluster algebras and LP algebras is that the latter allows seeds where the exchange polynomials $F_i$ have more than two terms, but in this case the mutation algorithm is more involved. Seeds arising from double quivers provide sufficient generality for our purposes, and we refer the reader to~\cite{LP} for the full account of the theory.
\end{remark}

One can construct certain discrete dynamical systems (called $T$-systems and introduced by Nakanishi~\cite{Nakanishi}) using cluster algebras or LP algebras as follows. Suppose that $(X,Q)$ is a cluster algebra seed. Let $k_1,k_2,\dots,k_m\in I$ be a sequence of mutations and let $\nu:I\to I$ be a bijection such that we have $Q'=\nu(Q)$, where $Q'$ is given by~\eqref{eq:mutation_sequence_cluster} and $\nu(Q)$ is the quiver with vertex set $I$ and arrow set $\nu(i)\to\nu(j)$ for each arrow $i\to j$ of $Q$. This data gives rise to a \emph{$T$-system}, i.e., a family of rational functions $T(t)=(T_i(t))_{i\in I}$ for all $t\geq 0$ defined by $T(0)=X$ and
\[(T(t+1),Q)=\nu^{-1}\left(\mu_{k_m}\circ\mu_{k_{m-1}}\circ\dots\circ\mu_{k_1}(T(t),Q)\right)\]
for all $t\geq 0$. Here we denote $\nu(X,Q):=(\nu(X),\nu(Q))$ and $\nu(X):=(X_{\nu(i)})_{i\in I}$.

Similarly, if $(X,F)$ is an LP algebra seed then one can also define a $T$-system whenever we have a sequence of mutations that produces a seed $(X',F')$ such that for some bijection $\nu:I\to I$, we have
\[F'_{\nu(i)}=F_i(\nu(X'))\]
for all $i\in I$.

{\large\part{Specific examples\label{pt:backgr-clust-algebr}}}

We investigate the behavior of the $R$-system for some classes of strongly connected digraphs. We start off by considering several simple cases in Section~\ref{sec:small-examples} and then we gradually increase the complexity from section to section until we arrive at the case of bidirected digraphs in Section~\ref{sec:bidirected-graphs} which turns out to be the hardest one to deal with, producing completely new beautiful recurrence sequences of (conjecturally) Laurent polynomials. We specify in the beginning of each section whether we are considering $R$-systems with or without coefficients.

\section{Small examples}\label{sec:small-examples}
In this section, we give several warm up examples of $R$-systems exhibiting either singularity confinement or zero algebraic entropy or both of these phenomena. We also give examples that do not have each of these two properties. Unless stated otherwise, we assume that the $R$-systems are coefficient-free. This section contains a large amount of computations which we carry out in detail, we refer the reader to the next sections where the needed verifications are less tedious and often omitted.

\subsection{Directed cycles}\label{sec:directed-cycles}
The simplest strongly connected digraph is a \emph{directed cycle}, i.e., a digraph $G=(V,E)$ with $V=[n]$ and $E=\{(i,i+1)\mid i\in [n]\}$, where the indices are taken modulo $n$. Note that $|\Arb(G,i)|=1$ for any $i\in [n]$ and thus the $R$-system consists of iterating a monomial transformation. More explicitly, for $\X=(\X_i)\in\RP {[n]}$, applying~\eqref{eq:arb} yields $\X'_i=\X_{i-1}$, where $\X'=\rowm(\X)$ and the indices again are taken modulo $n$. Thus in this case the (coefficient-free) $R$-system is periodic with period $n$, which is actually a special case of~\cite[Theorem~30]{GR2}. Note also that the $R$-system with coefficients evolves according to 
\[\X_i'=\wtt {i-1}i \X_{i-1}.\]
Thus the universal $R$-system with coefficients is still periodic with period $n$ and clearly has the singularity confinement and zero algebraic entropy properties. 

We now consider three examples of strongly connected digraphs that are obtained from posets using the procedure described in Figure~\ref{fig:poset}.

\subsection{An example with a $\tau$-sequence and zero algebraic entropy}\label{sec:poset_3}
Consider the poset $P$ shown in Figure \ref{fig:example_with_tau_zero}. First, let us build a $\tau$-sequence for the $R$-system associated with the digraph $G(P)$ (which is defined via the construction in Figure~\ref{fig:poset}).
\begin{figure}

\def\scl{0.3}
\begin{tikzpicture}[yscale=1.5]
\node[draw, circle, scale=\scl, fill=black,label=below:{$X_1(t)$}] (1) at (0,0) { };
\node[draw, circle, scale=\scl, fill=black,label=left:{$X_2(t)$}] (2)  at (-1,1) { };
\node[draw, circle, scale=\scl, fill=black,label=right:{$X_3(t)$}] (3)  at (1,1) { };
\node[draw, circle, scale=\scl, fill=black,label=above:{$X_4(t)$}] (4)  at (0,2) { };
\node[draw, circle, scale=\scl, fill=black,label=right:{$X_5(t)$}]  (5) at (2,2) { };
\node[draw, circle, scale=\scl, fill=black,label=right:{$X_6(t)$}] (6)  at (3,3) { };
\draw (1)--(2)--(4)--(3)--(1);
\draw (3)--(5)--(6);
\end{tikzpicture}
    \caption{The poset $P$ from Section~\ref{sec:poset_3}.}
    \label{fig:example_with_tau_zero}
\end{figure}
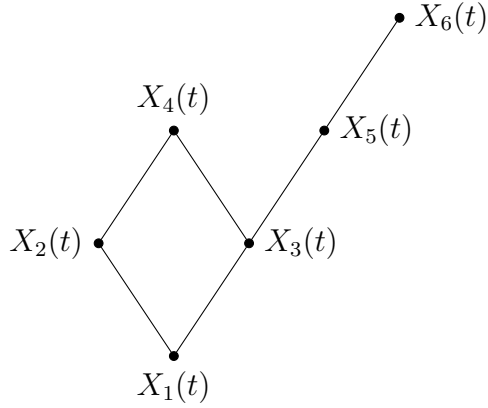

Consider a seed $(Y_t,Y_{t+1},\dots,Y_{t+6},Z_t)$ associated with the quiver $Q$ in Figure~\ref{fig:example_mutations} (left).

\begin{theorem}
Mutating $Q$ at the vertices $v_0$ and $u$ yields a quiver that is isomorphic to $Q$ via a map $\nu$ that sends $v_i$ to $v_{i-1}$ for $i=0,1,\dots,6$ (the indices are taken modulo $7$), so we get a $T$-system. Let us denote by $Y_{t+7}$ and $Z_{t+1}$ the variables obtained after mutating at $v_0$ and $u$ respectively. Then the exchange relations for $Y_{t}$ and $Z_{t}$ are
\begin{equation*}
  \begin{split}
Y_{t} Y_{t+7} &= Y_{t+3}Y_{t+4} + Z_{t},\\
Z_{t} Z_{t+1} &= Y_{t+1} Y_{t+3} Y_{t+5} Y_{t+7}  + Y_{t+2} Y_{t+4}^2 Y_{t+6}.
  \end{split}
\end{equation*}
\end{theorem}

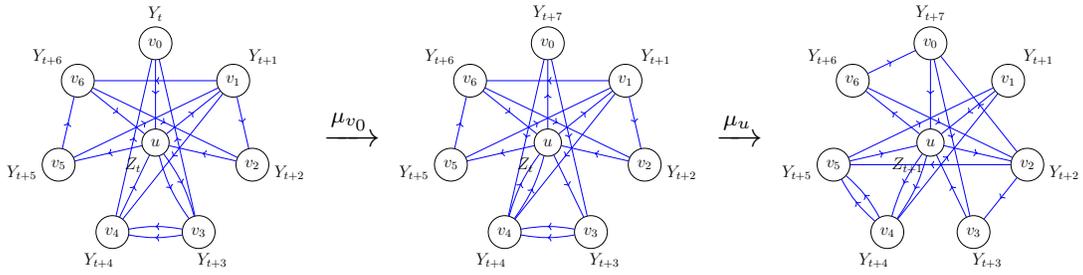
\begin{figure}
  \def\rad{2.5}
  \def\arrscl{2}
  \def\Zanchor{45}
  \resizebox{1.0\textwidth}{!}{
\begin{tabular}{ccccc}
\begin{tikzpicture}[baseline=(7.base)]
  \node[draw,circle] (7) at (0,0) {$u$};
  \node[draw,circle] (0) at ({90-0*360/7}:\rad) {$v_0$};
  \node[draw,circle] (1) at ({90-1*360/7}:\rad) {$v_1$};
  \node[draw,circle] (2) at ({90-2*360/7}:\rad) {$v_2$};
  \node[draw,circle] (3) at ({90-3*360/7}:\rad) {$v_3$};
  \node[draw,circle] (4) at ({90-4*360/7}:\rad) {$v_4$};
  \node[draw,circle] (5) at ({90-5*360/7}:\rad) {$v_5$};
  \node[draw,circle] (6) at ({90-6*360/7}:\rad) {$v_6$};
  \foreach \i/\j in {4/0,0/7,3/0,1/6,1/2,1/4,7/1,5/1,6/2,2/7,4/7,7/5,5/6,6/7}{
    \draw[postaction={decorate},blue] (\i) -- (\j);
  }
  \draw[postaction={decorate},blue] (3) to[bend right=10] (4);
  \draw[postaction={decorate},blue] (3) to[bend right=-10] (4);
  \draw[postaction={decorate},blue] (7) to[bend right=10] (3);
  \draw[postaction={decorate},blue] (7) to[bend right=-10] (3);
  \node[anchor={180+90-0*360/7}] at (0.{90-0*360/7}) {$Y_{t}$};
  \node[anchor={180+90-1*360/7}] at (1.{90-1*360/7}) {$Y_{t+1}$};
  \node[anchor={180+90-2*360/7}] at (2.{90-2*360/7}) {$Y_{t+2}$};
  \node[anchor={180+90-3*360/7}] at (3.{90-3*360/7}) {$Y_{t+3}$};
  \node[anchor={180+90-4*360/7}] at (4.{90-4*360/7}) {$Y_{t+4}$};
  \node[anchor={180+90-5*360/7}] at (5.{90-5*360/7}) {$Y_{t+5}$};
  \node[anchor={180+90-6*360/7}] at (6.{90-6*360/7}) {$Y_{t+6}$};
  \node[anchor={\Zanchor}] at (7.{180+\Zanchor}) {$Z_{t}$};
\end{tikzpicture} &
\scalebox{\arrscl}{$\xrightarrow{\mu_{v_0}}$}
  &                                                
\begin{tikzpicture}[baseline=(7.base)]
  \node[draw,circle] (7) at (0,0) {$u$};
  \node[draw,circle] (0) at ({90-0*360/7}:\rad) {$v_0$};
  \node[draw,circle] (1) at ({90-1*360/7}:\rad) {$v_1$};
  \node[draw,circle] (2) at ({90-2*360/7}:\rad) {$v_2$};
  \node[draw,circle] (3) at ({90-3*360/7}:\rad) {$v_3$};
  \node[draw,circle] (4) at ({90-4*360/7}:\rad) {$v_4$};
  \node[draw,circle] (5) at ({90-5*360/7}:\rad) {$v_5$};
  \node[draw,circle] (6) at ({90-6*360/7}:\rad) {$v_6$};
  \foreach \i/\j in {0/4,7/0,0/3,1/6,1/2,1/4,7/1,5/1,6/2,2/7,7/5,5/6,6/7,7/3}{
    \draw[postaction={decorate},blue] (\i) -- (\j);
  }

  \draw[postaction={decorate},blue] (4) to[bend right=10] (7);
  \draw[postaction={decorate},blue] (4) to[bend right=-10] (7);
  
  \draw[postaction={decorate},blue] (3) to[bend right=10] (4);
  \draw[postaction={decorate},blue] (3) to[bend right=-10] (4);
  \node[anchor={180+90-0*360/7}] at (0.{90-0*360/7}) {$Y_{t+7}$};
  \node[anchor={180+90-1*360/7}] at (1.{90-1*360/7}) {$Y_{t+1}$};
  \node[anchor={180+90-2*360/7}] at (2.{90-2*360/7}) {$Y_{t+2}$};
  \node[anchor={180+90-3*360/7}] at (3.{90-3*360/7}) {$Y_{t+3}$};
  \node[anchor={180+90-4*360/7}] at (4.{90-4*360/7}) {$Y_{t+4}$};
  \node[anchor={180+90-5*360/7}] at (5.{90-5*360/7}) {$Y_{t+5}$};
  \node[anchor={180+90-6*360/7}] at (6.{90-6*360/7}) {$Y_{t+6}$};
  \node[anchor={\Zanchor}] at (7.{180+\Zanchor}) {$Z_{t}$};
\end{tikzpicture}
  & \scalebox{\arrscl}{$\xrightarrow{\mu_{u}}$}
  &
    \begin{tikzpicture}[baseline=(7.base)]
  \node[draw,circle] (7) at (0,0) {$u$};
  \node[draw,circle] (0) at ({90-0*360/7}:\rad) {$v_0$};
  \node[draw,circle] (1) at ({90-1*360/7}:\rad) {$v_1$};
  \node[draw,circle] (2) at ({90-2*360/7}:\rad) {$v_2$};
  \node[draw,circle] (3) at ({90-3*360/7}:\rad) {$v_3$};
  \node[draw,circle] (4) at ({90-4*360/7}:\rad) {$v_4$};
  \node[draw,circle] (5) at ({90-5*360/7}:\rad) {$v_5$};
  \node[draw,circle] (6) at ({90-6*360/7}:\rad) {$v_6$};
  \foreach \i/\j in {5/1,1/7,4/1,2/0,2/3,2/5,7/2,6/2,0/3,3/7,5/7,7/6,6/0,0/7}{
    \draw[postaction={decorate},blue] (\i) -- (\j);
  }
  \draw[postaction={decorate},blue] (4) to[bend right=10] (5);
  \draw[postaction={decorate},blue] (4) to[bend right=-10] (5);
  \draw[postaction={decorate},blue] (7) to[bend right=10] (4);
  \draw[postaction={decorate},blue] (7) to[bend right=-10] (4);
  \node[anchor={180+90-0*360/7}] at (0.{90-0*360/7}) {$Y_{t+7}$};
  \node[anchor={180+90-1*360/7}] at (1.{90-1*360/7}) {$Y_{t+1}$};
  \node[anchor={180+90-2*360/7}] at (2.{90-2*360/7}) {$Y_{t+2}$};
  \node[anchor={180+90-3*360/7}] at (3.{90-3*360/7}) {$Y_{t+3}$};
  \node[anchor={180+90-4*360/7}] at (4.{90-4*360/7}) {$Y_{t+4}$};
  \node[anchor={180+90-5*360/7}] at (5.{90-5*360/7}) {$Y_{t+5}$};
  \node[anchor={180+90-6*360/7}] at (6.{90-6*360/7}) {$Y_{t+6}$};
  \node[anchor={\Zanchor}] at (7.{180+\Zanchor}) {$Z_{t+1}$};
\end{tikzpicture} 
\end{tabular}
}

    \caption{Mutating at $v_0$ and $u$ corresponds to rotating the quiver $Q$ on the left hand side clockwise by $2\pi/7$.}
    \label{fig:example_mutations}
\end{figure}

\begin{proof}
It is an easy exercise in quiver mutation, see Figure \ref{fig:example_mutations}. 
\end{proof}

Now we construct  a monomial transformation from this Laurent recurrence system to the $R$-system associated with $G(P)$ as follows.

\begin{equation}\label{eq:example_with_zero}
  \begin{split}
    X_1(t) &= \frac{Z_{t-1}}{Y_{t+2} Y_{t+3}},\\
X_2(t) &= \frac{Y_{t+1} Y_{t+5}}{Y_{t+2} Y_{t+4}},\\
X_3(t) &= \frac{Y_{t} Y_{t+6}}{Y_{t+3}^2},\\
X_4(t) &= \frac{Y_{t+1} Y_{t+6}}{Z_{t}},\\
X_5(t) &=  \frac{Y_{t+1} Y_{t+6}}{Y_{t+3} Y_{t+4}},\\
X_6(t) &=  \frac{Y_{t+2} Y_{t+6}}{Y_{t+3} Y_{t+5}}.
  \end{split}
\end{equation}

\begin{theorem}
 Assigning $X_i(t)$ to the vertices of $P$ as in Figure~\ref{fig:example_with_tau_zero} and assigning the value of $1$ to the extra vertex $\tb$ of $G(P)$ gives a solution to the $R$-system associated with $G(P)$. Thus, the $Y$-s and $Z$-s constructed above form a weak $\tau$-sequence for this $R$-system.
\end{theorem}

\begin{proof}
  The non-trivial toggle relations to verify are
  \begin{equation*}
    \begin{split}
\frac{Z_{t-1}}{Y_{t+2} Y_{t+3}} \frac{Z_{t}}{Y_{t+3} Y_{t+4}} &= \frac{Y_{t+1} Y_{t+5}}{Y_{t+2} Y_{t+4}} + \frac{Y_{t} Y_{t+6}}{Y_{t+3}^2},\\
\frac{Y_{t} Y_{t+6}}{Y_{t+3}^2} \frac{Y_{t+1} Y_{t+7}}{Y_{t+4}^2} &= \frac{    \frac{Y_{t+1} Y_{t+6}}{Z_{t}} + \frac{Y_{t+1} Y_{t+6}}{Y_{t+3} Y_{t+4}}    }{    \frac{Y_{t+3} Y_{t+4}}{Z_{t}}    },\\
\frac{Y_{t+1} Y_{t+6}}{Z_{t}} \frac{Y_{t+2} Y_{t+7}}{Z_{t+1}} &= \frac{1}{    \frac{Y_{t+3} Y_{t+5}}{Y_{t+2} Y_{t+6}}  + \frac{Y_{t+4}^2}{Y_{t+1} Y_{t+7}}   },
    \end{split}
  \end{equation*}
  which are equivalent respectively to
  \begin{align*}
Z_{t-1} Z_{t} &= Y_{t} Y_{t+2} Y_{t+4} Y_{t+6}  + Y_{t+1} Y_{t+3}^2 Y_{t+5},\\
Y_{t} Y_{t+7} &= Y_{t+3}Y_{t+4} + Z_{t},\\
Z_{t} Z_{t+1} &= Y_{t+1} Y_{t+3} Y_{t+5} Y_{t+7}  + Y_{t+2} Y_{t+4}^2 Y_{t+6}.\qedhere\\
    \end{align*}
\end{proof}

\begin{remark}
 Note that the monomial map~\eqref{eq:example_with_zero} is not dominant since we have $X_2(t) X_6(t) = X_5(t)$. There is however a way to add some coefficients to~\eqref{eq:example_with_zero} that depend on the residue of $t$ modulo $9$ so that the reduction becomes dominant.
\end{remark}

As Figure~\ref{fig:zero_entropy} suggests, the degrees of the values of the $R$-system associated with $G(P)$ grow quadratically and thus it has zero algebraic entropy.

\begin{figure}
  \resizebox{1\textwidth}{!}{
    \def\heightt{6cm}
        \def\widthh{6cm}
\begin{tabular}{c|c|c}
\begin{tikzpicture}
  \begin{axis}[
    height=\heightt,width=\widthh,
		xlabel=$t$,
		ylabel=$\deg(X(t))$]
	\addplot[very thick] file {zero_entropy.dat};
	\end{axis}
      \end{tikzpicture}
  &
    \begin{tikzpicture}
	\begin{axis}[
    height=\heightt,width=\widthh,
		xlabel=$t$,
		ylabel=$\sqrt{\deg(X(t))}$]
	\addplot[very thick] file {zero_entropy_sqrt.dat};
      \end{axis}
    \end{tikzpicture}
    &
    \begin{tikzpicture}
	\begin{axis}[
    height=\heightt,width=\widthh,
		xlabel=$t$,
		ylabel=$\log\deg(X(t))$]
	\addplot[very thick] file {zero_entropy_log.dat};
      \end{axis}
    \end{tikzpicture}    
\end{tabular}
}
  \caption{\label{fig:zero_entropy}The quadratic growth is apparent from the plot of the degrees of $X(t)$ from Section~\ref{sec:poset_3} for $0\leq t\leq 100$.}
\end{figure}
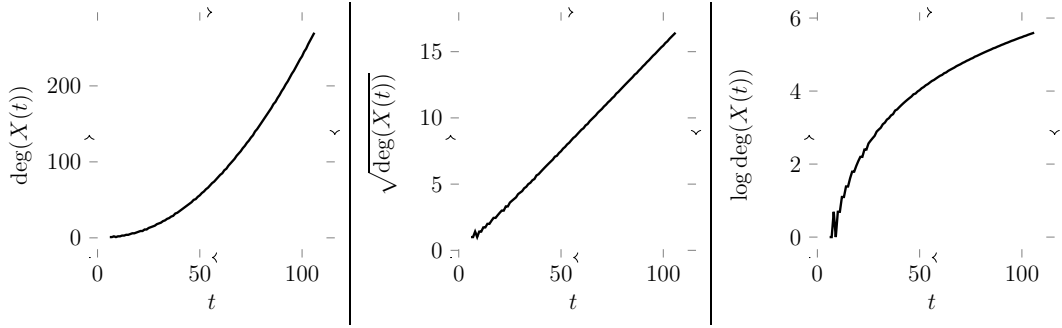

\subsection{An example with a $\tau$-sequence and non-zero algebraic entropy}\label{sec:poset_1}
Consider the poset $P$ shown in Figure~\ref{fig:example_with_tau_nonzero}. First, let us build a $\tau$-sequence for the $R$-system associated with $G(P)$.
\begin{figure}

\def\scl{0.3}
\begin{tikzpicture}[yscale=1.5]
\node[draw, circle, scale=\scl, fill=black,label=below:{$X_1(t)$}] (1) at (0,0) { };
\node[draw, circle, scale=\scl, fill=black,label=below:{$X_2(t)$}] (2)  at (2,0) { };
\node[draw, circle, scale=\scl, fill=black,label=above:{$X_3(t)$}] (3)  at (-1,1) { };
\node[draw, circle, scale=\scl, fill=black,label=right:{$X_4(t)$}] (4)  at (1,1) { };
\node[draw, circle, scale=\scl, fill=black,label=right:{$X_5(t)$}]  (5) at (1,2) { };
\draw (1)--(3);
\draw (1)--(4);
\draw (2)--(4);
\draw (4)--(5);
\end{tikzpicture}

    \caption{The poset $P$ from Section~\ref{sec:poset_1}.}
    \label{fig:example_with_tau_nonzero}
\end{figure}
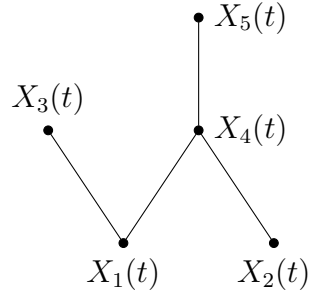

 Consider a Laurent phenomenon algebra with the seed 
$$Y_{t}, Y_{t+1}, Y_{t+2}, Y_{t+3}, Y_{t+4}, Y_{t+5}, Y_{t+6}, Y_{t+7}, Y_{t+8}, Y_{t+9}, Z_{t}, Z_{t+1}, Z_{t+2}$$
and the exchange polynomials 
\begin{align*}
F_{Y_{t}} &= Y_{t+3}Y_{t+7}Z_{t}Z_{t+1}Z_{t+2} + Y_{t+2}^2Y_{t+5}^2Y_{t+8}^2,\\
F_{Y_{t+1}} &= Y_{t+2} Y_{t+4} Y_{t+5} Y_{t+8}^2 + Y_{t+3}Y_{t+6}Z_{t}Y_{t+9},\\
F_{Y_{t+2}} &= Y_{t} Y_{t+5} Y_{t+6} Y_{t+9}^2 + Y_{t+4}Y_{t+7}^2Z_{t}Z_{t+1}^2Z_{t+2},\\
F_{Y_{t+3}} &= Y_{t} Y_{t+4} Z_{t+2}  + Y_{t+2}^2Y_{t+5}Y_{t+6}Z_{t},\\
F_{Y_{t+4}} &= Y_{t+1} Y_{t+2} Y_{t+5}^2 Y_{t+8} Y_{t+9} + Y_{t+3}Y_{t+7}Z_{t}Z_{t+1}^2Z_{t+2},\\
F_{Y_{t+5}} &= Y_{t+1} Y_{t+4}^2 Y_{t+7} Y_{t+8} Z_{t+2} + Y_{t+2}Y_{t+3}Y_{t+6}^2Z_{t}^2Y_{t+9},\\
F_{Y_{t+6}} &= Y_{t+1} Y_{t+2} Y_{t+5}^2 Y_{t+8} Y_{t+9} + Y_{t+3}Y_{t+7}Z_{t}Z_{t+1}^2Z_{t+2},\\
F_{Y_{t+7}} &= Y_{t+2}^2 Y_{t+5} Y_{t+6} Z_{t}  + Y_{t}Y_{t+4}Z_{t+2},\\
F_{Y_{t+8}} &= Y_{t} Y_{t+3} Y_{t+6} Z_{t+1} + Y_{t+1}^2Y_{t+4}Y_{t+5}Y_{t+7},\\
F_{Y_{t+9}} &= Y_{t+1} Y_{t+4} Y_{t+7} Z_{t+2} + Y_{t+2}^2Y_{t+5}Y_{t+6}Y_{t+8},\\
F_{Z_{t}} = F_{Z_{t+1}} = F_{Z_{t+2}} &= Y_{t} Y_{t+3} Y_{t+6} Y_{t+9} + Y_{t+1}Y_{t+2}Y_{t+4}Y_{t+5}Y_{t+7}Y_{t+8}.
\end{align*} 

\begin{theorem}
 Mutating $Y_{t}$ into $Y_{t+10}$ and then $Z_{t}$ into $Z_{t+3}$ preserves the shape of the original seed up to the shift of indices $i \mapsto i+1$. In other words, we have an analog of a $T$-system.
 The exchange relations for $Y_{t}$ and $Z_{t}$ are
 \begin{align*}
  Y_{t} Y_{t+10} &= Y_{t+3}Y_{t+7}Z_{t}Z_{t+1}Z_{t+2} + Y_{t+2}^2Y_{t+5}^2Y_{t+8}^2,\\
 Z_{t} Z_{t+1} Z_{t+2} Z_{t+3} &= Y_{t+2} Y_{t+3} Y_{t+5} Y_{t+6} Y_{t+8} Y_{t+9} + Y_{t+1} Y_{t+4} Y_{t+7} Y_{t+10}.
 \end{align*}
\end{theorem}

\begin{remark}
 We can see the non-cluster nature of this example in both the exchange polynomials and in the exchange relations. For example, $Y_{t+1}$ appears in the exchange polynomial $F_{Y_{t}}$, but ${Y_{t}}$ does not appear in the exchange polynomial $F_{Y_{t+1}}$. 
 Also, the product of the four variables $Z_{t} Z_{t+1} Z_{t+2} Z_{t+3}$ on the left hand side of the exchange relation is something that cannot happen in a cluster algebra. It is explained by the rules of the Laurent Phenomenon algebra dynamics and the fact that $Z_{t}$, $Z_{t+1}$, and $Z_{t+2}$ all have the same exchange polynomial at the moment of mutation of $Z_{t}$, and thus the result of the mutation is given by~\eqref{eq:LP_mutation_2}.
\end{remark}

\begin{proof}
 After mutating at $Y_{t}$, the new exchange polynomials are 
\begin{align*}
F_{Y_{t}} &= Y_{t+3}Y_{t+7}Z_{t}Z_{t+1}Z_{t+2} + Y_{t+2}^2Y_{t+5}^2Y_{t+8}^2,\\
F_{Y_{t+1}} &= Y_{t+2} Y_{t+4} Y_{t+5} Y_{t+8}^2 + Y_{t+3}Y_{t+6}Z_{t}Y_{t+9},\\
F_{Y_{t+2}} &= Y_{t+3} Y_{t+5} Y_{t+6} Y_{t+9}^2 + Y_{t+4}Y_{t+7}Z_{t+1}Y_{t+10},\\
F_{Y_{t+3}} &= Y_{t+4} Y_{t+5} Y_{t+8}^2  Z_{t+2} + Y_{t+6}Z_{t}Y_{t+10},\\
F_{Y_{t+4}} &= Y_{t+1} Y_{t+2} Y_{t+5}^2 Y_{t+8} Y_{t+9} + Y_{t+3}Y_{t+7}Z_{t}Z_{t+1}^2Z_{t+2},\\
F_{Y_{t+5}} &= Y_{t+1} Y_{t+4}^2 Y_{t+7} Y_{t+8} Z_{t+2} + Y_{t+2}Y_{t+3}Y_{t+6}^2Z_{t}^2Y_{t+9},\\
F_{Y_{t+6}} &= Y_{t+1} Y_{t+2} Y_{t+5}^2 Y_{t+8} Y_{t+9} + Y_{t+3}Y_{t+7}Z_{t}Z_{t+1}^2Z_{t+2},\\
F_{Y_{t+7}} &= Y_{t+4} Y_{t+5} Y_{t+8}^2 Z_{t+2}  + Y_{t+6}Z_{t}Y_{t+10},\\
F_{Y_{t+8}} &= Y_{t+1}^2 Y_{t+4} Y_{t+5} Y_{t+10} + Y_{t+3}^2Y_{t+6}Z_{t}Z_{t+1}^2Z_{t+2},\\
F_{Y_{t+9}} &= Y_{t+1} Y_{t+4} Y_{t+7} Z_{t+2} + Y_{t+2}^2Y_{t+5}Y_{t+6}Y_{t+8},\\
F_{Y_{t+10}} &= Y_{t+3}Y_{t+7}Z_{t}Z_{t+1}Z_{t+2} + Y_{t+2}^2Y_{t+5}^2Y_{t+8}^2,\\
F_{Z_{t}} = F_{Z_{t+1}} = F_{Z_{t+2}} &= Y_{t+2} Y_{t+3} Y_{t+5} Y_{t+6} Y_{t+8} Y_{t+9} + Y_{t+1} Y_{t+4} Y_{t+7} Y_{t+10}.
  \end{align*}
Now we mutate at $Z_{t}$, obtaining the new exchange polynomials 
\begin{align*}
F_{Y_{t+1}} &= Y_{t+4}Y_{t+8}Z_{t+1}Z_{t+2}Z_{t+3} + Y_{t+3}^2Y_{t+6}^2Y_{t+9}^2,\\
F_{Y_{t+2}} &= Y_{t+3} Y_{t+5} Y_{t+6} Y_{t+9}^2 + Y_{t+4}Y_{t+7}Z_{t+1}Y_{t+10},\\
F_{Y_{t+3}} &= Y_{t+1} Y_{t+6} Y_{t+7} Y_{t+10}^2 + Y_{t+5}Y_{t+8}^2Z_{t+1}Z_{t+2}^2Z_{t+3},\\
F_{Y_{t+4}} &= Y_{t+1} Y_{t+5} Z_{t+3}  + Y_{t+3}^2Y_{t+6}Y_{t+7}Z_{t+1},\\
F_{Y_{t+5}} &= Y_{t+2} Y_{t+3} Y_{t+6}^2 Y_{t+9} Y_{t+10} + Y_{t+4}Y_{t+8}Z_{t+1}Z_{t+2}^2Z_{t+3},\\
F_{Y_{t+6}} &= Y_{t+2} Y_{t+5}^2 Y_{t+8} Y_{t+9} Z_{t+3} + Y_{t+3}Y_{t+4}Y_{t+7}^2Z_{t+1}^2Y_{t+10},\\
F_{Y_{t+7}} &= Y_{t+2} Y_{t+3} Y_{t+6}^2 Y_{t+9} Y_{t+10} + Y_{t+4}Y_{t+8}Z_{t+1}Z_{t+2}^2Z_{t+3},\\
F_{Y_{t+8}} &= Y_{t+3}^2 Y_{t+6} Y_{t+7} Z_{t+1}  + Y_{t+1}Y_{t+5}Z_{t+3},\\
F_{Y_{t+9}} &= Y_{t+1} Y_{t+4} Y_{t+7} Z_{t+2} + Y_{t+2}^2Y_{t+5}Y_{t+6}Y_{t+8},\\
F_{Y_{t+10}} &= Y_{t+2} Y_{t+5} Y_{t+8} Z_{t+3} + Y_{t+3}^2Y_{t+6}Y_{t+7}Y_{t+9},\\
F_{Z_{t+1}} = F_{Z_{t+2}} = F_{Z_{t+3}} &= Y_{t+1} Y_{t+4} Y_{t+7} Y_{t+10} + Y_{t+2}Y_{t+3}Y_{t+5}Y_{t+6}Y_{t+8}Y_{t+9},\\
  \end{align*} 
which differ from the original seed by the shift $i \mapsto i+1$, as desired. 

Let us consider for example in detail what happens to $F_{Y_{t+8}}$, all other computations are similar. Setting $Y_{t+8}=0$ in the original $F_{Y_{t}}$, we get the substitution $$Y_{t} \leftarrow \frac{Y_{t+3}Y_{t+7}Z_{t}Z_{t+1}Z_{t+2}}{Y_{t+10}}$$
to be performed in $F_{Y_{t+8}}$. Bringing the resulting 
$$\frac{Y_{t+3}Y_{t+7}Z_{t}Z_{t+1}Z_{t+2}}{Y_{t+10}} Y_{t+3} Y_{t+6} Z_{t+1} + Y_{t+1}^2Y_{t+4}Y_{t+5}Y_{t+7}$$
to a common denominator and cancelling the common factor $Y_{t+7}$, we indeed obtain 
$$F_{Y_{t+8}} = Y_{t+1}^2 Y_{t+4} Y_{t+5} Y_{t+10} + Y_{t+3}^2Y_{t+6}Z_{t}Z_{t+1}^2Z_{t+2}.$$
Then, setting $Y_{t+8}=0$ in the new
$$F_{Z_{t}} = Y_{t+2} Y_{t+3} Y_{t+5} Y_{t+6} Y_{t+8} Y_{t+9} + Y_{t+1} Y_{t+4} Y_{t+7} Y_{t+10}$$
(which we assume to have computed already), we 
get the substitution
$$Z_{t} \leftarrow \frac{Y_{t+1} Y_{t+4} Y_{t+7} Y_{t+10}}{Z_{t+1}Z_{t+2}Z_{t+3}}$$
to be performed in $F_{Y_{t+8}}$. Bringing the resulting 
$$Y_{t+1}^2 Y_{t+4} Y_{t+5} Y_{t+10} + Y_{t+3}^2Y_{t+6}\frac{Y_{t+1} Y_{t+4} Y_{t+7} Y_{t+10}}{Z_{t+1}Z_{t+2}Z_{t+3}}Z_{t+1}^2Z_{t+2}$$
to a common denominator and cancelling the common factor 
$Y_{t+1} Y_{t+4} Y_{t+10}$, we indeed obtain
\begin{equation*}
F_{Y_{t+8}} = Y_{t+3}^2 Y_{t+6} Y_{t+7} Z_{t+1}  + Y_{t+1}Y_{t+5}Z_{t+3}.\qedhere
\end{equation*}
\end{proof}

Now, consider a reduction from this Laurent recurrence system to the $R$-system associated with $G(P)$ as follows.
\begin{align*}  
X_1(t) &= \frac{Y_{t} Y_{t+9}}{Y_{t+2} Y_{t+7} Z_{t} Z_{t+1}},\\
X_2(t) &= \frac{Y_{t+1} Y_{t+4} Y_{t+5} Y_{t+8}}{Y_{t+3} Y_{t+6} Z_{t} Z_{t+1}},\\
X_3(t) &= \frac{Y_{t+1} Y_{t+9}}{Y_{t+2} Y_{t+8} Z_{t+1}},\\
X_4(t) &= \frac{Y_{t+1} Y_{t+2} Y_{t+5}^2 Y_{t+8} Y_{t+9}}{Y_{t+3} Y_{t+7} Z_{t} Z_{t+1}^2 Z_{t+2}},\\
X_5(t) &=  \frac{Y_{t+2} Y_{t+5} Y_{t+6} Y_{t+9}}{Y_{t+4} Y_{t+7} Z_{t+1} Z_{t+2}}.
\end{align*}

\begin{theorem}
 The $X_i(t)$-s satisfy the relations of the $R$-system associated with $G(P)$. Thus, the $Y(t)$-s constructed above form a weak $\tau$-sequence for this $R$-system.
\end{theorem}

\begin{proof}
 The only non-trivial toggle relations to check are 
 $$X_1(t) X_1(t+1) = X_3(t) + X_4(t) \text{ and } X_4(t)X_4(t+1) = \frac{X_5(t)}{X_1^{-1}(t) + X_2^{-1}(t)},$$
 which are equivalent to
 \begin{align*}
 Y_{t} Y_{t+10} &= Y_{t+3}Y_{t+7}Z_{t}Z_{t+1}Z_{t+2} + Y_{t+2}^2Y_{t+5}^2Y_{t+8}^2 \quad\text{ and}\\
 Z_{t} Z_{t+1} Z_{t+2} Z_{t+3} &= Y_{t+2} Y_{t+3} Y_{t+5} Y_{t+6} Y_{t+8} Y_{t+9} + Y_{t+1} Y_{t+4} Y_{t+7} Y_{t+10}
 \end{align*}
 respectively. 
 \end{proof}

\begin{remark}
 Note that this reduction is not dominant since $X_2(t) X_5(t) = X_4(t)$. 
\end{remark}

As Figure~\ref{fig:nonzero_entropy} suggests, the degrees of the polynomials appearing in the $R$-system associated with $G(P)$ grow exponentially and thus it does not have zero algebraic entropy.

\begin{figure}
  \resizebox{1\textwidth}{!}{
    \def\heightt{6cm}
        \def\widthh{6cm}
\begin{tabular}{c|c|c}
\begin{tikzpicture}
  \begin{axis}[
    height=\heightt,width=\widthh,
		xlabel=$t$,
		ylabel=$\deg(X(t))$]
	\addplot[very thick] file {nonzero_entropy.dat};
	\end{axis}
      \end{tikzpicture}
  &
    \begin{tikzpicture}
        \begin{axis}[
    height=\heightt,width=\widthh,
        	xlabel=$t$,
        	ylabel=$\sqrt{\deg(X(t))}$]
        \addplot[very thick] file {nonzero_entropy_sqrt.dat};
      \end{axis}
    \end{tikzpicture}
    &
    \begin{tikzpicture}
        \begin{axis}[
    height=\heightt,width=\widthh,
        	xlabel=$t$,
        	ylabel=$\log\deg(X(t))$]
        \addplot[very thick] file {nonzero_entropy_log.dat};
      \end{axis}
    \end{tikzpicture}    
\end{tabular}
}
  \caption{\label{fig:nonzero_entropy}The exponential growth is apparent from the plot of the degrees of $X(t)$ from Section~\ref{sec:poset_1} for $0\leq t\leq 100$.}
\end{figure}
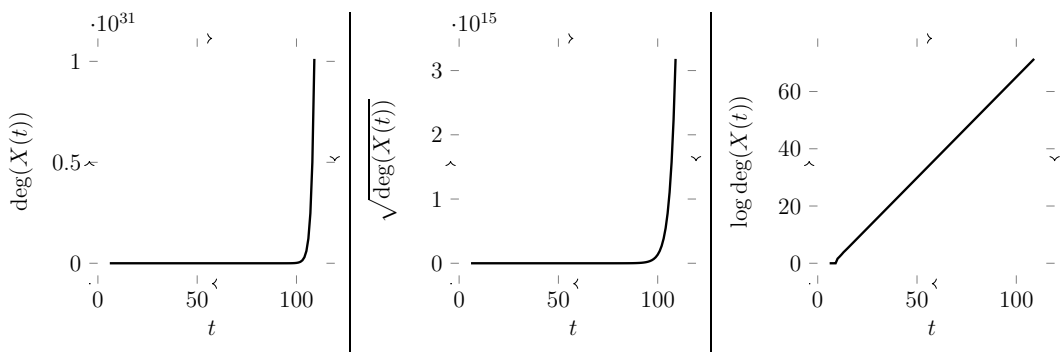

\subsection{An example without singularity confinement}\label{sec:an-example-without}
Consider the poset $P$ shown in Figure \ref{fig:example_without}.
After fourteen iterations of the $R$-system associated with $G(P)$, based on our computer simulation, we get the following values at the vertices.
\begin{align*}
X_1(14) &\sim Y_{23}Y_{15}Y_{7}Y_{0}/(Y_{22}Y_{17}Y_{14}Y_{13}Y_{10}Y_{9}Y_{6}Y_{5}Y_{3}),\\
X_2(14) &\sim Y_{21}Y_{18}Y_{17}Y_{14}Y_{8}Y_{5}Y_{2}Y_{3}^2/(Y_{22}Y_{19}Y_{16}Y_{10}Y_{1}Y_{4}Y_{0}),\\
X_3(14) &\sim Y_{23}Y_{18}Y_{15}Y_{12}Y_{9}Y_{6}Y_{2}Y_{1}/(Y_{22}Y_{19}Y_{16}Y_{14}Y_{13}Y_{11}Y_{10}Y_{8}Y_{5}Y_{3}Y_{4}),\\
X_4(14) &\sim Y_{23}Y_{20}Y_{17}Y_{14}Y_{7}Y_{8}Y_{2}Y_{3}Y_{4}/(Y_{24}Y_{22}Y_{12}Y_{6}Y_{1}),\\
X_5(14) &\sim Y_{23}Y_{20}Y_{15}Y_{12}Y_{11}Y_{8}Y_{1}Y_{0}/(Y_{22}Y_{21}Y_{16}Y_{13}^2Y_{10}^2Y_{5}Y_{4}),  
\end{align*}
where $Y$-s are the irreducible polynomial factors, in the order they appear as factors in $X_i(t)$-s, $1 \leq t \leq 14$, and $A\sim B$ means that $A/B$ is a monomial in the initial variables, i.e., an invertible element of the corresponding Laurent polynomial ring. We see that not a single factor $Y_i$ has disappeared at the fourteenth step, i.e., the corresponding singularity was not confined. Of course, this does not prove that such confinement does not happen in the future. Nevertheless, we conjecture the following.

\begin{figure}
  
\def\scl{0.3}
\begin{tikzpicture}[yscale=1.5]
\node[draw, circle, scale=\scl, fill=black,label=below:{$X_1(t)$}] (1) at (0,0) { };
\node[draw, circle, scale=\scl, fill=black,label=below:{$X_2(t)$}] (2)  at (2,0) { };
\node[draw, circle, scale=\scl, fill=black,label=left:{$X_3(t)$}] (3)  at (-1,1) { };
\node[draw, circle, scale=\scl, fill=black,label=above:{$X_4(t)$}] (4)  at (1,1) { };
\node[draw, circle, scale=\scl, fill=black,label=above:{$X_5(t)$}]  (5) at (-2,2) { };
\draw (1)--(3);
\draw (1)--(4);
\draw (2)--(4);
\draw (3)--(5);
\end{tikzpicture}

    \caption{The poset $P$ from Section~\ref{sec:an-example-without}.}
    \label{fig:example_without}
\end{figure}
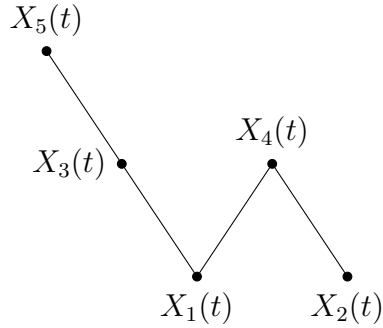

\begin{conjecture}
 The coefficient-free $R$-system associated with $G(P)$ for the poset $P$ from Figure~\ref{fig:example_without} does not possess a strong $\tau$-sequence and does not have the singularity confinement property. 
\end{conjecture}

\section{$R$-systems with Somos and Gale-Robinson $\tau$-sequences}\label{sec:somos-gale-robinson}

In this section, we solve the \emph{inverse problem}:  we start with an interesting Laurent sequence $(\y_N)_{N\geq 0}$ defined by a recurrence relation: 
\begin{equation}\label{eq:y}
  \begin{split}
\y_0&=x_0,\y_1=x_1,\dots,\y_{n-1}=x_{n-1},\quad \text{and}\\
\y_N\y_{N+n}&=P(\y_{N+1},\y_{N+2},\dots,\y_{N+n-1})\quad\text{for $N\geq 0$},
  \end{split}
\end{equation}
where $P$ is a polynomial. We then construct a strongly connected digraph $G$ for which the recurrence system $(\Y(t),P)$ given by $\Y(t)=(\Y_i(t))_{i=0}^n$ with $\Y_i(t)=y_{t+i}$ is conjecturally a strong or a weak $\tau$-sequence for the $R$-system associated with $G$.

\def\GR{\operatorname{GR}}

We will be interested only in the case of \emph{Gale-Robinson sequences} (which includes Somos sequences as a special case).

\begin{definition}
  Let $p<q$ be two distinct positive integers and let $n> q$. The \emph{Three Term Gale-Robinson sequence} $\GR(n;p,q)=(\y_N)_{N\geq0}$ is defined by~\eqref{eq:y} with
  \[P(\y_1,\dots,\y_{n-1}):=\alpha \y_p\y_{n-p}+\beta \y_q\y_{n-q},\]
  where $\alpha,\beta$ are arbitrary elements of $\field^\ast$.
\end{definition}

\begin{definition}
  Let $0<p<q<r$ be three distinct positive integers and let $n=p+q+r$ be their sum. The \emph{Four Term Gale-Robinson sequence} $\GR(p,q,r)=(\y_N)_{N\geq0}$ is defined by~\eqref{eq:y} with
  \begin{equation}\label{eq:GR_four_def}
P(\y_1,\dots,\y_{n-1}):=\alpha \y_p\y_{n-p}+\beta \y_q\y_{n-q}+\gamma\y_r\y_{n-r},
  \end{equation}
  where $\alpha,\beta,\gamma$ are arbitrary elements of $\field^\ast$.
\end{definition}

The following result was conjectured by Gale and Robinson~\cite{Gale}, \cite[E15]{Robinson} and shown for both Three and Four Term Gale-Robinson sequences by Fomin and Zelevinsky.
\begin{proposition}[\cite{FZCube}]\leavevmode
  \begin{enumerate}
  \item For each $N\geq 0$, the entry $\y_N$ of $\GR(n;p,q)$ is a Laurent polynomial in $x_0,\dots,x_{n-1},\alpha,\beta$.
  \item For each $N\geq 0$, the entry $\y_N$ of $\GR(p,q,r)$ is a Laurent polynomial in $x_0,\dots,x_{n-1},\alpha,\beta,\gamma$.
  \end{enumerate}
\end{proposition}

Let us also mention the algebraic entropy of these sequences:
\begin{proposition}\label{prop:GR_quadratic}
The degrees of the Laurent polynomials $(\y_N)$ grow quadratically in $N$. In particular, the Three and Four Term Gale-Robinson sequences have zero algebraic entropy.
\end{proposition}

These results can be deduced from the following two observations. First, the Three Term (resp., Four Term) Gale-Robinson sequence can obtained by a reduction from the \emph{octahedron recurrence}, also known as the \emph{Hirota-Miwa equation}~\cite[Eq.~(1.1)]{Miwa} (resp., from the \emph{cube recurrence}, also known as the \emph{discrete BKP equation}~\cite[Eq.~(1.4)]{Miwa}), as it was first suggested by J.~Propp, see~\cite{FZCube}. Second, for the octahedron recurrence (resp., for the cube recurrence), Speyer~\cite{Sp} (resp., Carroll and Speyer~\cite{CS}) gave a combinatorial formula from which the quadratic degree growth follows almost trivially. See~\cite[Corollary~3.3]{GP3} (resp., \cite[Theorem~2.12]{Cube}) for details.

In the remainder of this section, we describe which strongly connected digraphs correspond to the sequences $\GR(p+q+2;p,q)$ and $\GR(1,q,r)$, leaving the other cases open. We will see that the Somos-$6$ and Somos-$7$ sequences as special cases of the above, specifically, as $\GR(1,2,3)$ and $\GR(1,2,4)$ respectively. We will see later in Figure~\ref{fig:subgraph_somos_dp3}  that the Somos-$4$ and Somos-$5$ sequences appear as strong $\tau$-sequences for certain $R$-systems. (The Somos-$4$ sequence digraph is also shown in Figure~\ref{fig:somos_4}, while the Somos-$5$ sequence provides a weak $\tau$-sequence for the digraph shown in Figure~\ref{fig:somos_5}.) Thus each of the four Somos sequences arises as a (strong or weak) $\tau$-sequence for an $R$-system.

\subsection{The Three Term Gale-Robinson sequence}\label{sec:three-term-gale}
Let $0<p<q$ and $n:=p+q+2$. We consider the sequence $\GR(n;p,q)=(\y_N)$ of the form~\eqref{eq:y} for $P(\y_1,\dots,\y_{n-1})$ given by
\begin{equation}\label{eq:GR_three_rec}
P(\y_1,\dots,\y_{n-1}):=\alpha\y_{q}\y_{p+2}+\beta\y_{p}\y_{q+2}.
\end{equation}

We now describe a certain digraph $G=(V,E)$. We let
\[V:=\{\tb,A=B_0=C_0,B_1,B_2,\dots,B_{p-1},C_1,C_2,\dots,C_{q-1},D=B_p=C_q\}\]
be its set of vertices, thus $|V|=n-1$. The edges of $G$ are
\[E:=\{(\tb,A),(B_0,B_1),\dots,(B_{p-1},B_{p}),(C_0,C_1),\dots,(C_{q-1},C_{q}),(D,\tb)\}.\]
The weight function $\wt:E\to \field$ is defined by
\[\wtt {B_0} {B_1}=\dots=\wtt{B_{p-1}}{B_P}:=\alpha;\quad \wtt {C_0} {C_1} =\dots=\wtt{C_{q-1}}{C_q}:=\beta,\]
with the weights of the remaining two edges $(\tb,A)$ and $(D,\tb)$ equal to $1$.

Thus $G$ is obtained from an \emph{$O$-shaped poset} by adding a vertex $\tb$ via the procedure described in Figure~\ref{fig:poset}. See Figure~\ref{fig:GR_three} for an example. We now explain how to write down a weak $\tau$-sequence for the $R$-system associated with $G$ in terms of $\Y(t)=(\y_{i+t})_{i=0}^{n-1}$. Consider the lattice $\Z^n$ with basis vectors $e_0,e_1,\dots,e_{n-1}$, and for an element $h=h_0e_0+\dots+h_{n-1}e_{n-1}\in\Z^{n}$, set $(\Y(t))^h:=\Y_0(t)^{h_0}\Y_1(t)^{h_1}\dots\Y_{n-1}(t)^{h_{n-1}}$. We are going to find vectors
\[a=b^\parr0=c^\parr0,b^\parr1,\dots,b^\parr{p-1},c^\parr 1,\dots,c^\parr{q-1},d=b^\parr p=c^\parr q\in\Z^{n}\]
such that substituting
\begin{equation}\label{eq:GR_three_subs}
\X_{B_i}(t):=(\Y(t))^{b^\parr i},\quad \X_{C_j}(t):=(\Y(t))^{c^\parr j}
\end{equation}
for $0\leq i\leq p$ and $0\leq j\leq q$ and $\X_\tb(t):=1$ gives a weak $\tau$-sequence for the $R$-system associated with $G$. Note that the cases $i,j=0$, $i=p$, and $j=q$ include the vertices $A$ and $D$.

In order to describe these vectors, we introduce the operator $\sigma$ on $\Z^n$ defined by
\begin{equation}\label{eq:sigma}
\sigma(h_0,\dots,h_{n-2},h_{n-1}):=(0,h_0,\dots,h_{n-2}).
\end{equation}
Now, set
\[a:=e_1+e_{n-1}-e_{p+1}-e_{q+1},\quad f:=e_{p+2}+e_{q},\quad g:=e_{p}+e_{q+2}, \]
and define $b^\parr i$ and $c^\parr j$ by induction on $i$ and $j$ as follows. Set $b^\parr0:=c^\parr0:=a$, and for each $1\leq i\leq p$ and $1\leq j\leq q$, set
\begin{equation}\label{eq:GR_three_b_c}
b^\parr i:=\sigma(b^\parr{i-1})+a+f-e_0,\quad c^\parr j:=\sigma(c^\parr{j-1})+a+g-e_0.
\end{equation}

One easily verifies that $b^\parr{p}=c^\parr{q}$ so we let $d$ be equal to either one of them. Another direct computation shows the following.
\begin{proposition}
The substitution~\eqref{eq:GR_three_subs} provides a weak $\tau$-sequence for the $R$-system associated with $G$.
\end{proposition}

\begin{example}
  Let $p=2,q=3,n=7$, so the recurrence relation is
  \begin{equation}\label{eq:GR_three_rec_ex}
\y_{N}\y_{N+7}=\alpha\y_{N+3}\y_{N+4}+\beta\y_{N+2}\y_{N+5}.
  \end{equation}
  
  We calculate the vectors $b^\parr i$ and $c^\parr j$ in Table~\ref{tab:GR_three}.

  \begin{table}
 \begin{center}
 \begin{tabular}{|r|ccccccc|}\hline
 $i=$   &    0  &    1  &    2  &    3  &    4  &    5 &   6 \\\hline
 $f=$   & 0     & 0     & 0     & 1     & 1     & 0 &   0\\
 $g=$   & 0     & 0     & 1     & 0     & 0     & 1 &   0\\\hline
   $b^\parr0=c^\parr0=a=$
        & 0     & 1     & 0    & -1    & -1    & 0 &   1\\\hline
   $b^\parr1=$
        & -1    & 1     & 1    & 0     & -1    & -1 &   1\\
   $d=b^\parr2=$
        & -1    & 0     & 1    & 1     & 0    & -1 &   0\\\hline
   $c^\parr1=$
        & -1    & 1     & 2     & -1    & -2    & 0 &   1\\
   $c^\parr2=$
        & -1    & 0     & 2     & 1    & -2    & -1 &   1\\
   $d=c^\parr3=$
        & -1    & 0     & 1     & 1    & 0    & -1 &   0\\\hline
 \end{tabular}
\end{center}
\caption{\label{tab:GR_three} Computing the vectors $a=b^\parr0,\dots,b^\parr p=d$ and $a=c^\parr0,\dots,c^\parr q=d$ for $p=2$, $q=3$, and $n=7$ using~\eqref{eq:GR_three_b_c}.}
\end{table}

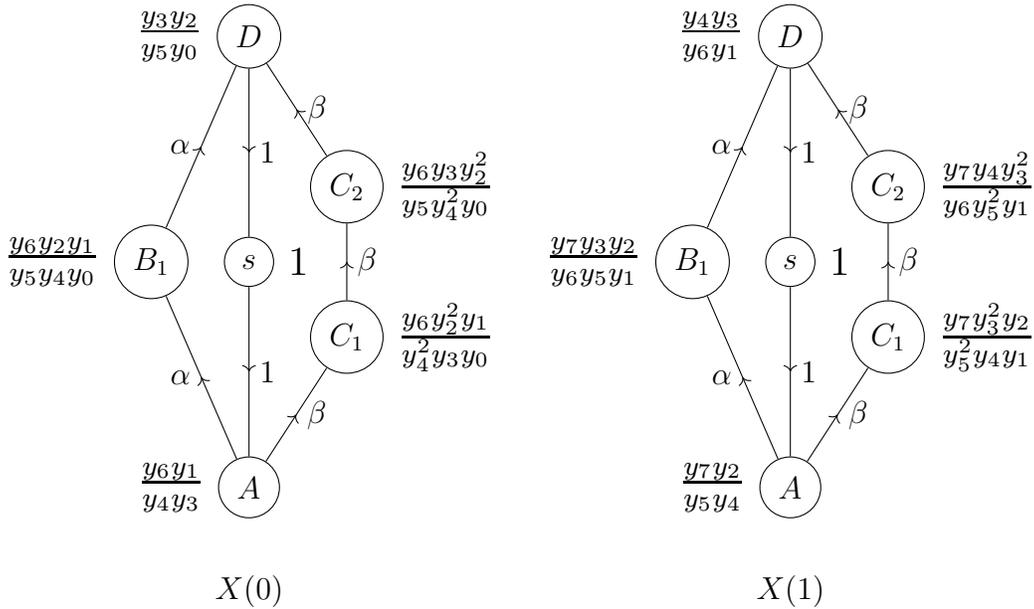
\begin{figure}
  \def\nodesc{1}
  \def\labelsc{1.3}
  \def\tikzscx{1.3}
  \def\tikzscy{1}
  \def\arrsc{1}
  \begin{tabular}{cc}
    \begin{tikzpicture}[xscale=\tikzscx,yscale=\tikzscy]
    \node[scale=\nodesc,draw,circle] (A) at (0,0) {$A$};
    \node[scale=\nodesc,draw,circle] (B1) at (-1,3) {$B_1$};
    \node[scale=\nodesc,draw,circle] (C1) at (1,2) {$C_1$};
    \node[scale=\nodesc,draw,circle] (C2) at (1,4) {$C_2$};
    \node[scale=\nodesc,draw,circle] (D) at (0,6) {$D$};
    \node[scale=\nodesc,draw,circle] (TB) at (0,3) {$\tb$};
    \draw[postaction={decorate}] (A)--(B1) node[pos=0.5,left,scale=\arrsc] {$\alpha$};
    \draw[postaction={decorate}] (B1)--(D) node[pos=0.5,left,scale=\arrsc] {$\alpha$};
    \draw[postaction={decorate}] (A)--(C1) node[pos=0.5,right,scale=\arrsc] {$\beta$};
    \draw[postaction={decorate}] (C1)--(C2)node[pos=0.5,right,scale=\arrsc] {$\beta$};
    \draw[postaction={decorate}] (C2)--(D) node[pos=0.5,right,scale=\arrsc] {$\beta$};
    \draw[postaction={decorate}] (D)--(TB)node[pos=0.5,right,scale=\arrsc] {$1$};
    \draw[postaction={decorate}] (TB)--(A)node[pos=0.5,right,scale=\arrsc] {$1$};
    \node[anchor=east,scale=\labelsc] (Al) at (A.west) {$\frac{\y_6\y_1}{\y_4\y_3}$};
    \node[anchor=east,scale=\labelsc] (B1l) at (B1.west) {$\frac{\y_6\y_2\y_1}{\y_5\y_4\y_0}$};
    \node[anchor=west,scale=\labelsc] (C1l) at (C1.east) {$\frac{\y_6\y^2_2\y_1}{\y_4^2\y_3\y_0}$};
    \node[anchor=west,scale=\labelsc] (C2l) at (C2.east) {$\frac{\y_6\y_3\y^2_2}{\y_5\y^2_4\y_0}$};
    \node[anchor=east,scale=\labelsc] (Dl) at (D.west) {$\frac{\y_3\y_2}{\y_5\y_0}$};
    
    \node[anchor=west,scale=\labelsc] (TBl) at (TB.east) {$1$};
  \end{tikzpicture}&
    \begin{tikzpicture}[xscale=\tikzscx,yscale=\tikzscy]
    \node[scale=\nodesc,draw,circle] (A) at (0,0) {$A$};
    \node[scale=\nodesc,draw,circle] (B1) at (-1,3) {$B_1$};
    \node[scale=\nodesc,draw,circle] (C1) at (1,2) {$C_1$};
    \node[scale=\nodesc,draw,circle] (C2) at (1,4) {$C_2$};
    \node[scale=\nodesc,draw,circle] (D) at (0,6) {$D$};
    \node[scale=\nodesc,draw,circle] (TB) at (0,3) {$\tb$};
    \draw[postaction={decorate}] (A)--(B1) node[pos=0.5,left,scale=\arrsc] {$\alpha$};
    \draw[postaction={decorate}] (B1)--(D) node[pos=0.5,left,scale=\arrsc] {$\alpha$};
    \draw[postaction={decorate}] (A)--(C1) node[pos=0.5,right,scale=\arrsc] {$\beta$};
    \draw[postaction={decorate}] (C1)--(C2)node[pos=0.5,right,scale=\arrsc] {$\beta$};
    \draw[postaction={decorate}] (C2)--(D) node[pos=0.5,right,scale=\arrsc] {$\beta$};
    \draw[postaction={decorate}] (D)--(TB)node[pos=0.5,right,scale=\arrsc] {$1$};
    \draw[postaction={decorate}] (TB)--(A)node[pos=0.5,right,scale=\arrsc] {$1$};
    \node[anchor=east,scale=\labelsc] (Al) at (A.west) {$\frac{\y_7\y_2}{\y_5\y_4}$};
    \node[anchor=east,scale=\labelsc] (B1l) at (B1.west) {$\frac{\y_7\y_3\y_2}{\y_6\y_5\y_1}$};
    \node[anchor=west,scale=\labelsc] (C1l) at (C1.east) {$\frac{\y_7\y^2_3\y_2}{\y_5^2\y_4\y_1}$};
    \node[anchor=west,scale=\labelsc] (C2l) at (C2.east) {$\frac{\y_7\y_4\y^2_3}{\y_6\y^2_5\y_1}$};
    \node[anchor=east,scale=\labelsc] (Dl) at (D.west) {$\frac{\y_4\y_3}{\y_6\y_1}$};
    \node[anchor=west,scale=\labelsc] (TBl) at (TB.east) {$1$};
  \end{tikzpicture}\\ & \\
$\X(0)$ & $\X(1)$
    
  \end{tabular}
  \caption{\label{fig:GR_three} The digraph $G$ corresponding to $\GR(7;2,3)$ with assignments $\X(t)$ for $t=0$ (left) and $t=1$ (right).}
\end{figure}

Using them, we obtain an assignment $\X(t)$ of values of the $R$-system shown in Figure~\ref{fig:GR_three} for $t=0,1$. Let us verify some toggle relations using this data. For the vertices $A$, $C_1$, and $D$, the toggle relation~\eqref{eq:toggle} reads respectively
\begin{equation}\label{eq:GR_three_check}
\begin{split}
\frac{\y_6\y_1}{\y_4\y_3}\cdot\frac{\y_7\y_2}{\y_5\y_4}&=\alpha\frac{\y_6\y_2\y_1}{\y_5\y_4\y_0}+
  \beta\frac{\y_6\y_2^2 \y_1}{\y_4^2\y_3\y_0},\\
\frac{\y_6\y_2^2\y_1}{\y_4^2\y_3\y_0}\cdot\frac{\y_7\y_3^2\y_2}{\y_5^2\y_4\y_1}&=
  \beta\frac{\y_6\y_3\y_2^2}{\y_5\y_4^2\y_0} \left(
    \beta\frac{\y_5 \y_4}{\y_7\y_2}\right)^{-1},\quad\text{and}\\
  \frac{\y_3\y_2}{\y_5\y_0}\cdot\frac{\y_4\y_3}{\y_6\y_1}&=
\left(\alpha\frac{\y_6\y_5 \y_1}{\y_7\y_3\y_2}+\beta\frac{\y_6\y_5^2\y_1}{\y_7\y_4\y_3^2}\right)^{-1}.
\end{split}
\end{equation}
After all cancellations, quite miraculously, the relations for $A$ and $D$ reduce to~\eqref{eq:GR_three_rec_ex} while the relation for $C_1$ holds trivially. It also holds trivially for the remaining vertices $B_1,C_2$, and $\tb$. More generally, for arbitrary $p$ and $q$, the only non-trivial cancellations happen for the vertices $A$ and $D$.
\end{example}

\begin{remark}
When $p=q$, all the constructions in this section make sense. The right hand side of~\eqref{eq:GR_three_rec}, however, becomes a single monomial, so the sequence itself is not that interesting. The values of the $R$-system in this case are periodic by~\cite[Proposition~71]{GR1}, and indeed, its values (which are monomial functions in the entries of $\GR(2p+2;p,p)$) are periodic even though the sequence itself is not periodic. 
\end{remark}

Based on our computations, we find that something much more general appears to hold.
\begin{conjecture}\label{conj:GR_three_master}
For any $0<p<q$ and $n=p+q+2$, the universal $R$-system with coefficients associated with the digraph $G$ admits $\GR(n;p,q)$ as a strong $\tau$-sequence with the substitution that differs from~\eqref{eq:GR_three_subs} by some periodic monomial factors in the edge and vertex variables.
\end{conjecture}
In particular, this includes the conjecture that the entries of $\GR(p+q+2;p,q)$ are irreducible Laurent polynomials. A corollary to Conjecture~\ref{conj:GR_three_master} would be that the $R$-system associated with $G$ has the singularity confinement and zero algebraic entropy properties.

\subsection{The Four Term Gale-Robinson sequence}\label{sec:four-term-gale}
In this section, we focus on the sequences $\GR(1,q,r)$ for $1<q<r$. Note that in the previous section, each strongly connected digraph that we got was obtained from a poset using Remark~\ref{rmk:poset}. We will see that this is no longer the case for the Four Term Gale-Robinson sequences, in particular, for $\GR(1,2,3)$ and $\GR(1,2,4)$, also known as \emph{Somos-$6$} and \emph{Somos-$7$}, respectively. Throughout this section, we assume that
\[\alpha=\beta=\gamma=1.\]

Let us describe the digraph $G=(V,E)$ whose $R$-system will admit $\GR(1,q,r)$ as a weak $\tau$-sequence. The vertex set of $G$ will be
\begin{equation*}\label{eq:}
\begin{split}
  V=&\{A=A_0,A_1,\dots,A_r,A_{r+1}=B=B_0,B_1,\dots,B_{q-2},B_{q-1}=A\}\cup \\
    &\{C=C_0,C_1,\dots,C_q,C_{q+1}=D=D_0,D_1,\dots,D_{r-2},D_{r-1}=C\}.
\end{split}
\end{equation*}
The edges of $G$ are
\begin{equation*}\label{eq:}
\begin{split}
  E=  &\{(A,B),(A,D),(A_0,A_1),\dots,(A_r,A_{r+1}),(B_0,B_1),\dots,(B_{q-2},B_{q-1})\}\cup\\
      &\{(C,B),(C,D),(C_0,C_1),\dots,(C_q,C_{q+1}),(D_0,D_1),\dots,(D_{r-2},D_{r-1})\}.
\end{split}
\end{equation*}

\def\mfun{m}
It remains to find for each vertex $v\in V$ a vector $\mfun'(v)\in\Z^n$ such that substituting $\X_v(t)=(\Y(t))^{\mfun'(v)}$ yields a solution to the $R$-system. Here again $\Y(t)=(\y_{i+t})_{i=0}^{n-1}$ where $(\y_N)_{N\geq0}$ is the sequence $\GR(1,q,r)$. Note that $\X(t)$ is an element of the projective space so equivalently we can just specify for every edge $(u,v)\in E$ a vector $\mfun(u,v):=\mfun'(v)-\mfun'(u)$. In other words, we prefer to specify the \emph{ratios} of the $\X_v(t)$-s rather than the $\X_v(t)$-s themselves.

\def\D{\Delta}
\def\d{\delta}
Recall that $\sigma:\Z^n\to\Z^n$ is the shift operator given by~\eqref{eq:sigma}. For a vector $\D=(\d_1,\dots,\d_{n-1})\in\Z^{n-1}$ and an integer $\d\in\Z$, we write $(\d,\D):=(\d,\d_1,\dots,\d_{n-1})\in\Z^n$ and $(\D,\d):=(\d_1,\dots,\d_{n-1},\d)\in\Z^n$. Finally, we denote by $f,g,h\in\Z^{n-1}$ the exponents in the right hand side of~\eqref{eq:GR_four_def}:
\[f:=e_1+e_{n-1},\quad g:=e_q+e_{n-q},\quad h:=e_r+e_{n-r}.\]

Before we construct the function $\mfun(u,v)$, let us give a simple way to locally verify the toggle relations.

\def\yy{{\tilde{y}}}
\begin{proposition}\label{prop:universes}
  Suppose that we have constructed the function $\mfun(u,v)$ such that for some vertex $v\in V$, one of the following holds:
  \begin{enumerate}
  \item\label{item:expand} the vertex $v$ has three outgoing edges $(v,w_1),(v,w_2),(v,w_3)$ and one incoming edge $(u,v)$, and there is a vector $\D\in\Z^{n-1}$ such that $\mfun(u,v)=(\D,1)$ and 
    \[\mfun(v,w_1)=(-1,\D+f),\quad \mfun(v,w_2)=(-1,\D+g),\quad \mfun(v,w_3)=(-1,\D+h);\]
  \item\label{item:collapse} the vertex $v$ has three incoming edges $(u_1,v),(u_2,v),(u_3,v)$ and one outgoing edge $(v,w)$, and there is a vector $\D\in\Z^{n-1}$ such that $\mfun(v,w)=(1,\D)$ and 
    \[\mfun(u_1,v)=(\D+f,-1),\quad \mfun(u_2,v)=(\D+g,-1),\quad \mfun(u_3,v)=(\D+h,-1);\]
  \item\label{item:zero} the vertex $v$ has one incoming edge $(u,v)$ and one outgoing edge $(v,w)$, and there is a vector $\D\in\Z^{n-1}$ such that $\mfun(u,v)=(\D,0)$ and $\mfun(v,w)=(0,\D)$.
  \end{enumerate}
  Then the toggle relation~\eqref{eq:toggle} is satisfied at $v$.
\end{proposition}
\begin{proof}
  This is just a restatement of the toggle relations. Suppose for instance that Case~\eqref{item:collapse} happens. After rescaling all entries of $\X(t)$ and $\X(t+1)$, we may assume that $\X_v(t)=\X_v(t+1)=1$. Let us put $\yy=(\y_t,\dots,\y_{t+n})$, and for a vector $\D=(\d_1,\dots,\d_{n-1})\in\Z^{n-1}$ and two integers $\d_0,\d_n\in\Z$, write
  \[\yy^{(\d_0,\D,\d_n)}:=\y_t^{\d_0}\y_{t+1}^{\d_1}\dots\y_{t+n}^{\d_n}.\]
    Then~\eqref{eq:GR_four_def} becomes
    \begin{equation}\label{eq:GR_four_yy}
\yy^{(1,0,1)}=\yy^{(0,f,0)}+\yy^{(0,g,0)}+\yy^{(0,h,0)},
    \end{equation}
    while~\eqref{eq:toggle} for $v$ becomes
    \begin{equation*}\label{eq:}
\begin{split}
  1=\X_v(t)\X_v(t+1)&= \X_w(t) \left(\frac1{\X_{u_1}(t+1)}+\frac1{\X_{u_2}(t+1)}+\frac1{\X_{u_3}(t+1)} \right)^{-1}\\
  &=\yy^{(1,\Delta,0)} \left(\yy^{(0,\D+f,-1)}+\yy^{(0,\D+g,-1)}+\yy^{(0,\D+h,-1)}\right)^{-1}\\
  &=\frac{\yy^{(1,\Delta,0)}}{\yy^{(0,\D,-1)}} \left(\yy^{(0,f,0)}+\yy^{(0,g,0)}+\yy^{(0,h,0)}\right)^{-1}.
\end{split}
    \end{equation*}
The result thus follows by~\eqref{eq:GR_four_yy}. The other two cases are handled similarly.
\end{proof}

We are going to construct the function $\mfun$ so that for \emph{every} vertex $v\in V$, either one of the three possibilities~\eqref{item:expand}--\eqref{item:zero} happens (and of course $\mfun$ will be of the form $\mfun(u,v)=\mfun'(v)-\mfun'(u)$ for some function $\mfun':V\to \Z^n$).

Let $e_1,\dots,e_{n-1}$ be the coordinate basis of $\Z^{n-1}$, and define
\[\D_A:=e_{q-1}-e_q-e_{n-1};\quad \D_B:=-e_1-e_{r+1}+e_{r+2};\]
\[\D_C:=e_{r-1}-e_r-e_{n-1};\quad \D_D:=-e_1-e_{q+1}+e_{q+2}.\]
We now define the function $\mfun$:
\begin{equation*}
  \begin{split}
\resizebox{1\textwidth}{!} 
{
 $\displaystyle\mfun(A,B):=(-1,g+\D_A)=(g+\D_B,-1);\quad \mfun(A,D)=(-1,h+\D_A)=(g+\D_D,-1);$ 
}\\
\resizebox{1\textwidth}{!} 
{
$\displaystyle\mfun(C,D):=(-1,h+\D_C)=(h+\D_D,-1);\quad \mfun(C,B)=(-1,g+\D_C)=(h+\D_B,-1),$   
}
  \end{split}
\end{equation*}

and for the remaining edges it is defined as follows:
\begin{equation*}\label{eq:}
\begin{split}
\mfun(A_i,A_{i+1})=\sigma^i(-1,f+\D_A),\quad \mfun(B_j,B_{j+1})=\sigma^j(1,\D_B), 
\end{split}
\end{equation*}
for $0\leq i\leq r$ and $0\leq j\leq q-2$, and 
\begin{equation*}\label{eq:}
\begin{split}
\mfun(C_i,C_{i+1})=\sigma^i(-1,f+\D_C),\quad \mfun(D_j,D_{j+1})=\sigma^j(1,\D_D), 
\end{split}
\end{equation*}
for $0\leq i\leq q$ and $0\leq j\leq r-2$. Note in particular that
\[\mfun(A_r,B)=(f+\D_B,-1),\quad \mfun(B_{q-2},A)=(\D_A,1),\]
\[\mfun(C_q,D)=(f+\D_D,-1),\quad \mfun(D_{r-2},C)=(\D_C,1).\]
It is now trivial to see that one of the three options in Proposition~\ref{prop:universes} holds for each vertex of $G$. One also easily checks that $\mfun$ indeed satisfies $\mfun(u,v)=\mfun'(v)-\mfun'(u)$ for some $\mfun':V\to\Z$ since its (signed) sum over each (undirected) cycle of $G$ is zero.

\begin{corollary}
  There exists a unique assignment $\X(t)\in\RP V$ such that for any edge $(u,v)\in E$, we have
  \[\frac{\X_v(t)}{\X_u(t)}=(\Y(t))^{\mfun(u,v)}\]
  for the function $\mfun$ constructed above. This assignment satisfies the toggle relations~\eqref{eq:toggle} and thus provides a weak $\tau$-sequence for the $R$-system associated with $G$.
\end{corollary}

According to our computations, even though the digraphs $G$ constructed in this section admit solutions in terms of Gale-Robinson sequences, the associated $R$-systems do not seem to have the singularity confinement or zero algebraic entropy properties. This means that one cannot in general write every initial data for the $R$-system associated with $G$ in terms of the initial seed for the Gale-Robinson sequence.

\begin{example}
  Consider the case $\GR(1,2,3)$. In this case $(\y_N)$ is the Somos-$6$ sequence defined by the recurrence relation
  \begin{equation}\label{eq:somos_6}
\y_N\y_{N+6}=\y_{N+1}\y_{N+5}+\y_{N+2}\y_{N+4}+\y_{N+3}^2.
  \end{equation}
  
  Thus
  \[f=(1,0,0,0,1),\quad g=(0,1,0,1,0),\quad h=(0,0,2,0,0).\]
  We now calculate
  \[\D_A=(1,-1,0,0,-1),\quad \D_B=(-1,0,0,-1,1),\]
  \[\D_C=(0,1,-1,0,-1),\quad \D_D=(-1,0,-1,1,0).\]
  Using this, it is easy to compute $\X(t)$ for all $t$. For $t=0$, $\X(t)$ is shown in Figure~\ref{fig:GR_four_ex} (left), and for other values of $t$ one just shifts all indices by $1$. For example, we find that 
  \[\frac{\X_{A_1}(0)}{\X_{A}(0)}=\frac{\y_1/\y_3}{\y_2\y_0/\y_3\y_1}=(\Y(0))^{(-1,2,-1,0,0,0)}=(\Y(0))^{(-1,f+\D_A)},\]
  since $f+\D_A=(1,0,0,0,1)+(1,-1,0,0,-1)=(2,-1,0,0,0)$. We can also check that the toggle relations are satisfied as predicted by Proposition~\ref{prop:universes}. For example, for $v=A$, we get
  \[\X_A(0)\X_A(1)=(\X_{A_1}(0)+\X_{B}(0)+\X_D(0))\left(\frac1{\X_B(1)}\right)^{-1}=\frac{\y_1\y_5+\y_2\y_4+\y_3^2}{\y_3\y_5}\cdot\frac{\y_5\y_3}{\y_6\y_4}.\]
  Since the left hand side is equal to
  \[\frac{\y_2\y_0}{\y_3\y_1}\cdot\frac{\y_3\y_1}{\y_4\y_2}=\frac{\y_0}{\y_4},\]
  the toggle relation~\eqref{eq:toggle} again reduces to~\eqref{eq:somos_6}.

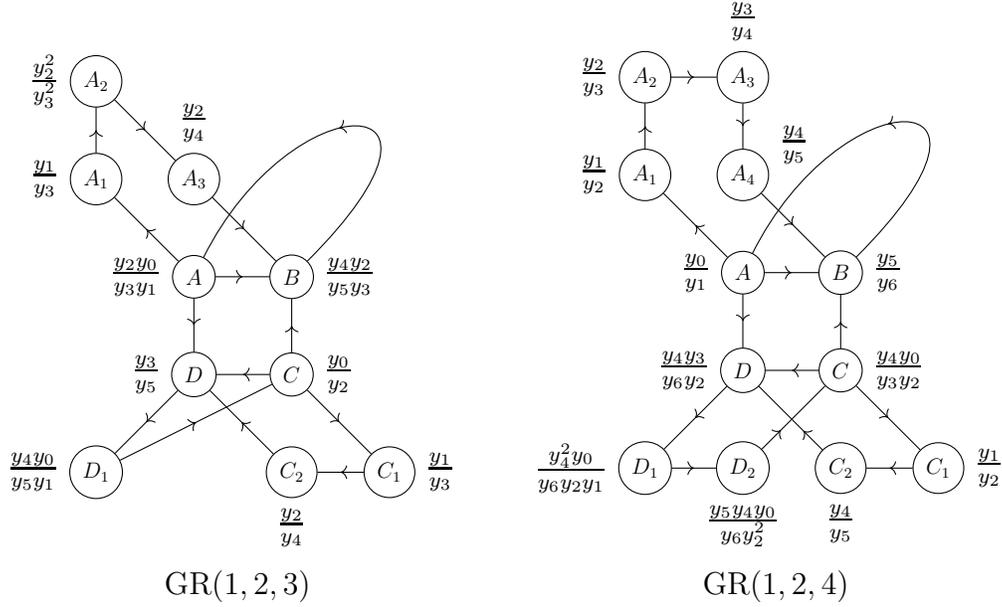
\begin{figure}
  \def\nodesc{0.7}
  \def\labelsc{1}
  \def\tikzscx{1.3}
  \def\tikzscy{1.3}
  \def\arrsc{0.6}
  \def\sclbx{1.0}
 \def\mi{-}
 \begin{tabular}{cc}
   \scalebox{\sclbx}{
    \begin{tikzpicture}[xscale=\tikzscx,yscale=\tikzscy]
    \node[scale=\nodesc,draw,circle] (A) at (0,1) {$A$};
    \node[scale=\nodesc,draw,circle] (B) at (1,1) {$B$};
    \node[scale=\nodesc,draw,circle] (C) at (1,0) {$C$};
    \node[scale=\nodesc,draw,circle] (D) at (0,0) {$D$};
    \node[scale=\nodesc,draw,circle] (A1) at (-1,2) {$A_1$};
    \node[scale=\nodesc,draw,circle] (A2) at (-1,3) {$A_2$};
    \node[scale=\nodesc,draw,circle] (A3) at (0,2) {$A_3$};
    \node[scale=\nodesc,draw,circle] (D1) at (-1,-1) {$D_1$};
    \node[scale=\nodesc,draw,circle] (C1) at (2,-1) {$C_1$};
    \node[scale=\nodesc,draw,circle] (C2) at (1,-1) {$C_2$};
    
    \node[anchor=east,scale=\labelsc] (Al) at (A.west) {$\frac{\y_2\y_0}{\y_3\y_1}$};
    \node[anchor=east,scale=\labelsc] (A1l) at (A1.west) {$\frac{\y_1}{\y_3}$};
    \node[anchor=east,scale=\labelsc] (A2l) at (A2.west) {$\frac{\y_2^2}{\y_3^2}$};
    \node[anchor=south,scale=\labelsc] (A3l) at (A3.north) {$\frac{\y_2}{\y_4}$};
    \node[anchor=west,scale=\labelsc] (Bl) at (B.east) {$\frac{\y_4\y_2}{\y_5\y_3}$};
    \node[anchor=west,scale=\labelsc] (Cl) at (C.east) {$\frac{\y_0}{\y_2}$};
    \node[anchor=west,scale=\labelsc] (C1l) at (C1.east) {$\frac{\y_1}{\y_3}$};
    \node[anchor=north,scale=\labelsc] (C2l) at (C2.south) {$\frac{\y_2}{\y_4}$};
    \node[anchor=east,scale=\labelsc] (D1l) at (D1.west) {$\frac{\y_4\y_0}{\y_5\y_1}$};
    \node[anchor=east,scale=\labelsc] (Dl) at (D.west) {$\frac{\y_3}{\y_5}$};
    \draw[postaction={decorate}] (A)--(B);
    \draw[postaction={decorate}] (A)--(D);
    \draw[postaction={decorate}] (C)--(B);
    \draw[postaction={decorate}] (C)--(D);
    \draw[postaction={decorate}] (A)--(A1);
    \draw[postaction={decorate}] (A1)--(A2);
    \draw[postaction={decorate}] (A2)--(A3);
    \draw[postaction={decorate}] (A3)--(B);
    \draw[postaction={decorate}] (D)--(D1);
    \draw[postaction={decorate}] (D1)--(C);
    \draw[postaction={decorate}] (C)--(C1);
    \draw[postaction={decorate}] (C1)--(C2);
    \draw[postaction={decorate}] (C2)--(D);
    \draw[postaction={decorate}] (B) ..controls (3,3) and (1,3) .. (A); 
  \end{tikzpicture}}
   &
     
   \scalebox{\sclbx}{
    \begin{tikzpicture}[xscale=\tikzscx,yscale=\tikzscy]
    \node[scale=\nodesc,draw,circle] (A) at (0,1) {$A$};
    \node[scale=\nodesc,draw,circle] (B) at (1,1) {$B$};
    \node[scale=\nodesc,draw,circle] (C) at (1,0) {$C$};
    \node[scale=\nodesc,draw,circle] (D) at (0,0) {$D$};
    \node[scale=\nodesc,draw,circle] (A1) at (-1,2) {$A_1$};
    \node[scale=\nodesc,draw,circle] (A2) at (-1,3) {$A_2$};
    \node[scale=\nodesc,draw,circle] (A3) at (0,3) {$A_3$};
    \node[scale=\nodesc,draw,circle] (A4) at (0,2) {$A_4$};
    \node[scale=\nodesc,draw,circle] (D1) at (-1,-1) {$D_1$};
    \node[scale=\nodesc,draw,circle] (D2) at (0,-1) {$D_2$};
    \node[scale=\nodesc,draw,circle] (C1) at (2,-1) {$C_1$};
    \node[scale=\nodesc,draw,circle] (C2) at (1,-1) {$C_2$};
    
    \node[anchor=east,scale=\labelsc] (Al) at (A.west) {$\frac{\y_0}{\y_1}$};
    \node[anchor=east,scale=\labelsc] (A1l) at (A1.west) {$\frac{\y_1}{\y_2}$};
    \node[anchor=east,scale=\labelsc] (A2l) at (A2.west) {$\frac{\y_2}{\y_3}$};
    \node[anchor=south,scale=\labelsc] (A3l) at (A3.north) {$\frac{\y_3}{\y_4}$};
    \node[anchor=south west,scale=\labelsc] (A4l) at (A4.0) {$\frac{\y_4}{\y_5}$};
    \node[anchor=west,scale=\labelsc] (Bl) at (B.east) {$\frac{\y_5}{\y_6}$};
    \node[anchor=west,scale=\labelsc] (Cl) at (C.east) {$\frac{\y_4\y_0}{\y_3\y_2}$};
    \node[anchor=west,scale=\labelsc] (C1l) at (C1.east) {$\frac{\y_1}{\y_2}$};
    \node[anchor=north,scale=\labelsc] (C2l) at (C2.south) {$\frac{\y_4}{\y_5}$};
    \node[anchor=east,scale=\labelsc] (D1l) at (D1.west) {$\frac{\y_4^2\y_0}{\y_6\y_2\y_1}$};
    \node[anchor=north,scale=\labelsc] (D2l) at (D2.south) {$\frac{\y_5\y_4\y_0}{\y_6\y_2^2}$};
    \node[anchor=east,scale=\labelsc] (Dl) at (D.west) {$\frac{\y_4\y_3}{\y_6\y_2}$};
    \draw[postaction={decorate}] (A)--(B);
    \draw[postaction={decorate}] (A)--(D);
    \draw[postaction={decorate}] (C)--(B);
    \draw[postaction={decorate}] (C)--(D);
    \draw[postaction={decorate}] (A)--(A1);
    \draw[postaction={decorate}] (A1)--(A2);
    \draw[postaction={decorate}] (A2)--(A3);
    \draw[postaction={decorate}] (A3)--(A4);
    \draw[postaction={decorate}] (A4)--(B);
    \draw[postaction={decorate}] (D)--(D1);
    \draw[postaction={decorate}] (D1)--(D2);
    \def\positn{0.3}
    \draw[postaction={decorate}] (D2)--(C);
    \draw[postaction={decorate}] (C2)--(D);
    \def\positn{0.5}
    \draw[postaction={decorate}] (C)--(C1);
    \draw[postaction={decorate}] (C1)--(C2);
    \draw[postaction={decorate}] (B) ..controls (3,3) and (1,3) .. (A); 
  \end{tikzpicture}}\\
   $\GR(1,2,3)$ & $\GR(1,2,4)$\\
  \end{tabular}
 
  \caption{\label{fig:GR_four_ex} The digraphs corresponding to the Somos-$6$ (left) and Somos-$7$ (right) sequences shown together with the values of $\X(0)$.}
\end{figure}
\end{example}

\def\spec{I}
\newcommand{\Pa}[1]{P_{#1}}
\newcommand{\Pb}[1]{P'_{#1}}
\newcommand{\Pc}[1]{P''_{#1}}
\newcommand{\Ya}[1]{Y_{#1}}
\newcommand{\Yb}[1]{Y'_{#1}}
\newcommand{\Yc}[1]{Y''_{#1}}
\newcommand{\va}[1]{v_{#1}}
\newcommand{\vb}[1]{v'_{#1}}
\newcommand{\vc}[1]{v''_{#1}}

\section{Subgraphs of the bidirected cycle}\label{sec:circles-with-doubled}
Consider a digraph $G=(V,E)$ with $V=[n]$, where the indices are taken modulo $n$. Fix some subset $\spec\subset [n]$. We let the set of edges be
$$E=\{(i,i+1)\mid i\in [n]\} \cup \{(i+1,i) \mid i \in \spec\}.$$
We call the indices $i \in \spec$ {\it {special}}. Let $n':=|\spec|$ be the total number of special vertices. If $n'=0$ then $G$ is a directed cycle which has been considered in Section~\ref{sec:directed-cycles}. If $n'=n$ then the $R$-system associated with $G$ is periodic as we will see in Theorem~\ref{prop:bidirected_periodic}.  We may thus assume $0 < n' < n$. After a possible cyclic shift, let us also assume that $n\notin \spec$.

\begin{figure}
  \def\scl{0.3}
  \def\textscl{1}
  \def\rad{2}
  \def\bnd{30}
  \def\bndd{60}
  \def\bnddd{10}
  \def\lw{1}
  \def\postn{0.7}
  \begin{tabular}{cc}

\begin{tikzpicture}[baseline=(1.base)]
  \foreach \i in {1,2,...,8}{
    \node[draw,scale=\scl,circle,fill=black] (\i) at ({260-45*\i}:\rad) { };
    \node[scale=\textscl,anchor={260-45*\i},inner sep=7pt] (l\i) at (\i.center) {$\i$};
  } 
  \draw[postaction={decorate},line width=\lw] (1) to[bend right=\bnd] (2);
  \draw[postaction={decorate},line width=\lw] (2) to[bend right=\bnd] (1);
  \draw[postaction={decorate},line width=\lw] (2) to[bend right=0] (3);
  \draw[postaction={decorate},line width=\lw] (3) to[bend right=\bnd] (4);
  \draw[postaction={decorate},line width=\lw] (4) to[bend right=\bnd] (3);
  \draw[postaction={decorate},line width=\lw] (4) to[bend right=\bnd] (5);
  \draw[postaction={decorate},line width=\lw] (5) to[bend right=\bnd] (4);
  \draw[postaction={decorate},line width=\lw] (5) to[bend right=\bnd] (6);
  \draw[postaction={decorate},line width=\lw] (6) to[bend right=\bnd] (5);
  \draw[postaction={decorate},line width=\lw] (6) to[bend right=0] (7);
  \draw[postaction={decorate},line width=\lw] (7) to[bend right=\bnd] (8);
  \draw[postaction={decorate},line width=\lw] (8) to[bend right=\bnd] (7);
  \draw[postaction={decorate},line width=\lw] (8) to[bend right=0] (1);

    \node[scale=\textscl,anchor={180+260-45*1}] (y1) at (1.center) {$\frac{\Ya 8}{\Yc 1}$};
    \node[scale=\textscl,anchor={180+260-45*2}] (y2) at (2.center) {$\frac{\Yb 1}{\Ya 2}$};
    \node[scale=\textscl,anchor={180+260-45*3}] (y3) at (3.center) {$\frac{\Ya 2}{\Yc 3}$};
    \node[scale=\textscl,anchor={180+260-45*4}] (y4) at (4.center) {$\frac{\Yb 3}{\Yc 4}$};
    \node[scale=\textscl,anchor={180+260-45*5}] (y5) at (5.center) {$\frac{\Yb 4}{\Yc 5}$};
    \node[scale=\textscl,anchor={180+260-45*6}] (y6) at (6.center) {$\frac{\Yb 5}{\Ya 6}$};
    \node[scale=\textscl,anchor={180+260-45*7}] (y7) at (7.center) {$\frac{\Ya 6}{\Yc 7}$};
    \node[scale=\textscl,anchor={180+260-45*8}] (y8) at (8.center) {$\frac{\Yb 7}{\Ya 8}$};
\end{tikzpicture}
    &
\begin{tikzpicture}[baseline=(1.base)]
  \foreach \i in {1,2,...,8}{
    \coordinate (\i) at ({260-45*\i}:\rad);
  } 
  \draw[dashed,line width=\lw] (1) to[bend right=\bnd]
  node[scale=\scl,pos=\postn,solid, draw=blue,circle,fill=blue] (Y1) {} (2);
  \draw[dashed,line width=\lw] (2) to[bend right=\bnd]
  node[scale=\scl,pos=\postn,solid, draw=blue,circle,fill=blue] (YY1) {} (1);
  \draw[dashed,line width=\lw] (2) to[bend right=0] 
  node[scale=\scl,pos=0.5,solid, draw=blue,circle,fill=blue] (Y2) {} (3);
  \draw[dashed,line width=\lw] (3) to[bend right=\bnd] 
  node[scale=\scl,pos=\postn,solid, draw=blue,circle,fill=blue] (Y3) {} (4);
  \draw[dashed,line width=\lw] (4) to[bend right=\bnd]
  node[scale=\scl,pos=\postn,solid, draw=blue,circle,fill=blue] (YY3) {} (3);
  \draw[dashed,line width=\lw] (4) to[bend right=\bnd]
  node[scale=\scl,pos=\postn,solid, draw=blue,circle,fill=blue] (Y4) {} (5);
  \draw[dashed,line width=\lw] (5) to[bend right=\bnd]
  node[scale=\scl,pos=\postn,solid, draw=blue,circle,fill=blue] (YY4) {} (4);
  \draw[dashed,line width=\lw] (5) to[bend right=\bnd]
  node[scale=\scl,pos=\postn,solid, draw=blue,circle,fill=blue] (Y5) {} (6);
  \draw[dashed,line width=\lw] (6) to[bend right=\bnd]
  node[scale=\scl,pos=\postn,solid, draw=blue,circle,fill=blue] (YY5) {} (5);
  \draw[dashed,line width=\lw] (6) to[bend right=0]
  node[scale=\scl,pos=0.5,solid, draw=blue,circle,fill=blue] (Y6) {} (7);
  \draw[dashed,line width=\lw] (7) to[bend right=\bnd] 
  node[scale=\scl,pos=\postn,solid, draw=blue,circle,fill=blue] (Y7) {} (8);
  \draw[dashed,line width=\lw] (8) to[bend right=\bnd] 
  node[scale=\scl,pos=\postn,solid, draw=blue,circle,fill=blue] (YY7) {} (7);
  \draw[dashed,line width=\lw] (8) to[bend right=0] 
  node[scale=\scl,pos=0.5,solid, draw=blue,circle,fill=blue] (Y8) {} (1);
  \draw[postaction={decorate},blue] (Y8)--(Y7);
  \draw[postaction={decorate},blue] (Y7)--(Y6);
  \draw[postaction={decorate},blue] (Y6)--(Y5);
  \draw[postaction={decorate},blue] (Y5)--(Y4);
  \draw[postaction={decorate},blue] (Y4)--(Y3);
  \draw[postaction={decorate},blue] (Y3)--(Y2);
  \draw[postaction={decorate},blue] (Y2)--(Y1);
  \draw[postaction={decorate},blue] (Y1)--(Y8);

  \draw[postaction={decorate},red] (YY1)--(Y1);
  \draw[postaction={decorate},red] (YY3)--(Y3);
  \draw[postaction={decorate},red] (YY4)--(Y4);
  \draw[postaction={decorate},red] (YY5)--(Y5);
  \draw[postaction={decorate},red] (YY7)--(Y7);
  
  \draw[postaction={decorate},blue] (Y8) to[bend right=\bndd] (YY7);
  \draw[postaction={decorate},blue] (YY7) to[bend right=\bndd] (Y6);
  \draw[postaction={decorate},blue] (Y6) to[bend right=\bndd] (YY5);
  \draw[postaction={decorate},blue] (YY5) to[bend right=\bndd] (YY4);
  \draw[postaction={decorate},blue] (YY4) to[bend right=\bndd] (YY3);
  \draw[postaction={decorate},blue] (YY3) to[bend right=\bndd] (Y2);
  \draw[postaction={decorate},blue] (Y2) to[bend right=\bndd] (YY1);
  \draw[postaction={decorate},blue] (YY1) to[bend right=\bndd] (Y8);

\draw[postaction={decorate},green!70!black] (Y1) to[bend right=\bnddd] (YY3);
\draw[postaction={decorate},green!70!black] (Y2) to[bend right=\bnddd] (YY4);
\draw[postaction={decorate},green!70!black] (Y3) to[bend right=\bnddd] (YY5);
\draw[postaction={decorate},green!70!black] (Y4) to[bend right=\bnddd] (Y6);
\draw[postaction={decorate},green!70!black] (Y5) to[bend right=\bnddd] (YY7);
\draw[postaction={decorate},green!70!black] (Y6) to[bend right=\bnddd] (Y8);
\draw[postaction={decorate},green!70!black] (Y7) to[bend right=\bnddd] (YY1);
\draw[postaction={decorate},green!70!black] (Y8) to[bend right=\bnddd] (Y2);

\def\Yscl{0.7}

  \foreach \i in {1,3,4,5,7}{
    \node[scale=\Yscl,anchor={260-45*\i}] (lY\i) at (Y\i.center) {$\vb \i$};
    \node[scale=\Yscl,anchor={180+260-45*\i}] (lYY\i) at (YY\i.center) {$\vc \i$};
  }
  
  \foreach \i in {2,6,8}{
    \node[scale=\Yscl,anchor={180+260-45*\i}] (lY\i) at (Y\i.center) {$\va \i$};
    }
\end{tikzpicture}
  \end{tabular}
  \caption{\label{fig:subgraph_bidirected_seed} A subgraph $G$ of the bidirected cycle for $n=8$ and $\spec=\{1,3,4,5,7\}$ (left), shown together with the assignments of the $Y$-variables from~\eqref{eq:subgraph_X}. The quiver $Q(G)$ (right).}
\end{figure}
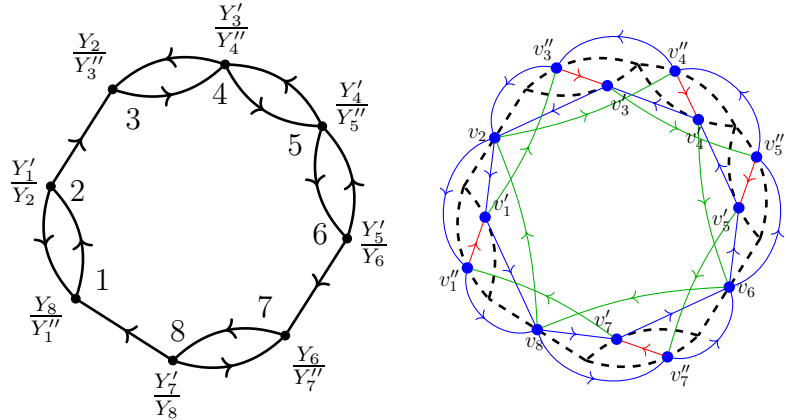 
\begin{example}\label{ex:subgraph_bidirected_cycle}
An example for $n=8$ and $\spec=\{1,3,4,5,7\}$ is shown in Figure~\ref{fig:subgraph_bidirected_seed} (left). 
\end{example}
We will build a Laurent recurrence system using a $T$-system arising from a cluster algebra. Then we will show that this $T$-system forms a weak $\tau$-sequence for the coefficient-free $R$-system associated with $G$.

Let us create a quiver $Q(G)=Q(n,\spec)$ as follows.  For each $i\in\spec$, introduce two vertices $\vb i$ and $\vc i$. For each $i \in [n]\setminus \spec$, introduce a single vertex $\va i=\vb i=\vc i$. Thus, the total number of vertices is $n+n'$, and $\vb i, \vc i$ are defined for each $i\in [n]$ and $\va i$ is defined for each $i\in [n]\setminus \spec$.

The arrows of $Q(G)$ are:
\begin{enumerate}[(a)]
\item\label{edge:a} $\vb i \rightarrow \vb{i-1}$ for all $i\in [n]$;
\item\label{edge:b} $\vc i \rightarrow \vc{i-1}$ for all $i\in [n]$;
\item\label{edge:c} $\vb i \rightarrow \vc{i+2}$ for all $i\in [n]$;
\item\label{edge:d} $\vc i \rightarrow \vb i$ for all $i\in \spec$.
\end{enumerate}

Note that for $n\leq4$, it may happen that $Q(G)$ contains arrows $u\to v$ and $v\to u$ for some vertices $u$ and $v$, in which case such pairs of arrows should be removed.

\begin{example}
For $G$ from Example~\ref{ex:subgraph_bidirected_cycle}, the quiver $Q(G)$ is shown in Figure~\ref{fig:subgraph_bidirected_seed} (right). The edges~\eqref{edge:a} and~\eqref{edge:b} are shown in blue, the edges~\eqref{edge:c} are shown in green, and the edges~\eqref{edge:d} are shown in red.
\end{example}

\begin{figure}
 \def\scl{0.3}
  \def\textscl{1}
  \def\rad{2}
  \def\bnd{30}
  \def\bndd{60}
  \def\bnddd{10}
  \def\lw{1}
  \def\postn{0.7}
  \scalebox{0.95}{
  \begin{tabular}{|c|c|c|c|}\hline
$G$ &

 \begin{tikzpicture}[baseline=(0.base)]
\coordinate (0) at (0,0);
  \foreach \i in {1,2,3}{
    \node[draw,scale=\scl,circle,fill=black] (\i) at ({260-120*\i}:\rad) { };
    \node[scale=\textscl,anchor={180+260-120*\i}] (l\i) at (\i.center) {$\i$};
  } 
  \draw[postaction={decorate},line width=\lw] (1) to[bend right=\bnd] (2);
  \draw[postaction={decorate},line width=\lw] (2) to[bend right=\bnd] (1);
  \draw[postaction={decorate},line width=\lw] (2) to[bend right=0] (3);
  \draw[postaction={decorate},line width=\lw] (3) to[bend right=0] (1);
\end{tikzpicture}
    &

 \begin{tikzpicture}[baseline=(0.base)]
\coordinate (0) at (0,0);
  \foreach \i in {1,2,3,4}{
    \node[draw,scale=\scl,circle,fill=black] (\i) at ({225-90*\i}:\rad) { };
    \node[scale=\textscl,anchor={180+225-90*\i}] (l\i) at (\i.center) {$\i$};
  } 
  \draw[postaction={decorate},line width=\lw] (1) to[bend right=\bnd] (2);
  \draw[postaction={decorate},line width=\lw] (2) to[bend right=\bnd] (1);
  \draw[postaction={decorate},line width=\lw] (2) to[bend right=0] (3);
  \draw[postaction={decorate},line width=\lw] (3) to[bend right=0] (4);
  \draw[postaction={decorate},line width=\lw] (4) to[bend right=0] (1);
\end{tikzpicture}
      &

 \begin{tikzpicture}[baseline=(0.base)]
\coordinate (0) at (0,0);
  \foreach \i in {1,2,3,4}{
    \node[draw,scale=\scl,circle,fill=black] (\i) at ({225-90*\i}:\rad) { };
    \node[scale=\textscl,anchor={180+225-90*\i}] (l\i) at (\i.center) {$\i$};
  } 
  \draw[postaction={decorate},line width=\lw] (1) to[bend right=\bnd] (2);
  \draw[postaction={decorate},line width=\lw] (2) to[bend right=\bnd] (1);
  \draw[postaction={decorate},line width=\lw] (2) to[bend right=0] (3);
  \draw[postaction={decorate},line width=\lw] (3) to[bend right=\bnd] (4);
  \draw[postaction={decorate},line width=\lw] (4) to[bend right=\bnd] (3);
  \draw[postaction={decorate},line width=\lw] (4) to[bend right=0] (1);
\end{tikzpicture}

   \\
   $Q(G)$&   
    
 \begin{tikzpicture}[baseline=(0.base)]
\coordinate (0) at (0,0);

  \foreach \i in {1,2,3}{
    \coordinate (\i) at ({260-120*\i}:\rad);
  } 
  \draw[dashed,line width=\lw] (1) to[bend right=\bnd] 
  node[scale=\scl,pos=\postn,solid, draw=blue,circle,fill=blue] (Y1) {}(2);
  \draw[dashed,line width=\lw] (2) to[bend right=\bnd] 
  node[scale=\scl,pos=\postn,solid, draw=blue,circle,fill=blue] (YY1) {}(1);
  \draw[dashed,line width=\lw] (2) to[bend right=0] 
  node[scale=\scl,pos=\postn,solid, draw=blue,circle,fill=blue] (Y2) {}(3);
  \draw[dashed,line width=\lw] (3) to[bend right=0] 
  node[scale=\scl,pos=\postn,solid, draw=blue,circle,fill=blue] (Y3) {}(1);

  \draw[postaction={decorate},blue] (Y3)--(Y2);
  \draw[postaction={decorate},blue] (Y2)--(Y1);
  \draw[postaction={decorate},blue] (Y1)--(Y3);
  
  \draw[postaction={decorate},red] (YY1)--(Y1);
  
\def\positn{0.7}
\draw[postaction={decorate},blue] (Y2) to[bend right=60, looseness=2] (YY1);
\def\positn{0.5}
  \draw[postaction={decorate},blue] (YY1) to[bend right=0] (Y3);
  \draw[postaction={decorate},blue] (Y3) to[bend right=20] (Y2);

\draw[postaction={decorate},green!70!black] (Y1) to[bend right=\bnddd] (Y3);
\draw[postaction={decorate},green!70!black] (Y3) to[bend right=\bnddd] (Y2);
\draw[postaction={decorate},green!70!black] (Y2) to[bend right=30] (YY1);
\end{tikzpicture}
    &

 \begin{tikzpicture}[baseline=(0.base)]
\coordinate (0) at (0,0);

  \foreach \i in {1,2,3,4}{
    \coordinate (\i) at ({225-90*\i}:\rad);
  } 
  \draw[dashed,line width=\lw] (1) to[bend right=\bnd] 
  node[scale=\scl,pos=\postn,solid, draw=blue,circle,fill=blue] (Y1) {}(2);
  \draw[dashed,line width=\lw] (2) to[bend right=\bnd] 
  node[scale=\scl,pos=\postn,solid, draw=blue,circle,fill=blue] (YY1) {}(1);
  \draw[dashed,line width=\lw] (2) to[bend right=0] 
  node[scale=\scl,pos=\postn,solid, draw=blue,circle,fill=blue] (Y2) {}(3);
  \draw[dashed,line width=\lw] (3) to[bend right=0] 
  node[scale=\scl,pos=\postn,solid, draw=blue,circle,fill=blue] (Y3) {}(4);
  \draw[dashed,line width=\lw] (4) to[bend right=0] 
  node[scale=\scl,pos=\postn,solid, draw=blue,circle,fill=blue] (Y4) {}(1);

  \draw[postaction={decorate},blue] (Y3)--(Y2);
  \draw[postaction={decorate},blue] (Y2)--(Y1);
  \draw[postaction={decorate},blue] (Y1)--(Y4);
  \draw[postaction={decorate},blue] (Y4)--(Y3);
  
  \draw[postaction={decorate},red] (YY1)--(Y1);

\draw[postaction={decorate},blue] (Y2) to[bend right=\bndd] (YY1);
  \draw[postaction={decorate},blue] (YY1) to[bend right=0] (Y4);
  \draw[postaction={decorate},blue] (Y4) to[bend right=20] (Y3);
  \draw[postaction={decorate},blue] (Y3) to[bend right=20] (Y2);

\draw[postaction={decorate},green!70!black] (Y1) to[bend right=0] (Y3);
\draw[postaction={decorate},green!70!black] (Y3) to[bend right=0] (YY1);
\end{tikzpicture}
 &   
 \begin{tikzpicture}[baseline=(0.base)]
\coordinate (0) at (0,0);
  \foreach \i in {1,2,3,4}{
    \coordinate (\i) at ({225-90*\i}:\rad);
  } 
  \draw[dashed,line width=\lw] (1) to[bend right=\bnd] 
  node[scale=\scl,pos=\postn,solid, draw=blue,circle,fill=blue] (Y1) {}(2);
  \draw[dashed,line width=\lw] (2) to[bend right=\bnd] 
  node[scale=\scl,pos=\postn,solid, draw=blue,circle,fill=blue] (YY1) {}(1);
  \draw[dashed,line width=\lw] (2) to[bend right=0] 
  node[scale=\scl,pos=\postn,solid, draw=blue,circle,fill=blue] (Y2) {}(3);
  \draw[dashed,line width=\lw] (3) to[bend right=\bnd] 
  node[scale=\scl,pos=\postn,solid, draw=blue,circle,fill=blue] (Y3) {}(4);
  \draw[dashed,line width=\lw] (4) to[bend right=\bnd] 
  node[scale=\scl,pos=\postn,solid, draw=blue,circle,fill=blue] (YY3) {}(3);
  \draw[dashed,line width=\lw] (4) to[bend right=0] 
  node[scale=\scl,pos=\postn,solid, draw=blue,circle,fill=blue] (Y4) {}(1);

  \draw[postaction={decorate},blue] (Y3)--(Y2);
  \draw[postaction={decorate},blue] (Y2)--(Y1);
  \draw[postaction={decorate},blue] (Y1)--(Y4);
  \draw[postaction={decorate},blue] (Y4)--(Y3);
  
  \draw[postaction={decorate},red] (YY1)--(Y1);
  \draw[postaction={decorate},red] (YY3)--(Y3);
  
\def\positn{0.7}
\draw[postaction={decorate},blue] (Y2) to[bend right=\bndd] (YY1);
  \draw[postaction={decorate},blue] (YY1) to[bend right=0] (Y4);
  \draw[postaction={decorate},blue] (Y4) to[bend right=\bndd] (YY3);
  \draw[postaction={decorate},blue] (YY3) to[bend right=0] (Y2);
\def\positn{0.5}
  
\draw[postaction={decorate},green!70!black] (Y1) to[bend right=0] (YY3);
\draw[postaction={decorate},green!70!black] (Y3) to[bend right=0] (YY1);
\end{tikzpicture} \\
    &&&  \\\hline
    $(n,I)$ & $(3,\{1\})$ & $(4,\{1\})$ & $(4,\{1,3\})$\\\hline
\begin{tabular}{c}
  Name \\
  of $Q(G)$
\end{tabular} & Somos-$4$ & Somos-$5$ & $dP3$\\\hline
                                    
  \end{tabular}
  }
  \caption{\label{fig:subgraph_somos_dp3} Small examples of subgraphs of the bidirected cycle that give rise to some well known quivers. Here $dP3$ stands for the \emph{del Pezzo $3$ quiver} studied in the physics literature, see, e.g.,~\cite{FHHU,Musiker}. Note that the pairs of green arrows connecting two vertices in the opposite directions are removed.}
\end{figure}

Let $\spec=\{i_1<i_2<\dots<i_k\}$, thus $i_k<n$. Let us also introduce a total order $\prec$ on the vertices of $Q(G)$ defined by
\[\vc 1 \preceq \vb 1\preceq \vc 2\preceq \vb 2\preceq\dots\preceq \vc{n}\preceq\vb{n},\]
where we have the equality $\vc i=\vb i$ if and only if $i\notin\spec$. Let the vertices of $Q(G)$ be $u_1\prec u_2\prec\dots\prec u_{n+n'}$. Introduce the map $\nu$ on the vertices of $Q(G)$ given by $\nu(u_i)=u_{i+1}$ for all $i\in [n+n']$, where the indices are taken modulo $n+n'$, i.e., $\nu(u_{n+n'})=u_1$. The following result is an easy direct verification.

\def\QQ{{\tilde{Q}}}
\begin{proposition}\label{prop:mutations_subgraphs_bidirected_cycle}
Let $\QQ$ be the quiver obtained from $Q(G)$ by mutating at the vertices $\vb{i_k},\vb{i_{k-1}},\dots,\vb{i_1}$ in this order. Then $\nu$ is an isomorphism between $Q(G)$ and $\QQ$.
\end{proposition}

\begin{example}
  For $G$ from Example~\ref{ex:subgraph_bidirected_cycle}, we have
  \[(u_1,u_2,\dots,u_{13})=(\vc1,\vb1,\va2,\vc3,\vb3,\vc4,\vb4,\vc5,\vb5,\va6,\vc7,\vb7,\va8).\]
  Thus $\nu$ is a cyclic shift of the sequence on the right hand side. One checks that mutating $Q(G)$ at the vertices $\vb7,\vb5,\vb4,\vb3,\vb1$ produces a quiver $\QQ$ obtained from $Q(G)$ by applying this cyclic shift.
\end{example}

Proposition~\ref{prop:mutations_subgraphs_bidirected_cycle} allows us to consider the $T$-system associated with $Q(G)$. Let us denote its values by $Y(t):=(Y_v(t))$, where $t\geq 0$ is an integer and $v$ is a vertex of $Q(G)$. Thus the seed $(Q(G),Y(t+1))$ is obtained from $(Q(G),Y(t))$ by first mutating at $\vb{i_k},\vb{i_{k-1}},\dots,\vb{i_1}$ and then applying the map $\nu$. For $v=\va i,\vb i, \vc i$, we denote $Y_v(t)$ by $\Ya i(t),\Yb i(t), \Yc i(t)$ respectively. Similarly to the proof of Proposition~\ref{prop:mutations_subgraphs_bidirected_cycle}, we obtain the following recurrence relation for $Y(t+1)$.

\begin{proposition}\label{prop:subgraph_Y_recurrence}
For $i\notin\spec$, we have
\[\Ya i(t+1)=\Yb i(t+1)=\Yc i(t+1)=\Yc{i+1}(t).\]
For $i\in \spec$, we still have  $\Yb i(t+1)=\Yc{i+1}(t)$, but the non-trivial relation is
\begin{equation}\label{eq:subgraph_cycle_recurrence}
  \begin{split}
\Yc i(t+1)=\frac{\Yc i(t)\Yc{i+1}(t)+\Yb{i-1}(t)\Yc{i+1}(t+1)}{\Yb i(t)}.
  \end{split}
\end{equation}
\end{proposition}
This is indeed a recurrence relation: since $n\notin\spec$, it allows one to compute $\Ya n(t+1)=\Yb n(t+1)=\Yc n(t+1)=\Yc 1(t)$, and then to compute $\Ya i(t+1),\Yb i(t+1),\Yc i(t+1)$ for each $i=n-1,n-2,\dots,1$. Using this procedure, for each vertex $v$ of $Q(G)$, $Y_v(t+1)$ becomes a rational function $P_v$ in the variables $Y_u(t)$, where $u$ runs over the vertices of $Q(G)$. We let $Y(t)=(Y_{u_1}(t),\dots,Y_{u_{n+n'}(t)})$ and $P=(P_{u_1},\dots,P_{u_{n+n'}})$. By the Laurent property of cluster algebras (Theorem~\ref{thm:FZ_Laurent}), we get that $(Y(t),P)$ is a Laurent recurrence system.

\begin{example}
  For $G$ from Example~\ref{ex:subgraph_bidirected_cycle}, we have
  \begin{equation*}
    \begin{split}
      \Yc 5(t+1) =& \frac{\Yc 5(t) \Ya 6(t) + \Yb 4(t) \Yc 7(t)}{\Yb 5(t)},\\
      \Yc 4(t+1) =& \frac{\Yc 4(t)\Yc 5(t) + \Yb 3(t)\Yc 5(t+1)}{\Yb 4(t)}\\
                 =&\frac{\Yc 4(t)\Yc 5(t)\Yb 5(t)+\Yb 3(t)\Yc 5(t)\Ya 6(t)+\Yb 3(t)\Yb 4(t)\Yc 7(t)}{\Yb 4(t)\Yb 5(t)},\\
      \Yc 3(t+1) =& \frac{\Yc 3(t)\Yc 4(t) + \Ya 2(t)\Yc 4(t+1)}{\Yb 3(t)}\\
                 =&\frac{\Yc 3(t)\Yc 4(t)\Yb 4(t)\Yb 5(t)+\Ya 2(t)\Yc 4(t)\Yc 5(t)\Yb 5(t)}{\Yb 3(t)\Yb 4(t)\Yb 5(t)}+\\
                  &\frac{\Ya 2(t)\Yb 3(t)\Yc 5(t)\Ya 6(t) + \Ya 2(t)\Yb 3(t)\Yb 4(t)\Yc 7(t)}{\Yb 3(t)\Yb 4(t)\Yb 5(t)}.
    \end{split}
  \end{equation*}
\end{example}

Now let us give a reduction to the $R$-system of $G$. For each $i \in [n]$, let 
\begin{equation}\label{eq:subgraph_X}
X_i(t) = \frac{\Yb{i-1}(t)}{\Yc i(t)}.
\end{equation}
The values of $X_i(t)$ are shown in Figure~\ref{fig:subgraph_bidirected_seed} (left).

\begin{theorem}
 The $X_i(t)$-s satisfy the toggle relations~\eqref{eq:toggle} of the $R$-system associated with $G$. Thus, the Laurent recurrence system $(Y(t), P)$ we constructed is a weak $\tau$-sequence for the $R$-system associated with $G$.  
\end{theorem}
\begin{proof}
 There are four cases to consider, depending on whether $i$ is special, and also whether $i-1$ is special.
  Assume $i-1$ is special but $i$ is not. Then the toggling identity reads 
$$\frac{\Yb {i-1}(t)}{\Yc i(t)} \frac{\Yb {i-1}(t+1)}{\Yc i(t+1)} = \frac{ \frac{\Yb {i-2}(t)}{\Yc {i-1}(t)}+\frac{\Yb {i}(t)}{\Yc {i+1}(t)} }{  \left(\frac{\Yb {i-2}(t+1)}{\Yc {i-1}(t+1)}\right)^{-1} }.$$
By Proposition~\ref{prop:subgraph_Y_recurrence}, we have  $\Yb {i-1}(t+1)=\Yc i(t)=\Ya i(t)$,  $\Yb {i-2}(t+1)=\Yc {i-1}(t)$. Using $\Ya {i}(t+1)= \Yb {i}(t+1)=\Yc i(t+1) = \Yc {i+1}(t)$, we see that this is equivalent to 
 $$\Yb {i-1}(t) \Yc {i-1}(t+1) = \Yb i(t) \Yc {i-1}(t) + \Yb {i-2}(t) \Yc{i+1}(t),$$
 which follows from~\eqref{eq:subgraph_cycle_recurrence} because $\Yc {i+1}(t) = \Yc i(t+1)$ for $i\notin \spec$.
 
 Assume $i-1,i\in\spec$, i.e., both $i$ and $i-1$ are special. Then the toggling identity reads
 \begin{equation}\label{eq:subgraph_toggling_both_special}
\frac{\Yb {i-1}(t)}{\Yc i(t)} \frac{\Yb {i-1}(t+1)}{\Yc i(t+1)} = \frac{ \frac{\Yb {i-2}(t)}{\Yc {i-1}(t)}+\frac{\Yb {i}(t)}{\Yc {i+1}(t)} }{  \left(\frac{\Yb {i-2}(t+1)}{\Yc {i-1}(t+1)}\right)^{-1} + \left(\frac{\Yb {i}(t+1)}{\Yc {i+1}(t+1)}\right)^{-1}  }.
 \end{equation}

  By Proposition~\ref{prop:subgraph_Y_recurrence},  we have $\Yb {i-1}(t+1)=\Yc i(t)$,  $\Yb {i-2}(t+1)=\Yc {i-1}(t)$, and $\Yb {i}(t+1) = \Yc {i+1}(t)$. Thus the above identity is equivalent to
   \begin{equation*}
    \begin{split}
      &\Yb{i-1}(t)\left(\Yc{i+1}(t+1)\Yc{i-1}(t)+\Yc{i-1}(t+1)\Yc{i+1}(t)\right)=\\
      &=\Yc i(t+1)\left(\Yb i(t)\Yc{i-1}(t)+\Yb{i-2}(t)\Yc{i+1}(t)\right).
    \end{split}
  \end{equation*}
      
  The left hand side contains the term $\Yb{i-1}(t)\Yc{i-1}(t+1)\Yc{i+1}(t)$ while the right hand side contains the term $\Yc i(t+1)\Yb i(t)\Yc{i-1}(t)$. Expanding both of these terms using~\eqref{eq:subgraph_cycle_recurrence} yields
  \begin{equation*}
    \begin{split}
      &\Yb {i-1}(t)  \Yc{i+1}(t+1)\Yc {i-1}(t) + \Yc {i+1}(t) (  \Yc {i-1}(t) \Yc {i}(t) + \Yb {i-2}(t) \Yc {i}(t+1)  ) =\\
       &=\Yc {i-1}(t) (  \Yc i(t) \Yc {i+1}(t) + \Yb {i-1}(t) \Yc {i+1}(t+1)  ) + \Yc{i}(t+1)\Yb {i-2}(t) \Yc {i+1}(t) ,
    \end{split}
  \end{equation*}
 which is true by term to term comparison.  
 
 Assume $i$ is special but $i-1$ is not. Then the toggling identity reads
$$\frac{\Yb {i-1}(t)}{\Yc i(t)} \frac{\Yb {i-1}(t+1)}{\Yc i(t+1)} = \frac{ \frac{\Yb {i}(t)}{\Yc {i+1}(t)} }{  \left(\frac{\Yb {i-2}(t+1)}{\Yc {i-1}(t+1)}\right)^{-1} + \left(\frac{\Yb {i}(t+1)}{\Yc {i+1}(t+1)}\right)^{-1}  }.$$
By Proposition~\ref{prop:subgraph_Y_recurrence},  we have  $\Yb {i-1}(t+1)=\Yc i(t)$, $\Yb{i-2}(t+1)=\Yc{i-1}(t)$, and $\Yb i(t+1)=\Yc{i+1}(t)$. Using this and the fact that $\Ya{i-1}(t)=\Yb{i-1}(t)=\Yc{i-1}(t)$, the above toggling identity becomes equivalent to
\[\Yb i(t)\Yc i (t+1)=\Yc{i-1}(t+1)\Yc{i+1}(t)+\Yc{i+1}(t+1)\Yc{i-1}(t).\]
This follows from~\eqref{eq:subgraph_cycle_recurrence} since $\Yc{i-1}(t)=\Yb{i-1}(t)$ and $\Yc{i-1}(t+1)=\Yb{i-1}(t+1)=\Yc{i}(t)$.
 
Finally, assume neither $i$ nor $i-1$ is special. Then the toggling identity reads
 $$\frac{\Yb {i-1}(t)}{\Yc i(t)} \frac{\Yb {i-1}(t+1)}{\Yc i(t+1)} = \frac{ \frac{\Yb {i}(t)}{\Yc {i+1}(t)} }{  \left(\frac{\Yb {i-2}(t+1)}{\Yc {i-1}(t+1)}\right)^{-1} }.$$
We see that this is true since by Proposition~\ref{prop:subgraph_Y_recurrence},  we have  $\Yb {i-1}(t+1)=\Yc i(t)$, $\Yb {i-2}(t+1) = \Yc {i-1}(t) = \Yb {i-1}(t)$, $\Yc i(t+1)= \Yb i(t+1) = \Yc {i+1}(t)$, and $\Yc {i-1}(t+1) = \Yb {i-1}(t+1) = \Yc i(t) = \Yb i(t)$.
\end{proof}

\begin{theorem}
  If both $i$ and $i-1$ are special, the expression
  $$\frac{\Yb {i-2}(t) \Yc {i+1}(t) + \Yb i(t) \Yc {i-1}(t)}{\Yb {i-1}(t) \Yc i(t)}$$
  is a conserved quantity of the Laurent recurrence system $(Y(t),P)$.
\end{theorem}

\begin{proof}
 The identity 
 $$\frac{\Yb {i-2}(t+1) \Yc {i+1}(t+1) + \Yb i(t+1) \Yc {i-1}(t+1)}{\Yb {i-1}(t+1) \Yc i(t+1)} = \frac{\Yb {i-2}(t) \Yc {i+1}(t) + \Yb i(t) \Yc {i-1}(t)}{\Yb {i-1}(t) \Yc i(t)}$$
 is just another way of writing~\eqref{eq:subgraph_toggling_both_special}. 
\end{proof}

\begin{conjecture}\label{conj:subgraph_strong_tau}
The Laurent recurrence system $(Y(t), P)$ we constructed is a strong $\tau$-sequence for the $R$-system.  
\end{conjecture}

Conjecture~\ref{conj:subgraph_strong_tau} includes that $Y(t)$ consists of irreducible Laurent polynomials for any $t$. It is easy to see that the other needed condition from Definition~\ref{dfn:tau} holds. 

\begin{proposition}
 The map $(\Yb i(0), \Yc i(0))_{i\in [n]} \mapsto (X_i(0))_{i\in[n]}$ given by $X_i(0) = \frac{\Yb {i-1}(0)}{\Yc i(0)}$ is dominant. 
\end{proposition}

\begin{proof}
 Take the Zariski dense subset where the variables $\Yc i(0)$ do not vanish. Fix arbitrary values of $\Yc i(0)$ for $i\in\spec$. Start with some special $i$ and let $\Yb {i-1}(0) = X_i(0) \Yc i(0)$. Then let $\Yb {i-2}(0) = X_{i-1} (0)\Yc {i-1}(0)$, etc., until we complete the cycle to get $\Yb {i}(0) = X_{i+1}(0) \Yc {i+1}(0)$.
 It is clear that we define each of the $n+n'$ variables exactly once this way, and that $X_j(0) = \frac{\Yb {j-1}(0)}{\Yc j(0)}$ for all $j\in[n]$. Since we did not make any assumptions on the $X_j(0)$-s other than them being non-zero, the claim follows. 
\end{proof}

\section{Cylindric posets}\label{sec:cylindr-posets-octag}

Let $n\geq 3$, $m \geq 2$. Consider a digraph $G_{n,m}$ with vertices $A_{i,j}$, $1 \leq i \leq n$, $1 \leq j \leq m+2$, and one extra vertex $\tb$. The edges are as follows, for all $i\in[n]$ (where the index $i$ is always treated modulo $n$):
\begin{itemize}
 \item $(A_{i,j}, A_{i,j+1})$, for $1 \leq j \leq m+1$;
 \item $(A_{i,j}, A_{i-1,j+1})$, for $1 \leq j \leq m+1$;
 \item $(A_{i,m+2}, \tb)$;
 \item $(\tb, A_{i,1})$.
\end{itemize} 
If we omit the vertex $\tb$, we can think of the remaining graph as a Hasse diagram of a poset naturally embeddable on a cylinder.

\begin{figure}
  
\def\rad{1.2}
\def\scl{0.3}
\def\stp{1.2}
\def\textscl{0.8}

\begin{tabular}{c|c}

\newcommand{\nodeA}[3]{\node[scale=\scl,draw,circle,fill=black,label=#3:{\scalebox{\textscl}{$A_{#1,#2}$}}] (A#1#2) at ({\rad*cos(#1*120+#2*60-60)},{\rad*sin(#1*120+#2*60-60)},{#2*\stp-\stp}) { };}

\tdplotsetmaincoords{80}{20}
\begin{tikzpicture}[tdplot_main_coords]
  \foreach \j in {0,...,4}{
\fill[blue,opacity=0.1] (0,0,{\j*\stp}) circle (\rad);
}
\foreach \i/\j in {2/1,1/2,1/3,3/4,3/5,2/2,1/4} {
       \nodeA{\i}{\j}{left}
    }
 
\foreach \i/\j in {3/1,3/2,2/3,2/4,1/5,1/1,3/3,2/5} {
       \nodeA{\i}{\j}{right}
    }   
  \def\positn{0.7}
\draw[postaction={decorate},dashed]         (A11) -- (A32);
\draw[postaction={decorate},dashed]         (A11) -- (A12);
\draw[postaction={decorate}]         (A31) -- (A22);
\draw[postaction={decorate}]         (A31) -- (A32);
\draw[postaction={decorate}]         (A21) -- (A22);
\draw[postaction={decorate}]         (A21) -- (A12);

\draw[postaction={decorate},dashed]         (A12) -- (A33);
\draw[postaction={decorate}]         (A12) -- (A13);
\draw[postaction={decorate}]         (A32) -- (A23);
\draw[postaction={decorate},dashed]         (A32) -- (A33);
\draw[postaction={decorate}]         (A22) -- (A23);
\draw[postaction={decorate}]         (A22) -- (A13);

\draw[postaction={decorate}]         (A13) -- (A34);
\draw[postaction={decorate}]         (A13) -- (A14);
\draw[postaction={decorate},dashed]         (A33) -- (A24);
\draw[postaction={decorate},dashed]         (A33) -- (A34);
\draw[postaction={decorate}]         (A23) -- (A24);
\draw[postaction={decorate}]         (A23) -- (A14);

\draw[postaction={decorate}]         (A14) -- (A35);
\draw[postaction={decorate}]         (A14) -- (A15);
\draw[postaction={decorate},dashed]         (A34) -- (A25);
\draw[postaction={decorate}]         (A34) -- (A35);
\draw[postaction={decorate},dashed]         (A24) -- (A25);
\draw[postaction={decorate}]         (A24) -- (A15);
\end{tikzpicture}
  &
    \begin{tikzpicture}[scale=0.6]
\def\textscl{0.8}
\node[scale=0.3,draw,circle,fill=black,label=left:{\scalebox{\textscl}{$A_{1,1}$}}] (A11) at (1.50,2.00) { };
\node[scale=0.3,draw,circle,fill=black,label=left:{\scalebox{\textscl}{$A_{2,1}$}}] (A31) at (4.50,2.00) { };
\node[scale=0.3,draw,circle,fill=black,label=left:{\scalebox{\textscl}{$A_{3,1}$}}] (A51) at (7.50,2.00) { };
\node[scale=0.3,draw,circle,fill=black,label=left:{\scalebox{\textscl}{$A_{1,1}$}}] (A71) at (10.50,2.00) { };
\node[scale=0.3,draw,circle,fill=black,label=left:{\scalebox{\textscl}{$A_{2,1}$}}] (A91) at (13.50,2.00) { };
\node[scale=\textscl] (A02) at (0.00,4.00) {$\dots$};
\node[scale=0.3,draw,circle,fill=black,label=left:{\scalebox{\textscl}{$A_{1,2}$}}] (A22) at (3.00,4.00) { };
\node[scale=0.3,draw,circle,fill=black,label=left:{\scalebox{\textscl}{$A_{2,2}$}}] (A42) at (6.00,4.00) { };
\node[scale=0.3,draw,circle,fill=black,label=left:{\scalebox{\textscl}{$A_{3,2}$}}] (A62) at (9.00,4.00) { };
\node[scale=0.3,draw,circle,fill=black,label=left:{\scalebox{\textscl}{$A_{1,2}$}}] (A82) at (12.00,4.00) { };
\node[scale=\textscl] (A102) at (15.00,4.00) {$\dots$};
\node[scale=0.3,draw,circle,fill=black,label=left:{\scalebox{\textscl}{$A_{3,3}$}}] (A13) at (1.50,6.00) { };
\node[scale=0.3,draw,circle,fill=black,label=left:{\scalebox{\textscl}{$A_{1,3}$}}] (A33) at (4.50,6.00) { };
\node[scale=0.3,draw,circle,fill=black,label=left:{\scalebox{\textscl}{$A_{2,3}$}}] (A53) at (7.50,6.00) { };
\node[scale=0.3,draw,circle,fill=black,label=left:{\scalebox{\textscl}{$A_{3,3}$}}] (A73) at (10.50,6.00) { };
\node[scale=0.3,draw,circle,fill=black,label=left:{\scalebox{\textscl}{$A_{1,3}$}}] (A93) at (13.50,6.00) { };
\node[scale=\textscl] (A04) at (0.00,8.00) {$\dots$};
\node[scale=0.3,draw,circle,fill=black,label=left:{\scalebox{\textscl}{$A_{3,4}$}}] (A24) at (3.00,8.00) { };
\node[scale=0.3,draw,circle,fill=black,label=left:{\scalebox{\textscl}{$A_{1,4}$}}] (A44) at (6.00,8.00) { };
\node[scale=0.3,draw,circle,fill=black,label=left:{\scalebox{\textscl}{$A_{2,4}$}}] (A64) at (9.00,8.00) { };
\node[scale=0.3,draw,circle,fill=black,label=left:{\scalebox{\textscl}{$A_{3,4}$}}] (A84) at (12.00,8.00) { };
\node[scale=\textscl] (A104) at (15.00,8.00) {$\dots$};
\node[scale=0.3,draw,circle,fill=black,label=left:{\scalebox{\textscl}{$A_{2,5}$}}] (A15) at (1.50,10.00) { };
\node[scale=0.3,draw,circle,fill=black,label=left:{\scalebox{\textscl}{$A_{3,5}$}}] (A35) at (4.50,10.00) { };
\node[scale=0.3,draw,circle,fill=black,label=left:{\scalebox{\textscl}{$A_{1,5}$}}] (A55) at (7.50,10.00) { };
\node[scale=0.3,draw,circle,fill=black,label=left:{\scalebox{\textscl}{$A_{2,5}$}}] (A75) at (10.50,10.00) { };
\node[scale=0.3,draw,circle,fill=black,label=left:{\scalebox{\textscl}{$A_{3,5}$}}] (A95) at (13.50,10.00) { };
\draw[postaction={decorate}] (A02) -- (A13);
\draw[postaction={decorate}] (A04) -- (A15);
\draw[postaction={decorate}] (A11) -- (A22);
\draw[postaction={decorate}] (A11) -- (A02);
\draw[postaction={decorate}] (A13) -- (A24);
\draw[postaction={decorate}] (A13) -- (A04);
\draw[postaction={decorate}] (A22) -- (A33);
\draw[postaction={decorate}] (A22) -- (A13);
\draw[postaction={decorate}] (A24) -- (A35);
\draw[postaction={decorate}] (A24) -- (A15);
\draw[postaction={decorate}] (A31) -- (A42);
\draw[postaction={decorate}] (A31) -- (A22);
\draw[postaction={decorate}] (A33) -- (A44);
\draw[postaction={decorate}] (A33) -- (A24);
\draw[postaction={decorate}] (A42) -- (A53);
\draw[postaction={decorate}] (A42) -- (A33);
\draw[postaction={decorate}] (A44) -- (A55);
\draw[postaction={decorate}] (A44) -- (A35);
\draw[postaction={decorate}] (A51) -- (A62);
\draw[postaction={decorate}] (A51) -- (A42);
\draw[postaction={decorate}] (A53) -- (A64);
\draw[postaction={decorate}] (A53) -- (A44);
\draw[postaction={decorate}] (A62) -- (A73);
\draw[postaction={decorate}] (A62) -- (A53);
\draw[postaction={decorate}] (A64) -- (A75);
\draw[postaction={decorate}] (A64) -- (A55);
\draw[postaction={decorate}] (A71) -- (A82);
\draw[postaction={decorate}] (A71) -- (A62);
\draw[postaction={decorate}] (A73) -- (A84);
\draw[postaction={decorate}] (A73) -- (A64);
\draw[postaction={decorate}] (A82) -- (A93);
\draw[postaction={decorate}] (A82) -- (A73);
\draw[postaction={decorate}] (A84) -- (A95);
\draw[postaction={decorate}] (A84) -- (A75);
\draw[postaction={decorate}] (A91) -- (A102);
\draw[postaction={decorate}] (A91) -- (A82);
\draw[postaction={decorate}] (A93) -- (A104);
\draw[postaction={decorate}] (A93) -- (A84);
\draw[postaction={decorate}] (A102) -- (A93);
\draw[postaction={decorate}] (A104) -- (A95);
\end{tikzpicture}
\end{tabular}

    \caption{The digraph $G_{n,m}$ for $n=3$, $m=3$, with vertex $\tb$ omitted (left). The same digraph drawn periodically on the universal cover of the cylinder (right).}
    \label{fig:cylindric_G_3_3}
\end{figure}
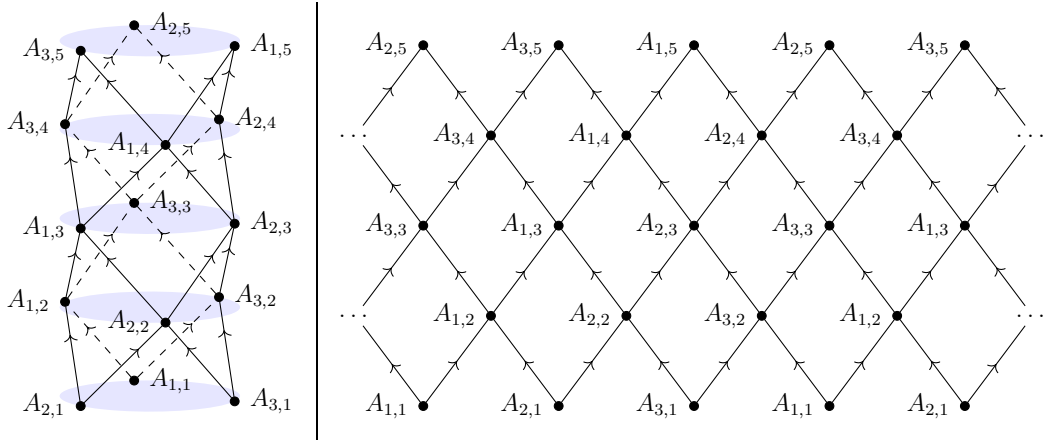

\begin{example}
 In Figure \ref{fig:cylindric_G_3_3} we see an example of $G_{n,m}$ for $n=3$, $m=3$. The right hand side of the figure shows the same poset but drawn in a \emph{strip} (i.e., on the universal cover of the cylinder) in a periodic way. 
\end{example}

We are going to create an initial seed of a Laurent Phenomenon algebra associated with $G_{n,m}$, and use it to build a $\tau$-sequence. To a large extent, the LP algebra we construct resembles a cluster algebra. We will point out explicitly the details where the differences with cluster algebras come into play. 

Construct a seed with variables $\alpha_{i,j}(t)$, $\beta_{i}(t)$, $\beta_{i}(t+1)$, $\gamma_{i}(t)$, and $\gamma_{i}(t+1)$, with the following exchange polynomials for $i\in[n]$ and $2\leq j\leq m-1$.
\begin{equation*}
  \begin{split}
F_{\alpha_{i,1}(t)}&= \alpha_{i-1,2}(t) \beta_{i+1}(t) \beta_{i+1}(t+1) + \alpha_{i,2}(t) \beta_{i}(t) \beta_{i}(t+1),\\
F_{\alpha_{i,j}(t)}&= \alpha_{i-1,j}(t) \alpha_{i,j+1}(t) \alpha_{i+1,j-1}(t) + \alpha_{i-1,j+1}(t) \alpha_{i+1,j}(t) \alpha_{i,j-1}(t)\\
F_{\alpha_{i,m}(t)}&= \alpha_{i+1,m-1}(t) \gamma_{i}(t) \gamma_{i}(t+1) \alpha_{i-1,m}(t) + \alpha_{i,m-1}(t) \gamma_{i-1}(t) \gamma_{i-1}(t+1) \alpha_{i+1,m}(t),\\
F_{\beta_{i}(t)} &= F_{\beta_{i}(t+1)} = \alpha_{i,1}(t) \beta_{i-1}(t) + \alpha_{i-1,1}(t) \beta_{i+1}(t),\\
F_{\gamma_{i}(t)} &= F_{\gamma_{i}(t+1)} = \alpha_{i,m}(t) \gamma_{i+1}(t) + \alpha_{i+1,m}(t) \gamma_{i-1}(t).\\
  \end{split}
\end{equation*}
It is convenient to think of the variables being associated to the faces of the graph $G_{n,m}$, as shown in Figure \ref{fig:cylindric_Q_3_3}. Then the exchange polynomials are essentially encoded by the quiver $Q_{n,m}$ in this figure. 
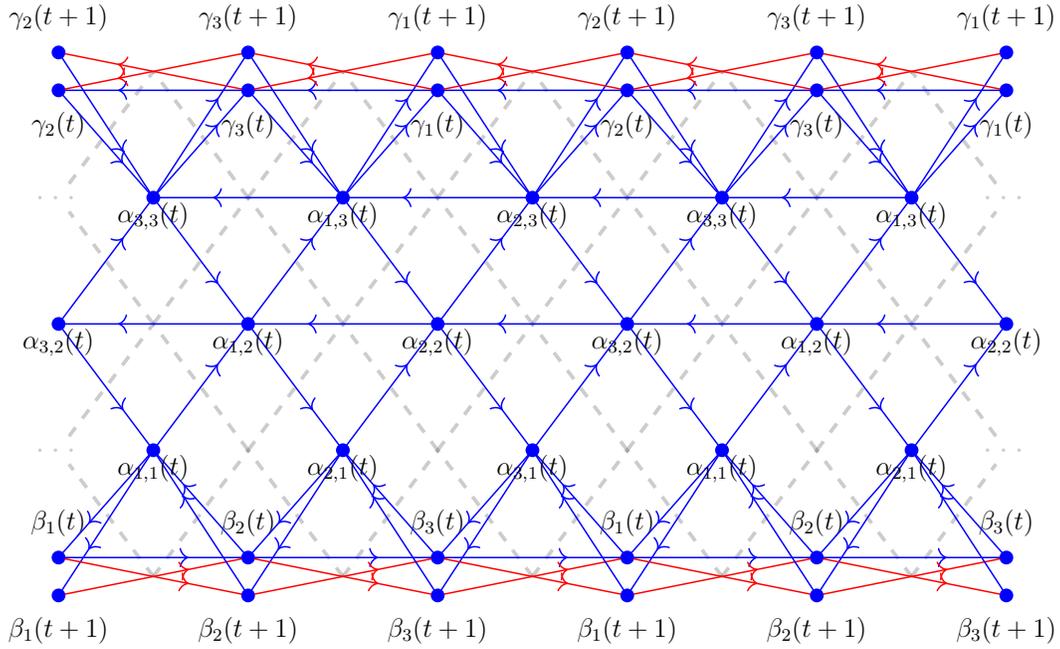
\begin{figure}
\scalebox{1.4}{
\begin{tikzpicture}[scale=0.6]
\def\textscl{0.6}
\def\positn{0.7}
\node[scale=0.3,draw=blue,circle,fill=blue,label=above:{\scalebox{\textscl}{$\beta_{1}(t)$}}] (A01) at (0.00,2.30) { };
\node[scale=0.3,draw=blue,circle,fill=blue,label=below:{\scalebox{\textscl}{$\beta_{1}(t+1)$}}] (AA01) at (0.00,1.70) { };
\coordinate (D11) at (1.50,2.00);
\node[scale=0.3,draw=blue,circle,fill=blue,label=above:{\scalebox{\textscl}{$\beta_{2}(t)$}}] (A21) at (3.00,2.30) { };
\node[scale=0.3,draw=blue,circle,fill=blue,label=below:{\scalebox{\textscl}{$\beta_{2}(t+1)$}}] (AA21) at (3.00,1.70) { };
\coordinate (D31) at (4.50,2.00);
\node[scale=0.3,draw=blue,circle,fill=blue,label=above:{\scalebox{\textscl}{$\beta_{3}(t)$}}] (A41) at (6.00,2.30) { };
\node[scale=0.3,draw=blue,circle,fill=blue,label=below:{\scalebox{\textscl}{$\beta_{3}(t+1)$}}] (AA41) at (6.00,1.70) { };
\coordinate (D51) at (7.50,2.00);
\node[scale=0.3,draw=blue,circle,fill=blue,label=above:{\scalebox{\textscl}{$\beta_{1}(t)$}}] (A61) at (9.00,2.30) { };
\node[scale=0.3,draw=blue,circle,fill=blue,label=below:{\scalebox{\textscl}{$\beta_{1}(t+1)$}}] (AA61) at (9.00,1.70) { };
\coordinate (D71) at (10.50,2.00);
\node[scale=0.3,draw=blue,circle,fill=blue,label=above:{\scalebox{\textscl}{$\beta_{2}(t)$}}] (A81) at (12.00,2.30) { };
\node[scale=0.3,draw=blue,circle,fill=blue,label=below:{\scalebox{\textscl}{$\beta_{2}(t+1)$}}] (AA81) at (12.00,1.70) { };
\coordinate (D91) at (13.50,2.00);
\node[scale=0.3,draw=blue,circle,fill=blue,label=above:{\scalebox{\textscl}{$\beta_{3}(t)$}}] (A101) at (15.00,2.30) { };
\node[scale=0.3,draw=blue,circle,fill=blue,label=below:{\scalebox{\textscl}{$\beta_{3}(t+1)$}}] (AA101) at (15.00,1.70) { };
\node[scale=0.8,opacity=0.2] (D02) at (0.00,4.00) {$\dots$};
\node[scale=0.3,draw=blue,circle,fill=blue] (A12) at (1.50,4.00) { };
\coordinate (D22) at (3.00,4.00);
\node[scale=0.3,draw=blue,circle,fill=blue] (A32) at (4.50,4.00) { };
\coordinate (D42) at (6.00,4.00);
\node[scale=0.3,draw=blue,circle,fill=blue] (A52) at (7.50,4.00) { };
\coordinate (D62) at (9.00,4.00);
\node[scale=0.3,draw=blue,circle,fill=blue] (A72) at (10.50,4.00) { };
\coordinate (D82) at (12.00,4.00);
\node[scale=0.3,draw=blue,circle,fill=blue] (A92) at (13.50,4.00) { };
\node[scale=0.8,opacity=0.2] (D102) at (15.00,4.00) {$\dots$};
\node[scale=0.3,draw=blue,circle,fill=blue] (A03) at (0.00,6.00) { };
\coordinate (D13) at (1.50,6.00);
\node[scale=0.3,draw=blue,circle,fill=blue] (A23) at (3.00,6.00) { };
\coordinate (D33) at (4.50,6.00);
\node[scale=0.3,draw=blue,circle,fill=blue] (A43) at (6.00,6.00) { };
\coordinate (D53) at (7.50,6.00);
\node[scale=0.3,draw=blue,circle,fill=blue] (A63) at (9.00,6.00) { };
\coordinate (D73) at (10.50,6.00);
\node[scale=0.3,draw=blue,circle,fill=blue] (A83) at (12.00,6.00) { };
\coordinate (D93) at (13.50,6.00);
\node[scale=0.3,draw=blue,circle,fill=blue] (A103) at (15.00,6.00) { };
\node[scale=0.8,opacity=0.2] (D04) at (0.00,8.00) {$\dots$};
\node[scale=0.3,draw=blue,circle,fill=blue] (A14) at (1.50,8.00) { };
\coordinate (D24) at (3.00,8.00);
\node[scale=0.3,draw=blue,circle,fill=blue] (A34) at (4.50,8.00) { };
\coordinate (D44) at (6.00,8.00);
\node[scale=0.3,draw=blue,circle,fill=blue] (A54) at (7.50,8.00) { };
\coordinate (D64) at (9.00,8.00);
\node[scale=0.3,draw=blue,circle,fill=blue] (A74) at (10.50,8.00) { };
\coordinate (D84) at (12.00,8.00);
\node[scale=0.3,draw=blue,circle,fill=blue] (A94) at (13.50,8.00) { };
\node[scale=0.8,opacity=0.2] (D104) at (15.00,8.00) {$\dots$};
\node[scale=0.3,draw=blue,circle,fill=blue,label=above:{\scalebox{\textscl}{$\gamma_{2}(t+1)$}}] (AA05) at (0.00,10.30) { };
\node[scale=0.3,draw=blue,circle,fill=blue,label=below:{\scalebox{\textscl}{$\gamma_{2}(t)$}}] (A05) at (0.00,9.70) { };
\coordinate (D15) at (1.50,10.00);
\node[scale=0.3,draw=blue,circle,fill=blue,label=above:{\scalebox{\textscl}{$\gamma_{3}(t+1)$}}] (AA25) at (3.00,10.30) { };
\node[scale=0.3,draw=blue,circle,fill=blue,label=below:{\scalebox{\textscl}{$\gamma_{3}(t)$}}] (A25) at (3.00,9.70) { };
\coordinate (D35) at (4.50,10.00);
\node[scale=0.3,draw=blue,circle,fill=blue,label=above:{\scalebox{\textscl}{$\gamma_{1}(t+1)$}}] (AA45) at (6.00,10.30) { };
\node[scale=0.3,draw=blue,circle,fill=blue,label=below:{\scalebox{\textscl}{$\gamma_{1}(t)$}}] (A45) at (6.00,9.70) { };
\coordinate (D55) at (7.50,10.00);
\node[scale=0.3,draw=blue,circle,fill=blue,label=above:{\scalebox{\textscl}{$\gamma_{2}(t+1)$}}] (AA65) at (9.00,10.30) { };
\node[scale=0.3,draw=blue,circle,fill=blue,label=below:{\scalebox{\textscl}{$\gamma_{2}(t)$}}] (A65) at (9.00,9.70) { };
\coordinate (D75) at (10.50,10.00);
\node[scale=0.3,draw=blue,circle,fill=blue,label=above:{\scalebox{\textscl}{$\gamma_{3}(t+1)$}}] (AA85) at (12.00,10.30) { };
\node[scale=0.3,draw=blue,circle,fill=blue,label=below:{\scalebox{\textscl}{$\gamma_{3}(t)$}}] (A85) at (12.00,9.70) { };
\coordinate (D95) at (13.50,10.00);
\node[scale=0.3,draw=blue,circle,fill=blue,label=above:{\scalebox{\textscl}{$\gamma_{1}(t+1)$}}] (AA105) at (15.00,10.30) { };
\node[scale=0.3,draw=blue,circle,fill=blue,label=below:{\scalebox{\textscl}{$\gamma_{1}(t)$}}] (A105) at (15.00,9.70) { };
\draw[dashed,line width=1.0,opacity=0.2] (D02) -- (D13);
\draw[dashed,line width=1.0,opacity=0.2] (D04) -- (D15);
\draw[dashed,line width=1.0,opacity=0.2] (D11) -- (D22);
\draw[dashed,line width=1.0,opacity=0.2] (D11) -- (D02);
\draw[dashed,line width=1.0,opacity=0.2] (D13) -- (D24);
\draw[dashed,line width=1.0,opacity=0.2] (D13) -- (D04);
\draw[dashed,line width=1.0,opacity=0.2] (D22) -- (D33);
\draw[dashed,line width=1.0,opacity=0.2] (D22) -- (D13);
\draw[dashed,line width=1.0,opacity=0.2] (D24) -- (D35);
\draw[dashed,line width=1.0,opacity=0.2] (D24) -- (D15);
\draw[dashed,line width=1.0,opacity=0.2] (D31) -- (D42);
\draw[dashed,line width=1.0,opacity=0.2] (D31) -- (D22);
\draw[dashed,line width=1.0,opacity=0.2] (D33) -- (D44);
\draw[dashed,line width=1.0,opacity=0.2] (D33) -- (D24);
\draw[dashed,line width=1.0,opacity=0.2] (D42) -- (D53);
\draw[dashed,line width=1.0,opacity=0.2] (D42) -- (D33);
\draw[dashed,line width=1.0,opacity=0.2] (D44) -- (D55);
\draw[dashed,line width=1.0,opacity=0.2] (D44) -- (D35);
\draw[dashed,line width=1.0,opacity=0.2] (D51) -- (D62);
\draw[dashed,line width=1.0,opacity=0.2] (D51) -- (D42);
\draw[dashed,line width=1.0,opacity=0.2] (D53) -- (D64);
\draw[dashed,line width=1.0,opacity=0.2] (D53) -- (D44);
\draw[dashed,line width=1.0,opacity=0.2] (D62) -- (D73);
\draw[dashed,line width=1.0,opacity=0.2] (D62) -- (D53);
\draw[dashed,line width=1.0,opacity=0.2] (D64) -- (D75);
\draw[dashed,line width=1.0,opacity=0.2] (D64) -- (D55);
\draw[dashed,line width=1.0,opacity=0.2] (D71) -- (D82);
\draw[dashed,line width=1.0,opacity=0.2] (D71) -- (D62);
\draw[dashed,line width=1.0,opacity=0.2] (D73) -- (D84);
\draw[dashed,line width=1.0,opacity=0.2] (D73) -- (D64);
\draw[dashed,line width=1.0,opacity=0.2] (D82) -- (D93);
\draw[dashed,line width=1.0,opacity=0.2] (D82) -- (D73);
\draw[dashed,line width=1.0,opacity=0.2] (D84) -- (D95);
\draw[dashed,line width=1.0,opacity=0.2] (D84) -- (D75);
\draw[dashed,line width=1.0,opacity=0.2] (D91) -- (D102);
\draw[dashed,line width=1.0,opacity=0.2] (D91) -- (D82);
\draw[dashed,line width=1.0,opacity=0.2] (D93) -- (D104);
\draw[dashed,line width=1.0,opacity=0.2] (D93) -- (D84);
\draw[dashed,line width=1.0,opacity=0.2] (D102) -- (D93);
\draw[dashed,line width=1.0,opacity=0.2] (D104) -- (D95);
\draw[postaction={decorate},blue] (A03) -- (A14);
\draw[postaction={decorate},blue] (A03) -- (A12);
\draw[postaction={decorate},blue] (A05) -- (A14);
\draw[postaction={decorate},blue] (AA05) -- (A14);
\draw[postaction={decorate},blue] (A12) -- (A23);
\draw[postaction={decorate},blue] (A14) -- (A25);
\draw[postaction={decorate},blue] (A14) -- (AA25);
\draw[postaction={decorate},blue] (A14) -- (A23);
\draw[postaction={decorate},blue] (A23) -- (A34);
\draw[postaction={decorate},blue] (A23) -- (A32);
\draw[postaction={decorate},blue] (A23) -- (A03);
\draw[postaction={decorate},blue] (A25) -- (A34);
\draw[postaction={decorate},blue] (AA25) -- (A34);
\draw[postaction={decorate},blue] (A25) -- (A05);
\draw[postaction={decorate},blue] (A32) -- (A43);
\draw[postaction={decorate},blue] (A34) -- (A45);
\draw[postaction={decorate},blue] (A34) -- (AA45);
\draw[postaction={decorate},blue] (A34) -- (A43);
\draw[postaction={decorate},blue] (A34) -- (A14);
\draw[postaction={decorate},blue] (A43) -- (A54);
\draw[postaction={decorate},blue] (A43) -- (A52);
\draw[postaction={decorate},blue] (A43) -- (A23);
\draw[postaction={decorate},blue] (A45) -- (A54);
\draw[postaction={decorate},blue] (AA45) -- (A54);
\draw[postaction={decorate},blue] (A45) -- (A25);
\draw[postaction={decorate},blue] (A52) -- (A63);
\draw[postaction={decorate},blue] (A54) -- (A65);
\draw[postaction={decorate},blue] (A54) -- (AA65);
\draw[postaction={decorate},blue] (A54) -- (A63);
\draw[postaction={decorate},blue] (A54) -- (A34);
\draw[postaction={decorate},blue] (A63) -- (A74);
\draw[postaction={decorate},blue] (A63) -- (A72);
\draw[postaction={decorate},blue] (A63) -- (A43);
\draw[postaction={decorate},blue] (A65) -- (A74);
\draw[postaction={decorate},blue] (AA65) -- (A74);
\draw[postaction={decorate},blue] (A65) -- (A45);
\draw[postaction={decorate},blue] (A72) -- (A83);
\draw[postaction={decorate},blue] (A74) -- (A85);
\draw[postaction={decorate},blue] (A74) -- (AA85);
\draw[postaction={decorate},blue] (A74) -- (A83);
\draw[postaction={decorate},blue] (A74) -- (A54);
\draw[postaction={decorate},blue] (A83) -- (A94);
\draw[postaction={decorate},blue] (A83) -- (A92);
\draw[postaction={decorate},blue] (A83) -- (A63);
\draw[postaction={decorate},blue] (A85) -- (A94);
\draw[postaction={decorate},blue] (AA85) -- (A94);
\draw[postaction={decorate},blue] (A85) -- (A65);
\draw[postaction={decorate},blue] (A92) -- (A103);
\draw[postaction={decorate},blue] (A94) -- (A105);
\draw[postaction={decorate},blue] (A94) -- (AA105);
\draw[postaction={decorate},blue] (A94) -- (A103);
\draw[postaction={decorate},blue] (A94) -- (A74);
\draw[postaction={decorate},blue] (A103) -- (A83);
\draw[postaction={decorate},blue] (A105) -- (A85);
\draw[postaction={decorate},blue] (A12) -- (A01);
\draw[postaction={decorate},blue] (A12) -- (AA01);
\draw[postaction={decorate},blue] (A01) -- (A21);
\draw[postaction={decorate},red] (A01) -- (AA21);
\draw[postaction={decorate},red] (AA01) -- (A21);
\draw[postaction={decorate},blue] (A32) -- (A21);
\draw[postaction={decorate},blue] (A32) -- (AA21);
\draw[postaction={decorate},blue] (A21) -- (A12);
\draw[postaction={decorate},blue] (AA21) -- (A12);
\draw[postaction={decorate},blue] (A21) -- (A41);
\draw[postaction={decorate},red] (A21) -- (AA41);
\draw[postaction={decorate},red] (AA21) -- (A41);
\draw[postaction={decorate},blue] (A52) -- (A41);
\draw[postaction={decorate},blue] (A52) -- (AA41);
\draw[postaction={decorate},blue] (A41) -- (A32);
\draw[postaction={decorate},blue] (AA41) -- (A32);
\draw[postaction={decorate},blue] (A41) -- (A61);
\draw[postaction={decorate},red] (A41) -- (AA61);
\draw[postaction={decorate},red] (AA41) -- (A61);
\draw[postaction={decorate},blue] (A72) -- (A61);
\draw[postaction={decorate},blue] (A72) -- (AA61);
\draw[postaction={decorate},blue] (A61) -- (A52);
\draw[postaction={decorate},blue] (AA61) -- (A52);
\draw[postaction={decorate},blue] (A61) -- (A81);
\draw[postaction={decorate},red] (A61) -- (AA81);
\draw[postaction={decorate},red] (AA61) -- (A81);
\draw[postaction={decorate},blue] (A92) -- (A81);
\draw[postaction={decorate},blue] (A92) -- (AA81);
\draw[postaction={decorate},blue] (A81) -- (A72);
\draw[postaction={decorate},blue] (AA81) -- (A72);
\draw[postaction={decorate},blue] (A81) -- (A101);
\draw[postaction={decorate},red] (A81) -- (AA101);
\draw[postaction={decorate},red] (AA81) -- (A101);
\draw[postaction={decorate},blue] (A101) -- (A92);
\draw[postaction={decorate},blue] (AA101) -- (A92);
\draw[postaction={decorate},red] (A25) -- (AA05);
\draw[postaction={decorate},red] (AA25) -- (A05);
\draw[postaction={decorate},red] (A45) -- (AA25);
\draw[postaction={decorate},red] (AA45) -- (A25);
\draw[postaction={decorate},red] (A65) -- (AA45);
\draw[postaction={decorate},red] (AA65) -- (A45);
\draw[postaction={decorate},red] (A85) -- (AA65);
\draw[postaction={decorate},red] (AA85) -- (A65);
\draw[postaction={decorate},red] (A105) -- (AA85);
\draw[postaction={decorate},red] (AA105) -- (A85);
\node[scale=\textscl,inner sep=2pt,anchor=north] (AAA) at (A03.center) {$\alpha_{3,2}(t)$};
\node[scale=\textscl,inner sep=2pt,anchor=north] (AAA) at (A12.center) {$\alpha_{1,1}(t)$};
\node[scale=\textscl,inner sep=2pt,anchor=north] (AAA) at (A14.center) {$\alpha_{3,3}(t)$};
\node[scale=\textscl,inner sep=2pt,anchor=north] (AAA) at (A23.center) {$\alpha_{1,2}(t)$};
\node[scale=\textscl,inner sep=2pt,anchor=north] (AAA) at (A32.center) {$\alpha_{2,1}(t)$};
\node[scale=\textscl,inner sep=2pt,anchor=north] (AAA) at (A34.center) {$\alpha_{1,3}(t)$};
\node[scale=\textscl,inner sep=2pt,anchor=north] (AAA) at (A43.center) {$\alpha_{2,2}(t)$};
\node[scale=\textscl,inner sep=2pt,anchor=north] (AAA) at (A52.center) {$\alpha_{3,1}(t)$};
\node[scale=\textscl,inner sep=2pt,anchor=north] (AAA) at (A54.center) {$\alpha_{2,3}(t)$};
\node[scale=\textscl,inner sep=2pt,anchor=north] (AAA) at (A63.center) {$\alpha_{3,2}(t)$};
\node[scale=\textscl,inner sep=2pt,anchor=north] (AAA) at (A72.center) {$\alpha_{1,1}(t)$};
\node[scale=\textscl,inner sep=2pt,anchor=north] (AAA) at (A74.center) {$\alpha_{3,3}(t)$};
\node[scale=\textscl,inner sep=2pt,anchor=north] (AAA) at (A83.center) {$\alpha_{1,2}(t)$};
\node[scale=\textscl,inner sep=2pt,anchor=north] (AAA) at (A92.center) {$\alpha_{2,1}(t)$};
\node[scale=\textscl,inner sep=2pt,anchor=north] (AAA) at (A94.center) {$\alpha_{1,3}(t)$};
\node[scale=\textscl,inner sep=2pt,anchor=north] (AAA) at (A103.center) {$\alpha_{2,2}(t)$};
\end{tikzpicture}}
    \caption{The quiver $Q_{n,m}$ for $n=3$, $m=3$.}
    \label{fig:cylindric_Q_3_3}
\end{figure}
The only place to be careful with are the arrows connecting $\beta$-s and $\gamma$-s: those arrows may contribute to exchange polynomials of only one of their ends, not both. For example, the arrow connecting $\beta_2(t+1)$ to $\beta_3(t)$ contributes to the exchange polynomial
$\beta_1(t) \alpha_{2,1}(t) + \alpha_{1,1}(t) \beta_3(t)$ of $\beta_2(t+1)$, but not to the exchange polynomial $\beta_2(t) \alpha_{3,1}(t) + \alpha_{2,1}(t) \beta_1(t)$ of $\beta_3(t)$. 

To obtain the seed at time $t+1$, we perform the following mutations (in that exact order). 
\begin{itemize}
 \item Mutate all the $\alpha_{i,1}(t)$-s. 
 \item Mutate all the $\alpha_{i,2}(t)$-s. 
 \item $\ldots$
 \item Mutate all the $\alpha_{i,m}(t)$-s.
 \item Mutate all the $\gamma_{i}(t)$-s. 
 \item Mutate all the $\beta_{i}(t)$-s.
\end{itemize}
Note that $\alpha_{i,j}(t)$ mutates into $\alpha_{i,j}(t+1)$, however $\beta_{i}(t)$ mutates into $\beta_{i}(t+2)$, and $\gamma_{i}(t)$ mutates into $\gamma_{i}(t+2)$. Note also that according to \cite[Theorem 2.4]{LP2}, mutations at the variables not appearing in each others' exchange polynomials commute. Thus, the following theorem implies that the way we define the sequence of mutations indeed does not depend on the exact order of mutations within each step, as implicitly suggested above. 

\begin{theorem} \label{thm:cylpoly}
  At the moment when each variable is mutated, its exchange polynomial is given by
  \begin{equation*}
    \begin{split}
 F_{\alpha_{i,1}(t)} &= \alpha_{i-1,2}(t) \beta_{i+1}(t) \beta_{i+1}(t+1) + \alpha_{i,2}(t) \beta_{i}(t) \beta_{i}(t+1),\\
 F_{\alpha_{i,j}(t)} &= \alpha_{i,j+1}(t) \alpha_{i+1,j-1}(t+1) + \alpha_{i-1,j+1}(t) \alpha_{i,j-1}(t+1),\quad \text { for } 1 < j < m,\\
 F_{\alpha_{i,m}(t)} &= \alpha_{i+1,m-1}(t+1) \gamma_{i-1}(t) \gamma_{i-1}(t+1) + \alpha_{i,m-1}(t+1) \gamma_{i}(t) \gamma_{i}(t+1),\\
 F_{\gamma_{i}(t)} &= \alpha_{i,m}(t+1) \gamma_{i+1}(t+1) + \alpha_{i+1,m}(t+1) \gamma_{i-1}(t+1),\\
 F_{\beta_{i}(t)} &= \alpha_{i,1}(t+1) \beta_{i-1}(t+1) + \alpha_{i-1,1}(t+1) \beta_{i+1}(t+1).
    \end{split}
  \end{equation*}
 The seed at time $t+1$ that one gets after applying the above mutations can be obtained from the seed at time $t$ via a substitution $t \mapsto t+1$. In other words, after performing the mutation sequence, the exchange polynomials retain their shape.  
\end{theorem}
Thus, this sequence of mutations is an analog of a $T$-system, except inside an LP algebra. 

Before we give some details of the proof, note the following. 
At the time of the mutation, the variables $\gamma_{i}(t), \gamma_{i}(t+1)$, and also the variables $\beta_{i}(t), \beta_{i}(t+1)$ have the same exchange polynomials. Because of that, the corresponding mutations are given by~\eqref{eq:LP_mutation_2} and therefore are not a special case of a cluster algebra:
 $$\gamma_{i}(t) \gamma_{i}(t+1) \gamma_{i}(t+2) = \alpha_{i,m}(t+1) \gamma_{i+1}(t+1) + \alpha_{i+1,m}(t+1) \gamma_{i-1}(t+1),$$
 $$\beta_{i}(t) \beta_{i}(t+1) \beta_{i}(t+2) = \alpha_{i,1}(t+1) \beta_{i-1}(t+1) + \alpha_{i-1,1}(t+1) \beta_{i+1}(t+1).$$
The rest of the variables mutate as they would in a cluster algebra with the same exchange polynomials.

\begin{proof}[Proof of Theorem~\ref{thm:cylpoly}]
 The mutations of $\alpha$-s happen as in a cluster algebra, and in fact the usual quiver mutations can be used to verify what happens to their exchange polynomials. The only non-trivial part of the process is the part involving $\beta$-s and $\gamma$-s. 
 We describe how mutations of $\alpha$-s impact the exchange polynomials of $\beta$-s, and vice versa. A similar computation works for $\gamma$-s. 
 
 At the time when $\alpha_{i,1}(t)$-s are mutated, the exchange polynomials of $\beta_i(t)$, $\alpha_{i,1}(t)$, and $\alpha_{i-1,1}(t)$ are as follows:
 \begin{equation*}
   \begin{split}
F_{\beta_i(t)} &= \alpha_{i,1}(t) \beta_{i-1}(t) + \alpha_{i-1,1}(t) \beta_{i+1}(t),\\
F_{\alpha_{i,1}(t)} &= \alpha_{i-1,2}(t) \beta_{i+1}(t) \beta_{i+1}(t+1) + \alpha_{i,2}(t) \beta_{i}(t) \beta_{i}(t+1),\\
F_{\alpha_{i-1,1}(t)} &= \alpha_{i-2,2}(t) \beta_{i}(t) \beta_{i}(t+1) + \alpha_{i-1,2}(t) \beta_{i-1}(t) \beta_{i-1}(t+1).\\
   \end{split}
 \end{equation*}
 Let us mutate $\alpha_{i,1}(t)$ first. Setting $\beta_i(t)=0$ in $F_{\alpha_{i,1}(t)}$, we obtain the following substitution:
\[
   F_{\beta_i(t)}\mid_{ \alpha_{i,1}(t) \leftarrow \frac{\alpha_{i-1,2}(t) \beta_{i+1}(t) \beta_{i+1}(t+1)}{\alpha_{i,1}(t+1)}     } =\frac{\left(\begin{aligned}
       &\beta_{i-1}(t) \alpha_{i-1,2}(t) \beta_{i+1}(t) \beta_{i+1}(t+1) +\\
       &+\alpha_{i,1}(t+1) \alpha_{i-1,1}(t) \beta_{i+1}(t)     
     \end{aligned}\right)}{\alpha_{i,1}(t+1)} .
\]

 Removing the denominator and the common monomial factor $\beta_{i+1}(t)$, we see that the new exchange polynomial of $\beta_i(t)$ is
\[F_{\beta_i(t)} = \beta_{i-1}(t) \alpha_{i-1,2}(t) \beta_{i+1}(t+1) + \alpha_{i,1}(t+1) \alpha_{i-1,1}(t).\]
Let us now mutate $\alpha_{i-1,1}(t)$. Setting $\beta_i(t)=0$ in $F_{\alpha_{i-1,1}(t)}$, we obtain the following substitution:
\begin{equation*}
  F_{\beta_i(t)}\mid_{ \alpha_{i-1,1}(t) \leftarrow \frac{\alpha_{i-1,2}(t) \beta_{i-1}(t) \beta_{i-1}(t+1)}{\alpha_{i-1,1}(t+1)}     } = \frac{\left( \begin{aligned}
        &\beta_{i-1}(t) \alpha_{i-1,2}(t) \beta_{i+1}(t+1) \alpha_{i-1,1}(t+1) +\\
        & +\alpha_{i,1}(t+1) \alpha_{i-1,2}(t) \beta_{i-1}(t)  \beta_{i-1}(t+1)
      \end{aligned}\right)}{\alpha_{i-1,1}(t+1)} .      
\end{equation*}

Removing the denominator and the common monomial factor $\beta_{i-1}(t) \alpha_{i-1,2}(t)$, we see that the new exchange polynomial of $\beta_i(t)$ is 
\[F_{\beta_i(t)} = \beta_{i+1}(t+1) \alpha_{i-1,1}(t+1) + \alpha_{i,1}(t+1) \beta_{i-1}(t+1),\]
as claimed in the theorem.
It is easy to see that the exchange polynomials of $\beta_i(t+1)$-s change in the same way as those of $\beta_i(t)$-s.

At the time $\beta_{i}(t)$-s are mutated, the exchange polynomials of $\alpha_{i,1}(t+1)$-s, $\beta_i(t)$-s, and $\beta_{i+1}(t)$-s are as follows:
\begin{equation*}
  \begin{split}
    F_{\alpha_{i,1}(t+1)} =& \beta_i(t) \beta_i(t+1) \alpha_{i-1,2}(t+1) \alpha_{i+1,1}(t+1) +\\
                       &\beta_{i+1}(t) \beta_{i+1}(t+1) \alpha_{i-1,1}(t+1) \alpha_{i-1,2}(t+1),\\
F_{\beta_i(t)} =& \alpha_{i,1}(t+1) \beta_{i-1}(t+1) + \alpha_{i-1,1}(t+1) \beta_{i+1}(t+1),\\
F_{\beta_{i+1}(t)} =& \alpha_{i+1,1}(t+1) \beta_{i}(t+1) + \alpha_{i,1}(t+1) \beta_{i+2}(t+1).
  \end{split}
\end{equation*}

Let us mutate $\beta_{i}(t)$ first. Setting $\alpha_{i,1}(t+1)=0$ in $F_{\beta_{i}(t)}$, we obtain the following substitution:

\begin{equation*}
\resizebox{1\textwidth}{!} 
{$\displaystyle
    F_{\alpha_{i,1}(t+1)} \mid_{ \beta_{i}(t) \leftarrow \frac{\alpha_{i-1,1}(t+1) \beta_{i+1}(t+1)}{\beta_{i}(t+1) \beta_{i}(t+2)   }} = \frac{\left(
  \begin{aligned}
    &\alpha_{i-1,2}(t+1) \alpha_{i+1,1}(t+1) \alpha_{i-1,1}(t+1) \beta_{i+1}(t+1) + \\
    &+\beta_{i+1}(t) \beta_{i+1}(t+1) \alpha_{i-1,1}(t+1) \alpha_{i,2}(t+1) \beta_i(t+2)     
  \end{aligned}\right)
}{\beta_{i}(t+1) \beta_{i}(t+2)}.
$}
\end{equation*}

Removing the denominator and the common monomial factor $\beta_{i+1}(t+1) \alpha_{i-1,1}(t+1)$, we see that the new exchange polynomial of $\alpha_{i,1}(t+1)$ is
\[F_{\alpha_{i,1}(t+1)}= \alpha_{i-1,2}(t+1) \alpha_{i+1,1}(t+1) + \beta_{i+1}(t) \alpha_{i,2}(t+1) \beta_i(t+2).\]
Let us now mutate $\beta_{i+1}(t)$. Setting $\alpha_{i,1}(t+1)=0$ in $F_{\beta_{i+1}(t)}$, we obtain the following substitution:

\begin{equation*}
\resizebox{1\textwidth}{!} 
{$\displaystyle
  F_{\alpha_{i,1}(t+1)} \mid_{ \beta_{i+1}(t) \leftarrow \frac{\alpha_{i+1,1}(t+1) \beta_{i}(t+1)}{\beta_{i+1}(t+1) \beta_{i+1}(t+2)   }} = \frac{ \left(
      \begin{aligned}
        &\beta_{i+1}(t+1) \beta_{i+1}(t+2) \alpha_{i-1,2}(t+1) \alpha_{i+1,1}(t+1) + \\
        &+\beta_{i}(t+1) \beta_{i}(t+2) \alpha_{i+1,1}(t+1) \alpha_{i,2}(t+1)
      \end{aligned}
 \right)}{\beta_{i+1}(t+1) \beta_{i+1}(t+2)}.$
}
\end{equation*}
Removing the denominator and the common monomial factor $\alpha_{i+1,1}(t+1)$, we see that the new exchange polynomial of $\alpha_{i,1}(t+1)$ is
$$F_{\alpha_{i,1}(t+1)} = \beta_{i+1}(t+1) \beta_{i+1}(t+2) \alpha_{i-1,2}(t+1)  + \beta_{i}(t+1) \beta_{i}(t+2) \alpha_{i,2}(t+1),$$
thus finishing the proof of the theorem. 
\end{proof}

Now let us give a reduction to the $R$-system of $G_{n,m}$. Define the $X_{i,j}(t)$ for each $i\in[n]$ as follows.
\begin{equation}\label{eq:cylindric_reduction}
  \begin{split}
X_{i,1}(t) &= \frac{\alpha_{i,1}(t)}{\beta_{i}(t) \beta_{i+1}(t)},\\
X_{i,2}(t) &= \frac{\alpha_{i,2}(t)}{\beta_{i+1}(t) \beta_{i+1}(t+1)},\\
X_{i,j}(t) &= \frac{\alpha_{i,j}(t)}{\alpha_{i+1,j-2}(t+1)}, \quad\text{for } 3 \leq j \leq m,\\
X_{i,m+1}(t) &= \frac{\gamma_{i}(t) \gamma_{i}(t+1)}{\alpha_{i+1,m-1}(t+1)},\\
X_{i,m+2}(t) &= \frac{\gamma_{i}(t+1) \gamma_{i+1}(t+1)}{\alpha_{i+1,m}(t+1)}.
  \end{split}
\end{equation}

\begin{example}
The above reduction for our running example of $G_{3,3}$ is shown in Figure \ref{fig:cylindric_reduction}.
\end{example}

\begin{figure}
  \scalebox{1.4}{
\begin{tikzpicture}[scale=0.6]
\def\textscl{0.6}
\def\positn{0.5}
\node[scale=0.3,draw,circle,fill=black] (A11) at (1.50,2.00) { };
\node[scale=\textscl,inner sep=2pt,anchor=east] (AAA) at (A11.center) {$\frac{\alpha_{1,1}(t)}{\beta_{1}(t)\beta_{2}(t)}$};
\node[scale=0.3,draw,circle,fill=black] (A31) at (4.50,2.00) { };
\node[scale=\textscl,inner sep=2pt,anchor=east] (AAA) at (A31.center) {$\frac{\alpha_{2,1}(t)}{\beta_{2}(t)\beta_{3}(t)}$};
\node[scale=0.3,draw,circle,fill=black] (A51) at (7.50,2.00) { };
\node[scale=\textscl,inner sep=2pt,anchor=east] (AAA) at (A51.center) {$\frac{\alpha_{3,1}(t)}{\beta_{3}(t)\beta_{1}(t)}$};
\node[scale=0.3,draw,circle,fill=black] (A71) at (10.50,2.00) { };
\node[scale=\textscl,inner sep=2pt,anchor=east] (AAA) at (A71.center) {$\frac{\alpha_{1,1}(t)}{\beta_{1}(t)\beta_{2}(t)}$};
\node[scale=0.3,draw,circle,fill=black] (A91) at (13.50,2.00) { };
\node[scale=\textscl,inner sep=2pt,anchor=east] (AAA) at (A91.center) {$\frac{\alpha_{2,1}(t)}{\beta_{2}(t)\beta_{3}(t)}$};
\node[scale=\textscl] (A02) at (0.00,4.00) {$\dots$};
\node[scale=0.3,draw,circle,fill=black] (A22) at (3.00,4.00) { };
\node[scale=\textscl,inner sep=2pt,anchor=east] (AAA) at (A22.center) {$\frac{\alpha_{1,2}(t)}{\beta_{2}(t)\beta_{2}(t+1)}$};
\node[scale=0.3,draw,circle,fill=black] (A42) at (6.00,4.00) { };
\node[scale=\textscl,inner sep=2pt,anchor=east] (AAA) at (A42.center) {$\frac{\alpha_{2,2}(t)}{\beta_{3}(t)\beta_{3}(t+1)}$};
\node[scale=0.3,draw,circle,fill=black] (A62) at (9.00,4.00) { };
\node[scale=\textscl,inner sep=2pt,anchor=east] (AAA) at (A62.center) {$\frac{\alpha_{3,2}(t)}{\beta_{1}(t)\beta_{1}(t+1)}$};
\node[scale=0.3,draw,circle,fill=black] (A82) at (12.00,4.00) { };
\node[scale=\textscl,inner sep=2pt,anchor=east] (AAA) at (A82.center) {$\frac{\alpha_{1,2}(t)}{\beta_{2}(t)\beta_{2}(t+1)}$};
\node[scale=\textscl] (A102) at (15.00,4.00) {$\dots$};
\node[scale=0.3,draw,circle,fill=black] (A13) at (1.50,6.00) { };
\node[scale=\textscl,inner sep=2pt,anchor=east] (AAA) at (A13.center) {$\frac{\alpha_{3,3}(t)}{\alpha_{1,1}(t+1)}$};
\node[scale=0.3,draw,circle,fill=black] (A33) at (4.50,6.00) { };
\node[scale=\textscl,inner sep=2pt,anchor=east] (AAA) at (A33.center) {$\frac{\alpha_{1,3}(t)}{\alpha_{2,1}(t+1)}$};
\node[scale=0.3,draw,circle,fill=black] (A53) at (7.50,6.00) { };
\node[scale=\textscl,inner sep=2pt,anchor=east] (AAA) at (A53.center) {$\frac{\alpha_{2,3}(t)}{\alpha_{3,1}(t+1)}$};
\node[scale=0.3,draw,circle,fill=black] (A73) at (10.50,6.00) { };
\node[scale=\textscl,inner sep=2pt,anchor=east] (AAA) at (A73.center) {$\frac{\alpha_{3,3}(t)}{\alpha_{1,1}(t+1)}$};
\node[scale=0.3,draw,circle,fill=black] (A93) at (13.50,6.00) { };
\node[scale=\textscl,inner sep=2pt,anchor=east] (AAA) at (A93.center) {$\frac{\alpha_{1,3}(t)}{\alpha_{2,1}(t+1)}$};
\node[scale=\textscl] (A04) at (0.00,8.00) {$\dots$};
\node[scale=0.3,draw,circle,fill=black] (A24) at (3.00,8.00) { };
\node[scale=\textscl,inner sep=2pt,anchor=east] (AAA) at (A24.center) {$\frac{\gamma_{3}(t)\gamma_{3}(t+1)}{\alpha_{1,2}(t+1)}$};
\node[scale=0.3,draw,circle,fill=black] (A44) at (6.00,8.00) { };
\node[scale=\textscl,inner sep=2pt,anchor=east] (AAA) at (A44.center) {$\frac{\gamma_{1}(t)\gamma_{1}(t+1)}{\alpha_{2,2}(t+1)}$};
\node[scale=0.3,draw,circle,fill=black] (A64) at (9.00,8.00) { };
\node[scale=\textscl,inner sep=2pt,anchor=east] (AAA) at (A64.center) {$\frac{\gamma_{2}(t)\gamma_{2}(t+1)}{\alpha_{3,2}(t+1)}$};
\node[scale=0.3,draw,circle,fill=black] (A84) at (12.00,8.00) { };
\node[scale=\textscl,inner sep=2pt,anchor=east] (AAA) at (A84.center) {$\frac{\gamma_{3}(t)\gamma_{3}(t+1)}{\alpha_{1,2}(t+1)}$};
\node[scale=\textscl] (A104) at (15.00,8.00) {$\dots$};
\node[scale=0.3,draw,circle,fill=black] (A15) at (1.50,10.00) { };
\node[scale=\textscl,inner sep=2pt,anchor=east] (AAA) at (A15.center) {$\frac{\gamma_{2}(t+1)\gamma_{3}(t+1)}{\alpha_{3,3}(t+1)}$};
\node[scale=0.3,draw,circle,fill=black] (A35) at (4.50,10.00) { };
\node[scale=\textscl,inner sep=2pt,anchor=east] (AAA) at (A35.center) {$\frac{\gamma_{3}(t+1)\gamma_{1}(t+1)}{\alpha_{1,3}(t+1)}$};
\node[scale=0.3,draw,circle,fill=black] (A55) at (7.50,10.00) { };
\node[scale=\textscl,inner sep=2pt,anchor=east] (AAA) at (A55.center) {$\frac{\gamma_{1}(t+1)\gamma_{2}(t+1)}{\alpha_{2,3}(t+1)}$};
\node[scale=0.3,draw,circle,fill=black] (A75) at (10.50,10.00) { };
\node[scale=\textscl,inner sep=2pt,anchor=east] (AAA) at (A75.center) {$\frac{\gamma_{2}(t+1)\gamma_{3}(t+1)}{\alpha_{3,3}(t+1)}$};
\node[scale=0.3,draw,circle,fill=black] (A95) at (13.50,10.00) { };
\node[scale=\textscl,inner sep=2pt,anchor=east] (AAA) at (A95.center) {$\frac{\gamma_{3}(t+1)\gamma_{1}(t+1)}{\alpha_{1,3}(t+1)}$};
\def\op{1}
\draw[postaction={decorate},opacity=\op] (A02) -- (A13);
\draw[postaction={decorate},opacity=\op] (A04) -- (A15);
\draw[postaction={decorate},opacity=\op] (A11) -- (A22);
\draw[postaction={decorate},opacity=\op] (A11) -- (A02);
\draw[postaction={decorate},opacity=\op] (A13) -- (A24);
\draw[postaction={decorate},opacity=\op] (A13) -- (A04);
\draw[postaction={decorate},opacity=\op] (A22) -- (A33);
\draw[postaction={decorate},opacity=\op] (A22) -- (A13);
\draw[postaction={decorate},opacity=\op] (A24) -- (A35);
\draw[postaction={decorate},opacity=\op] (A24) -- (A15);
\draw[postaction={decorate},opacity=\op] (A31) -- (A42);
\draw[postaction={decorate},opacity=\op] (A31) -- (A22);
\draw[postaction={decorate},opacity=\op] (A33) -- (A44);
\draw[postaction={decorate},opacity=\op] (A33) -- (A24);
\draw[postaction={decorate},opacity=\op] (A42) -- (A53);
\draw[postaction={decorate},opacity=\op] (A42) -- (A33);
\draw[postaction={decorate},opacity=\op] (A44) -- (A55);
\draw[postaction={decorate},opacity=\op] (A44) -- (A35);
\draw[postaction={decorate},opacity=\op] (A51) -- (A62);
\draw[postaction={decorate},opacity=\op] (A51) -- (A42);
\draw[postaction={decorate},opacity=\op] (A53) -- (A64);
\draw[postaction={decorate},opacity=\op] (A53) -- (A44);
\draw[postaction={decorate},opacity=\op] (A62) -- (A73);
\draw[postaction={decorate},opacity=\op] (A62) -- (A53);
\draw[postaction={decorate},opacity=\op] (A64) -- (A75);
\draw[postaction={decorate},opacity=\op] (A64) -- (A55);
\draw[postaction={decorate},opacity=\op] (A71) -- (A82);
\draw[postaction={decorate},opacity=\op] (A71) -- (A62);
\draw[postaction={decorate},opacity=\op] (A73) -- (A84);
\draw[postaction={decorate},opacity=\op] (A73) -- (A64);
\draw[postaction={decorate},opacity=\op] (A82) -- (A93);
\draw[postaction={decorate},opacity=\op] (A82) -- (A73);
\draw[postaction={decorate},opacity=\op] (A84) -- (A95);
\draw[postaction={decorate},opacity=\op] (A84) -- (A75);
\draw[postaction={decorate},opacity=\op] (A91) -- (A102);
\draw[postaction={decorate},opacity=\op] (A91) -- (A82);
\draw[postaction={decorate},opacity=\op] (A93) -- (A104);
\draw[postaction={decorate},opacity=\op] (A93) -- (A84);
\draw[postaction={decorate},opacity=\op] (A102) -- (A93);
\draw[postaction={decorate},opacity=\op] (A104) -- (A95);
\end{tikzpicture}}
    \caption{The reduction to the $R$-system associated with $G_{3,3}$ from its $\tau$-sequence using~\eqref{eq:cylindric_reduction}.}
    \label{fig:cylindric_reduction}
\end{figure}
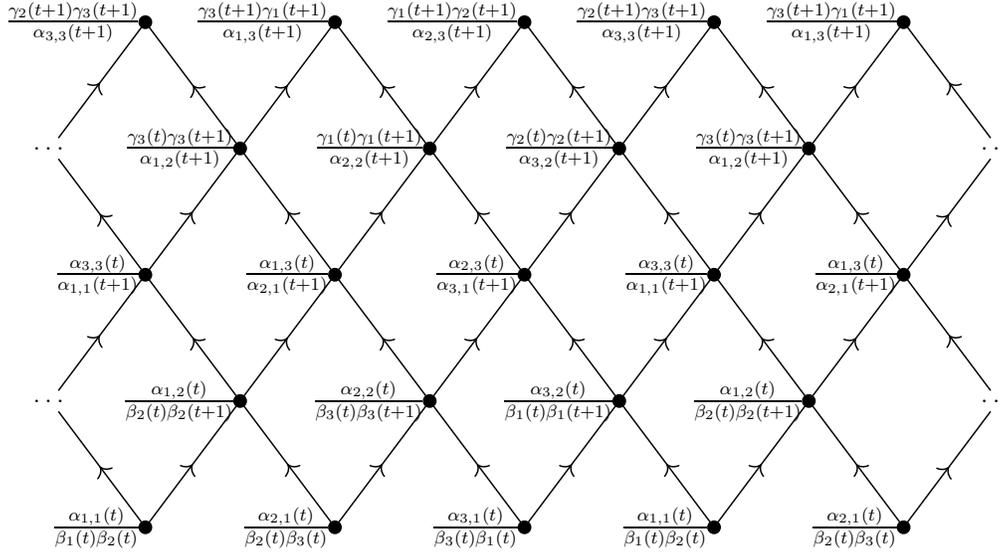

In what follows, we assume that the values of the $R$-system are rescaled so that the value at the vertex $\tb$ of $G_{n,m}$ is equal to $1$ for all $t$.
\begin{theorem} \label{thm:mapcyl}
 The $X_{i,j}(t)$-s satisfy the toggle relations~\eqref{eq:toggle} for the $R$-system associated with for $G_{n,m}$. Thus, the Laurent recurrence system with variables $\alpha, \beta, \gamma$ and polynomials $P$ given by Theorem \ref{thm:cylpoly} is a weak $\tau$-sequence for the $R$-system.  
\end{theorem}

\begin{proof}
We wish to verify the toggle relations for all the $X_{i,j}(t)$. There are several cases to consider, depending on the value of $j$, and also on the size of the parameter $m$. We do the computation for two of those cases, the other cases can be checked in a similar way. 

Let us assume $j=1$. The toggle relation is 
\[\frac{\alpha_{i,1}(t)}{\beta_i(t) \beta_{i+1}(t)} \frac{\alpha_{i,1}(t+1)}{\beta_i(t+1) \beta_{i+1}(t+1)} = \frac{\alpha_{i,2}(t)}{\beta_{i+1}(t) \beta_{i+1}(t+1)} + \frac{\alpha_{i-1,2}(t)}{\beta_{i}(t) \beta_{i}(t+1)}.\]
This is true since it is clearly equivalent to the exchange relation
\[\alpha_{i,1}(t) \alpha_{i,1}(t+1) = \alpha_{i,2}(t) \beta_{i}(t) \beta_{i}(t+1) + \alpha_{i-1,2}(t) \beta_{i+1}(t) \beta_{i+1}(t+1).\]

Let us now assume $j=2$ and $m \geq 3$. The toggle relation is 
\[\frac{\alpha_{i,2}(t)}{\beta_{i+1}(t) \beta_{i+1}(t+1)} \frac{\alpha_{i,2}(t)}{\beta_{i+1}(t) \beta_{i+1}(t+1)} = \frac{  \frac{\alpha_{i,3}(t)}{\alpha_{i+1,1}(t+1)} + \frac{\alpha_{i-1,3}(t)}{\alpha_{i,1}(t+1)}      }
  {   \frac{\beta_i(t+1) \beta_{i+1}(t+1)}{\alpha_{i,1}(t+1)}  + \frac{\beta_{i+1}(t+1) \beta_{i+2}(t+1)}{\alpha_{i+1,1}(t+1)}     }.\]
This is seen to be true by taking ratio of two exchange relations (and adjusting the denominator by an extra factor of $\beta_{i+1}(t+1)$):
\begin{align*}
    \alpha_{i,2}(t) \alpha_{i,2}(t+1) &= \alpha_{i,3}(t) \alpha_{i,1}(t+1) + \alpha_{i-1,3}(t) \alpha_{i+1,1}(t+1),\\
    \beta_{i+1}(t) \beta_{i+1}(t+1) \beta_{i+1}(t+2) &= \alpha_{i+1,1}(t+1) \beta_{i}(t+1) + \alpha_{i,1}(t+1) \beta_{i+2}(t+1).\qedhere
  \end{align*}
\end{proof}

\begin{theorem}
 For odd $n$, the map~\eqref{eq:cylindric_reduction} is dominant. 
\end{theorem}

\begin{proof}
 Take $X_{i,j}(t)$-s to be positive real numbers. This is a Zariski dense set. Pick arbitrary positive real $\beta$-s. Using $X_{i,j}(t)$-s for $j \leq m$ and exchange relations, we can uniquely solve for $\alpha_{i,j}(t)$-s. If $n$ is odd, we can use $X_{i,m+2}(t)$-s to solve for $\gamma_i(t+1)$-s, and then use  $X_{i,m+1}(t)$-s to solve for $\gamma_i(t)$-s. 
\end{proof}

\begin{conjecture}
The $R$-systems associated with cylindric posets $G_{n,m}$ have zero algebraic entropy.
\end{conjecture}

\subsection{Octagons}

Consider a square grid, which it will be convenient to orient so that the lines of the grid form angles $\pi/4$ and $3 \pi/4$ with the horizontal line. We consider all edges to be oriented upwards.
An {\it {octagon}} is a subset of all nodes enclosed by several lines forming angles $0, \pi/4, \pi/2, 3\pi/4$ with the horizontal line. Each octagon can be interpreted as a Hasse diagram of a partially ordered set. 
Note that an octagon can be highly degenerate, and have a de facto number of sides much smaller than eight. 
\begin{figure}

\scalebox{0.9}{
\begin{tikzpicture}[scale=0.6]
\def\positn{0.5}
\node[scale=0.3,draw,circle,fill=black] (A00) at (0.00,0.00) { };
\node[scale=0.3,draw,circle,fill=black] (A20) at (3.00,0.00) { };
\node[scale=0.3,draw,circle,fill=black] (A40) at (6.00,0.00) { };
\node[scale=0.3,draw,circle,fill=black] (A60) at (9.00,0.00) { };
\node[scale=0.3,draw,circle,fill=black] (A80) at (12.00,0.00) { };
\node[scale=0.3,draw,circle,fill=black] (A100) at (15.00,0.00) { };
\node[scale=0.3,draw,circle,fill=black] (A120) at (18.00,0.00) { };
\node[scale=0.3,draw,circle,fill=black] (A140) at (21.00,0.00) { };
\node[scale=0.3,draw,circle,fill=black] (A160) at (24.00,0.00) { };
\node[scale=0.3,draw,circle,fill=black] (A11) at (1.50,2.00) { };
\node[scale=0.3,draw,circle,fill=black] (A31) at (4.50,2.00) { };
\node[scale=0.3,draw,circle,fill=black] (A51) at (7.50,2.00) { };
\node[scale=0.3,draw,circle,fill=black] (A71) at (10.50,2.00) { };
\node[scale=0.3,draw,circle,fill=black] (A91) at (13.50,2.00) { };
\node[scale=0.3,draw,circle,fill=black] (A111) at (16.50,2.00) { };
\node[scale=0.3,draw,circle,fill=black] (A131) at (19.50,2.00) { };
\node[scale=0.3,draw,circle,fill=black] (A151) at (22.50,2.00) { };
\node[scale=0.3,draw,circle,fill=black] (A171) at (25.50,2.00) { };
\node[scale=0.3,draw,circle,fill=black] (A02) at (0.00,4.00) { };
\node[scale=0.3,draw,circle,fill=black] (A22) at (3.00,4.00) { };
\node[scale=0.3,draw,circle,fill=black] (A42) at (6.00,4.00) { };
\node[scale=0.3,draw,circle,fill=black] (A62) at (9.00,4.00) { };
\node[scale=0.3,draw,circle,fill=black] (A82) at (12.00,4.00) { };
\node[scale=0.3,draw,circle,fill=black] (A102) at (15.00,4.00) { };
\node[scale=0.3,draw,circle,fill=black] (A122) at (18.00,4.00) { };
\node[scale=0.3,draw,circle,fill=black] (A142) at (21.00,4.00) { };
\node[scale=0.3,draw,circle,fill=black] (A162) at (24.00,4.00) { };
\node[scale=0.3,draw,circle,fill=black] (A13) at (1.50,6.00) { };
\node[scale=0.3,draw,circle,fill=black] (A33) at (4.50,6.00) { };
\node[scale=0.3,draw,circle,fill=black] (A53) at (7.50,6.00) { };
\node[scale=0.3,draw,circle,fill=black] (A73) at (10.50,6.00) { };
\node[scale=0.3,draw,circle,fill=black] (A93) at (13.50,6.00) { };
\node[scale=0.3,draw,circle,fill=black] (A113) at (16.50,6.00) { };
\node[scale=0.3,draw,circle,fill=black] (A133) at (19.50,6.00) { };
\node[scale=0.3,draw,circle,fill=black] (A153) at (22.50,6.00) { };
\node[scale=0.3,draw,circle,fill=black] (A173) at (25.50,6.00) { };
\node[scale=0.3,draw,circle,fill=black] (A04) at (0.00,8.00) { };
\node[scale=0.3,draw,circle,fill=black] (A24) at (3.00,8.00) { };
\node[scale=0.3,draw,circle,fill=black] (A44) at (6.00,8.00) { };
\node[scale=0.3,draw,circle,fill=black] (A64) at (9.00,8.00) { };
\node[scale=0.3,draw,circle,fill=black] (A84) at (12.00,8.00) { };
\node[scale=0.3,draw,circle,fill=black] (A104) at (15.00,8.00) { };
\node[scale=0.3,draw,circle,fill=black] (A124) at (18.00,8.00) { };
\node[scale=0.3,draw,circle,fill=black] (A144) at (21.00,8.00) { };
\node[scale=0.3,draw,circle,fill=black] (A164) at (24.00,8.00) { };
\node[scale=0.3,draw,circle,fill=black] (A15) at (1.50,10.00) { };
\node[scale=0.3,draw,circle,fill=black] (A35) at (4.50,10.00) { };
\node[scale=0.3,draw,circle,fill=black] (A55) at (7.50,10.00) { };
\node[scale=0.3,draw,circle,fill=black] (A75) at (10.50,10.00) { };
\node[scale=0.3,draw,circle,fill=black] (A95) at (13.50,10.00) { };
\node[scale=0.3,draw,circle,fill=black] (A115) at (16.50,10.00) { };
\node[scale=0.3,draw,circle,fill=black] (A135) at (19.50,10.00) { };
\node[scale=0.3,draw,circle,fill=black] (A155) at (22.50,10.00) { };
\node[scale=0.3,draw,circle,fill=black] (A175) at (25.50,10.00) { };
\node[scale=0.3,draw,circle,fill=black] (A06) at (0.00,12.00) { };
\node[scale=0.3,draw,circle,fill=black] (A26) at (3.00,12.00) { };
\node[scale=0.3,draw,circle,fill=black] (A46) at (6.00,12.00) { };
\node[scale=0.3,draw,circle,fill=black] (A66) at (9.00,12.00) { };
\node[scale=0.3,draw,circle,fill=black] (A86) at (12.00,12.00) { };
\node[scale=0.3,draw,circle,fill=black] (A106) at (15.00,12.00) { };
\node[scale=0.3,draw,circle,fill=black] (A126) at (18.00,12.00) { };
\node[scale=0.3,draw,circle,fill=black] (A146) at (21.00,12.00) { };
\node[scale=0.3,draw,circle,fill=black] (A166) at (24.00,12.00) { };
\node[scale=0.3,draw,circle,fill=black] (A17) at (1.50,14.00) { };
\node[scale=0.3,draw,circle,fill=black] (A37) at (4.50,14.00) { };
\node[scale=0.3,draw,circle,fill=black] (A57) at (7.50,14.00) { };
\node[scale=0.3,draw,circle,fill=black] (A77) at (10.50,14.00) { };
\node[scale=0.3,draw,circle,fill=black] (A97) at (13.50,14.00) { };
\node[scale=0.3,draw,circle,fill=black] (A117) at (16.50,14.00) { };
\node[scale=0.3,draw,circle,fill=black] (A137) at (19.50,14.00) { };
\node[scale=0.3,draw,circle,fill=black] (A157) at (22.50,14.00) { };
\node[scale=0.3,draw,circle,fill=black] (A177) at (25.50,14.00) { };
\node[scale=0.3,draw,circle,fill=black] (A08) at (0.00,16.00) { };
\node[scale=0.3,draw,circle,fill=black] (A28) at (3.00,16.00) { };
\node[scale=0.3,draw,circle,fill=black] (A48) at (6.00,16.00) { };
\node[scale=0.3,draw,circle,fill=black] (A68) at (9.00,16.00) { };
\node[scale=0.3,draw,circle,fill=black] (A88) at (12.00,16.00) { };
\node[scale=0.3,draw,circle,fill=black] (A108) at (15.00,16.00) { };
\node[scale=0.3,draw,circle,fill=black] (A128) at (18.00,16.00) { };
\node[scale=0.3,draw,circle,fill=black] (A148) at (21.00,16.00) { };
\node[scale=0.3,draw,circle,fill=black] (A168) at (24.00,16.00) { };
\node[scale=0.3,draw,circle,fill=black] (A19) at (1.50,18.00) { };
\node[scale=0.3,draw,circle,fill=black] (A39) at (4.50,18.00) { };
\node[scale=0.3,draw,circle,fill=black] (A59) at (7.50,18.00) { };
\node[scale=0.3,draw,circle,fill=black] (A79) at (10.50,18.00) { };
\node[scale=0.3,draw,circle,fill=black] (A99) at (13.50,18.00) { };
\node[scale=0.3,draw,circle,fill=black] (A119) at (16.50,18.00) { };
\node[scale=0.3,draw,circle,fill=black] (A139) at (19.50,18.00) { };
\node[scale=0.3,draw,circle,fill=black] (A159) at (22.50,18.00) { };
\node[scale=0.3,draw,circle,fill=black] (A179) at (25.50,18.00) { };
\draw[postaction={decorate}] (A00) -- (A11);
\draw[postaction={decorate}] (A02) -- (A13);
\draw[postaction={decorate}] (A04) -- (A15);
\draw[postaction={decorate}] (A06) -- (A17);
\draw[postaction={decorate}] (A08) -- (A19);
\draw[postaction={decorate}] (A11) -- (A22);
\draw[postaction={decorate}] (A11) -- (A02);
\draw[postaction={decorate}] (A13) -- (A24);
\draw[postaction={decorate}] (A13) -- (A04);
\draw[postaction={decorate}] (A15) -- (A26);
\draw[postaction={decorate}] (A15) -- (A06);
\draw[postaction={decorate}] (A17) -- (A28);
\draw[postaction={decorate}] (A17) -- (A08);
\draw[postaction={decorate}] (A20) -- (A31);
\draw[postaction={decorate}] (A20) -- (A11);
\draw[postaction={decorate}] (A22) -- (A33);
\draw[postaction={decorate}] (A22) -- (A13);
\draw[postaction={decorate}] (A24) -- (A35);
\draw[postaction={decorate}] (A24) -- (A15);
\draw[postaction={decorate}] (A26) -- (A37);
\draw[postaction={decorate}] (A26) -- (A17);
\draw[postaction={decorate}] (A28) -- (A39);
\draw[postaction={decorate}] (A28) -- (A19);
\draw[postaction={decorate}] (A31) -- (A42);
\draw[postaction={decorate}] (A31) -- (A22);
\draw[postaction={decorate}] (A33) -- (A44);
\draw[postaction={decorate}] (A33) -- (A24);
\draw[postaction={decorate}] (A35) -- (A46);
\draw[postaction={decorate}] (A35) -- (A26);
\draw[postaction={decorate}] (A37) -- (A48);
\draw[postaction={decorate}] (A37) -- (A28);
\draw[postaction={decorate}] (A40) -- (A51);
\draw[postaction={decorate}] (A40) -- (A31);
\draw[postaction={decorate}] (A42) -- (A53);
\draw[postaction={decorate}] (A42) -- (A33);
\draw[postaction={decorate}] (A44) -- (A55);
\draw[postaction={decorate}] (A44) -- (A35);
\draw[postaction={decorate}] (A46) -- (A57);
\draw[postaction={decorate}] (A46) -- (A37);
\draw[postaction={decorate}] (A48) -- (A59);
\draw[postaction={decorate}] (A48) -- (A39);
\draw[postaction={decorate}] (A51) -- (A62);
\draw[postaction={decorate}] (A51) -- (A42);
\draw[postaction={decorate}] (A53) -- (A64);
\draw[postaction={decorate}] (A53) -- (A44);
\draw[postaction={decorate}] (A55) -- (A66);
\draw[postaction={decorate}] (A55) -- (A46);
\draw[postaction={decorate}] (A57) -- (A68);
\draw[postaction={decorate}] (A57) -- (A48);
\draw[postaction={decorate}] (A60) -- (A71);
\draw[postaction={decorate}] (A60) -- (A51);
\draw[postaction={decorate}] (A62) -- (A73);
\draw[postaction={decorate}] (A62) -- (A53);
\draw[postaction={decorate}] (A64) -- (A75);
\draw[postaction={decorate}] (A64) -- (A55);
\draw[postaction={decorate}] (A66) -- (A77);
\draw[postaction={decorate}] (A66) -- (A57);
\draw[postaction={decorate}] (A68) -- (A79);
\draw[postaction={decorate}] (A68) -- (A59);
\draw[postaction={decorate}] (A71) -- (A82);
\draw[postaction={decorate}] (A71) -- (A62);
\draw[postaction={decorate}] (A73) -- (A84);
\draw[postaction={decorate}] (A73) -- (A64);
\draw[postaction={decorate}] (A75) -- (A86);
\draw[postaction={decorate}] (A75) -- (A66);
\draw[postaction={decorate}] (A77) -- (A88);
\draw[postaction={decorate}] (A77) -- (A68);
\draw[postaction={decorate}] (A80) -- (A91);
\draw[postaction={decorate}] (A80) -- (A71);
\draw[postaction={decorate}] (A82) -- (A93);
\draw[postaction={decorate}] (A82) -- (A73);
\draw[postaction={decorate}] (A84) -- (A95);
\draw[postaction={decorate}] (A84) -- (A75);
\draw[postaction={decorate}] (A86) -- (A97);
\draw[postaction={decorate}] (A86) -- (A77);
\draw[postaction={decorate}] (A88) -- (A99);
\draw[postaction={decorate}] (A88) -- (A79);
\draw[postaction={decorate}] (A91) -- (A102);
\draw[postaction={decorate}] (A91) -- (A82);
\draw[postaction={decorate}] (A93) -- (A104);
\draw[postaction={decorate}] (A93) -- (A84);
\draw[postaction={decorate}] (A95) -- (A106);
\draw[postaction={decorate}] (A95) -- (A86);
\draw[postaction={decorate}] (A97) -- (A108);
\draw[postaction={decorate}] (A97) -- (A88);
\draw[postaction={decorate}] (A100) -- (A111);
\draw[postaction={decorate}] (A100) -- (A91);
\draw[postaction={decorate}] (A102) -- (A113);
\draw[postaction={decorate}] (A102) -- (A93);
\draw[postaction={decorate}] (A104) -- (A115);
\draw[postaction={decorate}] (A104) -- (A95);
\draw[postaction={decorate}] (A106) -- (A117);
\draw[postaction={decorate}] (A106) -- (A97);
\draw[postaction={decorate}] (A108) -- (A119);
\draw[postaction={decorate}] (A108) -- (A99);
\draw[postaction={decorate}] (A111) -- (A122);
\draw[postaction={decorate}] (A111) -- (A102);
\draw[postaction={decorate}] (A113) -- (A124);
\draw[postaction={decorate}] (A113) -- (A104);
\draw[postaction={decorate}] (A115) -- (A126);
\draw[postaction={decorate}] (A115) -- (A106);
\draw[postaction={decorate}] (A117) -- (A128);
\draw[postaction={decorate}] (A117) -- (A108);
\draw[postaction={decorate}] (A120) -- (A131);
\draw[postaction={decorate}] (A120) -- (A111);
\draw[postaction={decorate}] (A122) -- (A133);
\draw[postaction={decorate}] (A122) -- (A113);
\draw[postaction={decorate}] (A124) -- (A135);
\draw[postaction={decorate}] (A124) -- (A115);
\draw[postaction={decorate}] (A126) -- (A137);
\draw[postaction={decorate}] (A126) -- (A117);
\draw[postaction={decorate}] (A128) -- (A139);
\draw[postaction={decorate}] (A128) -- (A119);
\draw[postaction={decorate}] (A131) -- (A142);
\draw[postaction={decorate}] (A131) -- (A122);
\draw[postaction={decorate}] (A133) -- (A144);
\draw[postaction={decorate}] (A133) -- (A124);
\draw[postaction={decorate}] (A135) -- (A146);
\draw[postaction={decorate}] (A135) -- (A126);
\draw[postaction={decorate}] (A137) -- (A148);
\draw[postaction={decorate}] (A137) -- (A128);
\draw[postaction={decorate}] (A140) -- (A151);
\draw[postaction={decorate}] (A140) -- (A131);
\draw[postaction={decorate}] (A142) -- (A153);
\draw[postaction={decorate}] (A142) -- (A133);
\draw[postaction={decorate}] (A144) -- (A155);
\draw[postaction={decorate}] (A144) -- (A135);
\draw[postaction={decorate}] (A146) -- (A157);
\draw[postaction={decorate}] (A146) -- (A137);
\draw[postaction={decorate}] (A148) -- (A159);
\draw[postaction={decorate}] (A148) -- (A139);
\draw[postaction={decorate}] (A151) -- (A162);
\draw[postaction={decorate}] (A151) -- (A142);
\draw[postaction={decorate}] (A153) -- (A164);
\draw[postaction={decorate}] (A153) -- (A144);
\draw[postaction={decorate}] (A155) -- (A166);
\draw[postaction={decorate}] (A155) -- (A146);
\draw[postaction={decorate}] (A157) -- (A168);
\draw[postaction={decorate}] (A157) -- (A148);
\draw[postaction={decorate}] (A160) -- (A171);
\draw[postaction={decorate}] (A160) -- (A151);
\draw[postaction={decorate}] (A162) -- (A173);
\draw[postaction={decorate}] (A162) -- (A153);
\draw[postaction={decorate}] (A164) -- (A175);
\draw[postaction={decorate}] (A164) -- (A155);
\draw[postaction={decorate}] (A166) -- (A177);
\draw[postaction={decorate}] (A166) -- (A157);
\draw[postaction={decorate}] (A168) -- (A179);
\draw[postaction={decorate}] (A168) -- (A159);
\draw[postaction={decorate}] (A171) -- (A162);
\draw[postaction={decorate}] (A173) -- (A164);
\draw[postaction={decorate}] (A175) -- (A166);
\draw[postaction={decorate}] (A177) -- (A168);
\draw[blue,rounded corners,line width=1pt] (4.418820584978071,1.7386874070247247)--(13.58117941502193,1.7386874070247247)--(15.195984444731458,3.8917607799707605)--(15.195984444731458,8.108239220029239)--(10.581179415021928,14.261312592975276)--(4.418820584978071,14.261312592975276)--(1.3040155552685435,10.108239220029239)--(1.3040155552685435,5.8917607799707605)--cycle;
\draw[red,rounded corners,line width=1pt] (19.35,1.8)--(25.950000000000003,1.8)--(22.65,6.2)--(19.35,6.2)--(17.700000000000003,4.0)--cycle;
\draw[green!70!black,rounded corners,line width=1pt] (24.0,7.6)--(25.799999999999997,10.0)--(21.0,16.4)--(19.200000000000003,14.0)--cycle;
\end{tikzpicture}}
    \caption{Several octagon posets.}
    \label{fig:octagons}
\end{figure}
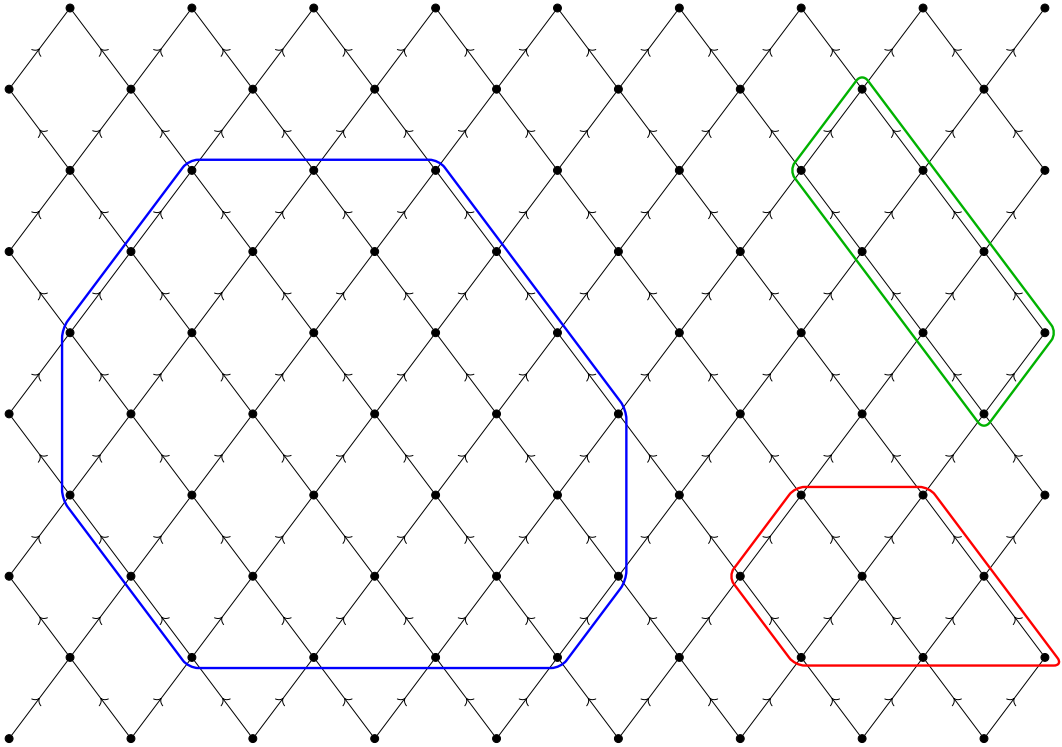

\begin{example}
 In Figure \ref{fig:octagons}, we see three octagons. The blue octagon is \emph{generic}, i.e., all of its sides are non-degenerate. The green and the purple octagons are degenerate.
\end{example}

The following conjecture says that the $R$-systems associated with octagon posets have the singularity confinement and zero algebraic entropy properties. 
\begin{conjecture}
 For each octagon poset $P$, the $R$-system associated with $G(P)$ admits a strong $\tau$-sequence which consists of (irreducible) Laurent polynomials whose degrees grow at most quadratically.
\end{conjecture}

\begin{example}
 A special case of this conjecture is already known. Specifically, {\it {rectangles}} similar to the green octagon in Figure~\ref{fig:octagons} have been shown by Grinberg and Roby to be periodic, see \cite{GR2}.
 Note that Grinberg and Roby conjecture periodicity of birational rowmotion for some other families of posets, all of which are degenerate octagons in our terminology. 
\end{example}

\begin{remark}
 Note that our cylindric posets $G_{n,m}$ can be viewed as ``periodic octagons''. 
\end{remark}

\section{Toric digraphs}\label{sec:toric-graphs}
In this section, we continue looking at more and more symmetric digraphs. Namely, we study the coefficient-free $R$-system associated with a \emph{toric digraph} which gives rise to a very symmetric recurrence sequence that conjecturally has the Laurent property.

\newcommand \cl[2]{\overline{(#1,#2)}}
\newcommand \cll[1]{[#1]}

\def\L{\Lambda}
\begin{definition}
  Let $\L\subset \Z^2$ be a lattice of rank $2$ that does not contain the basis vectors $e_1:=(1,0)$, $e_2:=(0,1)$, and $e_2\pm e_1$. Then the \emph{toric digraph} $G:=G(\L)=(V,E)$ is a digraph with vertex set $V:=\Z^2/\L$ and edge set
  \[E=\{(\cll v,\cll{v+e_1}),(\cll v,\cll{v+e_2})\mid v\in\Z^2\}.\]
  Here $\cll v$ denotes the class of $v\in\Z^2$ in $\Z^2/\L$.
\end{definition}

For example, see Figure~\ref{fig:toric_4}. The assumption $e_1,e_2\notin\L$ implies that $G$ has no loops and having $e_2-e_1\notin\L$ means that $G$ does not have multiple edges. If $e_2+e_1\in\L$ then $G$ is a bidirected cycle which will be considered in the next section. As we will see, toric digraphs are the first class of digraphs with a transitive group of automorphisms for which the coefficient-free $R$-system dynamics is non-trivial.

\def\pr{\pi}
\def\WTT{\alpha}
We now define a family $(\Y(t))_{t\geq -2}$ of conjecturally irreducible polynomials in the variables $\x_V=(x_v)_{v\in V}$ that will describe the behavior of the $R$-system associated with $G$. We define the $\Y$-polynomials recursively as follows. Set
\[\Y_v(-2):=x_{v+e};\quad \Y_v(-1):=1,\]
where  $e:=e_1+e_2$ and $v+e\in V$ is defined by $[v'+e]$ for $v'\in\Z^2$ such that $[v']=v$. Now, for $v\in V$, we denote $p:=v+e_1$, $q:=v+e_2$, $r:=v+e$, and for $t\geq -1$, set
\begin{equation}\label{eq:big_sum_toric}
  \Y_r(t+1):=\frac{\left(\prod_{u\in V}x_u \right)\sum_{\arb\in\Arb(G,v)}\WTT(\arb)}{\left(\prod_{u\in V} \Y_u(t-1)\right)\left(\prod_{u\in V\setminus \{p,q,r\}}\Y_u(t)\right)},
\end{equation}
where $\WTT(\arb)$ is a certain expression in $\Y(t)$ and $\Y(t-1)$. Namely, define $G_v:=(V,E_v)$ to be the digraph obtained from $G$ by removing the edges $(v,p)$ and $(v,q)$. In other words, $E_v:=E\setminus\{(v,p),(v,q)\}$. Then
\[\WTT(\arb):=\prod_{(a,b)\in E_v} 
  \begin{cases}
    \Y_b(t-1), &\text{if $\arb(a)=b$,}\\
    \Y_b(t), &\text{otherwise.}\\
  \end{cases}\]

This defines $\Y_v(t+1)$ as a rational function. Our computations suggest the following Laurent phenomenon:
\begin{conjecture}\label{conj:toric}
The rational functions $\Y_v(t)$ are pairwise coprime irreducible polynomials\footnote{More precisely, they are irreducible and coprime as elements of $\Z[\x_V^{\pm1}]$, i.e., as Laurent polynomials, even though we conjecture that they are actual polynomials. In other words, they are pairwise coprime and irreducible up to some monomial factors.} in $\x_V$. 
\end{conjecture}
Note that every monomial $\WTT(\arb)$ in the numerator is always divisible by $\Y_v(t-1)$ which however is included in the product in the denominator.

\begin{proposition}
  Suppose that Conjecture~\ref{conj:toric} holds. Then the following is true:
  \begin{enumerate}[\normalfont (i)]
  \item\label{item:toric_R} The values of the universal coefficient-free $R$-system associated with $G$ are given by
\begin{equation}\label{eq:toric_R}
    R_v(t+1)=\frac{\Y_{v+te}(t-1)}{\Y_{v+te}(t)}
\end{equation}
    for all $t\geq -1$. In particular, it has the singularity confinement property, and $\Y(t)$ is a strong $\tau$-sequence for this $R$-system.
  \item\label{item:toric_degree} For each $v\in V$ and $t\geq -2$, $\Y_v(t)$ is a sum of $\kappa^{t+2\choose 2}$ monomials, each of degree
    \[\frac12(nt^2+(3n-2)t+(2n-2)),\]
     where $n=|V|$ is the number of vertices of $G$ and $\kappa:=|\Arb(G,u)|$ is the \emph{complexity} of $G$ (which clearly does not depend on $u$).
  \item\label{item:toric_entropy} In particular, the $R$-system associated with $G$ has zero algebraic entropy.
  \item\label{item:toric_P} For each $v\in V$ and $t\geq -1$, the polynomials $\Y_u(t)$ satisfy
    \begin{equation}\label{eq:P_recurrence}
      (x_{q+e}+x_{p+e}) \Y_r(t) \Y_v(t-1)=x_r(\Y_{q}(t)\Y_{p}(t-1)+\Y_{p}(t)\Y_{q}(t-1)),
    \end{equation}
    where $p:=v+e_1$, $q:=v+e_2$, and $r:=v+e$ as before.
  \end{enumerate}
\end{proposition}
\begin{proof}
  Part~\eqref{item:toric_R} is just a restatement of~\eqref{eq:arb}. Part~\eqref{item:toric_degree} follows by induction from~\eqref{eq:big_sum_toric} (note that $\WTT(\arb)$ contains exactly $n-1$ factors of the form $\Y_b(t-1)$ and $n-1$ factors of the form $\Y_b(t)$). Part~\eqref{item:toric_entropy} is a direct consequence of Part~\eqref{item:toric_degree}. Finally, to show Part~\eqref{item:toric_P}, let us write down the toggle relations~\eqref{eq:toggle} for the $R$-system at vertex $v-(t-1)e$ at time $t$:
  \[R_{v-(t-1)e}(t)R_{v-(t-1)e}(t+1)=\frac{R_{p-(t-1)e}(t)+R_{q-(t-1)e}(t)}{R_{p-te}(t+1)^{-1}+R_{q-te}(t+1)^{-1}}. \]
  Applying~\eqref{eq:toric_R} yields
  \[\frac{\Y_{v}(t-2)}{\Y_{v}(t-1)}\cdot\frac{\Y_{v+e}(t-1)}{\Y_{v+e}(t)}=
    \frac{\Y_{p}(t-2)\Y_{q}(t-1)+\Y_{q}(t-2)\Y_{p}(t-1)}{\Y_{p}(t-1)\Y_{q}(t)+\Y_{q}(t-1)\Y_{p}(t)}.\]
  In other words, we get that
  \[\frac{\Y_{p}(t-1)\Y_{q}(t)+\Y_{q}(t-1)\Y_{p}(t)}{\Y_{v}(t-1)\Y_{v+e}(t)}
   = \frac{\Y_{p}(t-2)\Y_{q}(t-1)+\Y_{q}(t-2)\Y_{p}(t-1)}{\Y_{v}(t-2)\Y_{v+e}(t-1)}\]
  is a conserved quantity, and thus is equal to its value at $t=0$, namely,
 \[\frac{\Y_{p}(-2)\Y_{q}(-1)+\Y_{q}(-2)\Y_{p}(-1)}{\Y_{v}(-2)\Y_{v+e}(-1)}=\frac{x_{p+e}+x_{q+e}}{x_{v+e}}.\]
This finishes the proof of~\eqref{item:toric_P}.  
\end{proof}

Let us also give a certain step towards proving Conjecture~\ref{conj:toric}.

\def\bigsum{Q}
\begin{proposition}\label{prop:toric_divide}
Let $t\geq -2$ and suppose that $\Y_v(t')$ is a Laurent polynomial for all $v\in V$ and $-2\leq t'\leq t$. Moreover, assume that any two such Laurent polynomials are coprime. Then $\Y_v(t+1)$ is also a Laurent polynomial.
\end{proposition}
\begin{proof}
Denote by $\bigsum$ the sum in the numerator of~\eqref{eq:big_sum_toric}:
  \[\bigsum:=\sum_{\arb\in\Arb(G,v)}\WTT(\arb).\]
  It suffices to show that $\bigsum$ is divisible by each term in the denominator since these terms are pairwise coprime. Note that in the proof of~\eqref{eq:P_recurrence}, we did not use Conjecture~\ref{conj:toric}, we only used~\eqref{eq:toric_R}. Thus we can use~\eqref{eq:P_recurrence} here and it implies that for any vertex $u\in V$, the (Laurent) polynomial
  \[\Y_{u+e_2}(t)\Y_{u+e_1}(t-1)+\Y_{u+e_1}(t)\Y_{u+e_2}(t-1)\]
  is divisible by $\Y_{u+e}(t)$ and $\Y_u(t-1)$.

  Fix some vertex $u$ such that $u+e\in V\setminus \{p,q,r\}$. Note that for $\arb\in\Arb(G,v)$, $\WTT(\arb)$ does not contain $\Y_{u+e}(t)$ as a factor if and only if $\arb(u+e_1)=\arb(u+e_2)=u+e$ (note also that neither one of $u$, $u+e_1$, and $u+e_2$ is equal to $v$). For every such arborescence $\arb$, there is a unique arborescence $\arb'\in\Arb(G,v)$ such that $\arb'(w)= \arb(w)$ if and only if $w\neq u$. Note also that $\WTT(\arb)+\WTT(\arb')$ is divisible by $(\Y_{u+e_2}(t)\Y_{u+e_1}(t-1)+\Y_{u+e_1}(t)\Y_{u+e_2}(t-1))$ and thus by $\Y_{u+e}(t)$. This shows that $\bigsum$ is divisible by $\Y_{u+e}(t)$.

  As we have already mentioned, $\bigsum$ is divisible by $\Y_v(t-1)$, since every term of $\bigsum$ is divisible by it. Let now $u\neq v$ be a vertex and we would like to show that $\bigsum$ is divisible by $\Y_u(t-1)$. For $\arb\in\Arb(G,v)$, $\WTT(\arb)$ is divisible by $\Y_u(t-1)$ unless $u$ has indegree $0$ in $\arb$. But then there is again a unique arborescence $\arb'\in\Arb(G,v)$ such that $\arb'(w)= \arb(w)$ if and only if $w\neq u$. And again $\WTT(\arb)+\WTT(\arb')$ is divisible by $(\Y_{u+e_2}(t)\Y_{u+e_1}(t-1)+\Y_{u+e_1}(t)\Y_{u+e_2}(t-1))$ and thus by $\Y_{u}(t-1)$. We have shown that $\bigsum$ is divisible by each $\Y_u(t-1)$ and thus $\Y_v(t+1)$ is a Laurent polynomial.
\end{proof}

\begin{example}\label{ex:toric_4}
  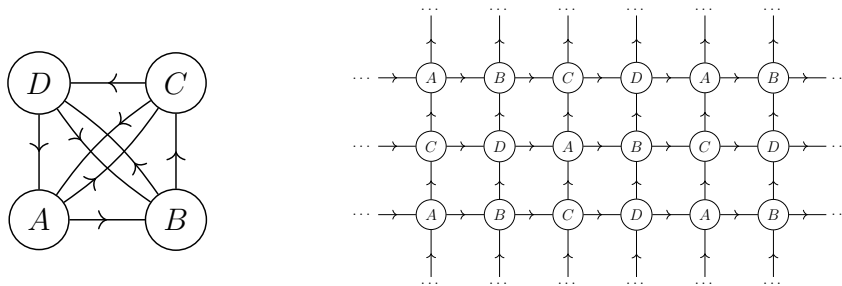
\begin{figure}
  \def\nodesc{0.7}
  \def\labelsc{1}
  \def\tikzscx{1.3}
  \def\tikzscy{1.3}
  \def\arrsc{0.6}
  \def\sclbx{0.7}
  \def\sclbxbig{1.4}
  \def\bnd{10}
 \begin{tabular}{cc}
   \scalebox{\sclbxbig}{
   \begin{tikzpicture}[xscale=\tikzscx,yscale=\tikzscy]
     \draw[white] (0,-0.5) -- (0,1.5);
     \draw[white] (-1,0) -- (2,0);
    \node[scale=\nodesc,draw,circle] (A) at (0,0) {$A$};
    \node[scale=\nodesc,draw,circle] (B) at (1,0) {$B$};
    \node[scale=\nodesc,draw,circle] (C) at (1,1) {$C$};
    \node[scale=\nodesc,draw,circle] (D) at (0,1) {$D$};
    \draw[postaction={decorate}] (A)--(B);
    \draw[postaction={decorate}] (B)--(C);
    \draw[postaction={decorate}] (C)--(D);
    \draw[postaction={decorate}] (D)--(A);
    \def\positn{0.3}
    \draw[postaction={decorate}] (A) to[bend right=\bnd] (C);
    \draw[postaction={decorate}] (C) to[bend right=\bnd] (A);
    \draw[postaction={decorate}] (B) to[bend right=\bnd] (D);
    \draw[postaction={decorate}] (D) to[bend right=\bnd] (B);
    \def\positn{0.5}
    
  \end{tikzpicture}}&
   \scalebox{\sclbx}{
    \begin{tikzpicture}[xscale=\tikzscx,yscale=\tikzscy]
    \node[scale=\nodesc] (X1) at (-1,0) {$\dots$};
    \node[scale=\nodesc,draw,circle] (A1) at (0,0) {$A$};
    \node[scale=\nodesc,draw,circle] (B1) at (1,0) {$B$};
    \node[scale=\nodesc,draw,circle] (C1) at (2,0) {$C$};
    \node[scale=\nodesc,draw,circle] (D1) at (3,0) {$D$};
    \node[scale=\nodesc,draw,circle] (A11) at (4,0) {$A$};
    \node[scale=\nodesc,draw,circle] (B11) at (5,0) {$B$};
    \node[scale=\nodesc] (Y1) at (6,0) {$\dots$};
    \node[scale=\nodesc] (X2) at (-1,1) {$\dots$};
    \node[scale=\nodesc,draw,circle] (C2) at (0,1) {$C$};
    \node[scale=\nodesc,draw,circle] (D2) at (1,1) {$D$};
    \node[scale=\nodesc,draw,circle] (A2) at (2,1) {$A$};
    \node[scale=\nodesc,draw,circle] (B2) at (3,1) {$B$};
    \node[scale=\nodesc,draw,circle] (C22) at (4,1) {$C$};
    \node[scale=\nodesc,draw,circle] (D22) at (5,1) {$D$};
    \node[scale=\nodesc] (Y2) at (6,1) {$\dots$};
    \node[scale=\nodesc] (X3) at (-1,2) {$\dots$};
    \node[scale=\nodesc,draw,circle] (A3) at (0,2) {$A$};
    \node[scale=\nodesc,draw,circle] (B3) at (1,2) {$B$};
    \node[scale=\nodesc,draw,circle] (C3) at (2,2) {$C$};
    \node[scale=\nodesc,draw,circle] (D3) at (3,2) {$D$};
    \node[scale=\nodesc,draw,circle] (A33) at (4,2) {$A$};
    \node[scale=\nodesc,draw,circle] (B33) at (5,2) {$B$};
    \node[scale=\nodesc] (Y3) at (6,2) {$\dots$};
    \node[scale=\nodesc] (S0) at (0,-1) {$\dots$};
    \node[scale=\nodesc] (S1) at (1,-1) {$\dots$};
    \node[scale=\nodesc] (S2) at (2,-1) {$\dots$};
    \node[scale=\nodesc] (S3) at (3,-1) {$\dots$};
    \node[scale=\nodesc] (S4) at (4,-1) {$\dots$};
    \node[scale=\nodesc] (S5) at (5,-1) {$\dots$};
    \node[scale=\nodesc] (N0) at (0,3) {$\dots$};
    \node[scale=\nodesc] (N1) at (1,3) {$\dots$};
    \node[scale=\nodesc] (N2) at (2,3) {$\dots$};
    \node[scale=\nodesc] (N3) at (3,3) {$\dots$};
    \node[scale=\nodesc] (N4) at (4,3) {$\dots$};
    \node[scale=\nodesc] (N5) at (5,3) {$\dots$};
    
    \draw[postaction={decorate}] (X1)--(A1);
    \draw[postaction={decorate}] (A1)--(B1);
    \draw[postaction={decorate}] (B1)--(C1);
    \draw[postaction={decorate}] (C1)--(D1);
    \draw[postaction={decorate}] (D1)--(A11);
    \draw[postaction={decorate}] (A11)--(B11);
    \draw[postaction={decorate}] (B11)--(Y1);
    \draw[postaction={decorate}] (X2)--(C2);
    \draw[postaction={decorate}] (C2)--(D2);
    \draw[postaction={decorate}] (D2)--(A2);
    \draw[postaction={decorate}] (A2)--(B2);
    \draw[postaction={decorate}] (B2)--(C22);
    \draw[postaction={decorate}] (C22)--(D22);
    \draw[postaction={decorate}] (D22)--(Y2);
    \draw[postaction={decorate}] (X3)--(A3);
    \draw[postaction={decorate}] (A3)--(B3);
    \draw[postaction={decorate}] (B3)--(C3);
    \draw[postaction={decorate}] (C3)--(D3);
    \draw[postaction={decorate}] (D3)--(A33);
    \draw[postaction={decorate}] (A33)--(B33);
    \draw[postaction={decorate}] (B33)--(Y3);
    \draw[postaction={decorate}] (S0)--(A1);
    \draw[postaction={decorate}] (S1)--(B1);
    \draw[postaction={decorate}] (S2)--(C1);
    \draw[postaction={decorate}] (S3)--(D1);
    \draw[postaction={decorate}] (S4)--(A11);
    \draw[postaction={decorate}] (S5)--(B11);
    \draw[postaction={decorate}] (A1)--(C2);
    \draw[postaction={decorate}] (B1)--(D2);
    \draw[postaction={decorate}] (C1)--(A2);
    \draw[postaction={decorate}] (D1)--(B2);
    \draw[postaction={decorate}] (A11)--(C22);
    \draw[postaction={decorate}] (B11)--(D22);
    \draw[postaction={decorate}] (C2)--(A3);
    \draw[postaction={decorate}] (D2)--(B3);
    \draw[postaction={decorate}] (A2)--(C3);
    \draw[postaction={decorate}] (B2)--(D3);
    \draw[postaction={decorate}] (C22)--(A33);
    \draw[postaction={decorate}] (D22)--(B33);
    \draw[postaction={decorate}] (A3)--(N0);
    \draw[postaction={decorate}] (B3)--(N1);
    \draw[postaction={decorate}] (C3)--(N2);
    \draw[postaction={decorate}] (D3)--(N3);
    \draw[postaction={decorate}] (A33)--(N4);
    \draw[postaction={decorate}] (B33)--(N5);
    
  \end{tikzpicture}}

\end{tabular}
    \caption{\label{fig:toric_4} The toric digraph $G$ (left) from Example~\ref{ex:toric_4}. The same digraph drawn on the \emph{universal cover} of the torus in a periodic way (right).}
  \end{figure}
  The smallest toric digraph $G$ comes from $\L$ spanned by the vectors $(3,1)$ and $(0,2)$, see Figure~\ref{fig:toric_4}. Thus $G$ has four vertices and $\kappa=5$ arborescences towards each vertex. Let us denote the vertices of $G$ by $A,B,C,D$ as in Figure~\ref{fig:toric_4}, and let $a:=x_A,b:=x_B,c:=x_C,d:=x_D$ be the vertex variables. Then we have:
  \[\Y_A(-2)=d,\quad \Y_B(-2)=a,\quad \Y_C(-2)=b,\quad \Y_D(-2)=c;\]
  \[\Y_A(-1)=\Y_B(-1)=\Y_C(-1)=\Y_D(-1)=1;\]
  \[\Y_A(0)=(ac + bc + ad + cd + d^2)a,\]
  \[\Y_B(0)=(a^2 + ab + ad + bd + cd)b,\]
  \[\Y_C(0)=(ab + b^2 + ac + bc + ad)c,\]
  \[\Y_D(0)=(ab + bc + c^2 + bd + cd)d.\]
  Let us check~\eqref{eq:P_recurrence} for example for $v=B$ and $t=0$. We have
  \[p=C,\quad q=D,\quad r=A,\quad p+e=B,\quad q+e=C,\]
  and therefore~\eqref{eq:P_recurrence} becomes
\begin{equation}\label{eq:toric_4_P}
 (b+c) \Y_A(0) =a(\Y_D(0)+\Y_C(0)).
\end{equation}
One easily verifies that this is indeed the case.

  The next generation of the $\Y$-polynomials is already too big to write down explicitly. For example,
  \[\Y_D(1)=\frac{abcd\cdot \bigsum }{Y_A(0)},\]
  where $\bigsum$ is
\begin{equation*}
\Y_B(0)\Y_D(0)^2+\Y_A(0)\Y_D(0)^2+\Y_B(0)\Y_C(0)\Y_D(0)+\Y_A(0)\Y_B(0)\Y_D(0)+\Y_A(0)\Y_B(0)\Y_C(0).
\end{equation*}
To see that it is divisible by $\Y_A(0)$, note that its only two terms which do not contain $\Y_A(0)$ as a factor are $\Y_B(0)\Y_D(0)^2$ and $\Y_B(0)\Y_C(0)\Y_D(0)$. By~\eqref{eq:toric_4_P}, their sum becomes
\[\Y_B(0)\Y_D(0)(\Y_D(0)+\Y_C(0))=\frac{b+c}{a}\Y_A(0)\Y_B(0)\Y_D(0),\]
and thus we compute that
\[\Y_D(1)=abcd\cdot(\Y_D(0)^2+\Y_B(0)\Y_D(0)+\Y_B(0)\Y_C(0))+bcd(b+c)\Y_B(0)\Y_D(0).\]
In particular, $\Y_D(1)$ is indeed a polynomial that contains $125=\kappa^{1+2\choose 2}$ monomials, each of degree
\[10=\frac12(4+10+6).\] 
\end{example}

\def\arbsum{\gamma}
\begin{remark}
It follows from~\eqref{eq:big_sum_toric} that we have
\[\Y_v(0)=\arbsum(v):=\sum_{\arb\in\Arb(G,v)} \prod_{u\in V\setminus\{v\}} x_{T(u)}.\]
Rewriting~\eqref{eq:P_recurrence} for $t=0$ yields
\[(x_{r+e_1}+x_{r+e_2})\arbsum(r)=x_r(\arbsum(p)+\arbsum(q)).\]
One may ask for a bijective proof of this identity and in fact there is a simple argument: both sides of the above equation count \emph{unicycle configurations} of the form $(v,\rho)$, see~\cite[Section~3]{HLMPP}. For $t=1$, one can give a combinatorial interpretation to $\Y_v(1)$ using the proof of Proposition~\ref{prop:toric_divide}, but it remains an open problem to give a combinatorial interpretation to the $\kappa^{t+2\choose 2}$ monomials appearing in $\Y_v(t)$ for $t\geq2$.
\end{remark}

\def\Gnot{G_0}
\def\Enot{E_0}
\section{Bidirected digraphs}\label{sec:bidirected-graphs}
Finally, we study the case when the digraph $G=(V,E)$ is \emph{bidirected}, i.e., it is obtained from an undirected connected digraph $\Gnot=(V,\Enot)$ by orienting each edge in both directions:
\[E=\{(u,v),(v,u)\mid \{u,v\}\in \Enot\}.\]
Surprisingly enough, for all bidirected digraphs, the $R$-system behaves in a very similar way. We first treat the coefficient-free case.

\begin{proposition}\label{prop:bidirected_periodic}
Let $G$ be a bidirected digraph. Then the universal coefficient-free $R$-system associated with $G$ is periodic with period $2$, and its values are monomials in $\x_V$.
\end{proposition}
\begin{proof}
  It is obvious that setting
  \[R_v(t)=
    \begin{cases}
      x_v, &\text{if $t$ is even,}\\
      1/x_v, &\text{if $t$ is odd}\\
    \end{cases} \]
  satisfies~\eqref{eq:toggle} for all $t$ and thus provides a (unique by Theorem~\ref{thm:arb}) solution to the universal coefficient-free $R$-system associated with $G$. 
\end{proof}

In contrast, $R$-systems with coefficients associated to bidirected digraphs possess rich and complicated structure. Our story will be very similar to the one in the previous section. Recall that the values of the universal $R$-system with coefficients are rational functions in the vertex variables $\x_V$ as well as the edge variables $\x_E$.

\def\Neigh{N}
For $v\in V$ and $t\geq-2$, we define a rational function $\Y_v(t)$. We set
\[\Y_v(-2)=x_v,\quad \Y_v(-1)=1\]
for all $v\in V$, and for $t\geq -1$ we let
\begin{equation}\label{eq:big_sum_bidirected}
  \Y_v(t+1):=\frac{\sum_{\arb\in\Arb(G,v)}\WTT(\arb)}{\left(\prod_{u\in V} \Y_u(t-1)\right)\left(\prod_{u\in V\setminus \Neigh(v)}\Y_u(t)\right)},
\end{equation}
where
\[\Neigh(v)=\{v\}\cup\{u\in V\mid (v,u)\in E\}\]
and 
\[\WTT(\arb):=\prod_{(a,b)\in E_v} 
  \begin{cases}
    \Y_b(t-1), &\text{if $\arb(a)=b$,}\\
    \Y_b(t), &\text{otherwise.}\\
  \end{cases}\]
Here again $E_v=\{(a,b)\in E\mid a\neq v\}$. So far everything is almost identical to the formulas we had in the previous section. Note however that for bidirected digraphs, $\WTT(\arb)$ is always divisible by
\[\Y_v(t-1)\left(\prod_{u\in V\setminus \Neigh(v)}\Y_u(t)\right)\]
because for any $u\in V\setminus \Neigh(v)$, we have $\arb(w)\neq u$ for $w=\arb(u)$, and thus the edge $(w,u)\in E_v$ contributes $\Y_u(t)$ to $\WTT(\arb)$. Thus the ``actual'' denominator in~\ref{eq:big_sum_bidirected} is $\prod_{u\neq v} \Y_u(t-1)$. 

\begin{conjecture}\label{conj:bidirected}
For any bidirected digraph $G$ with at least $3$ vertices\footnote{If $G$ has two vertices then it is a directed cycle and this example has been considered in Section~\ref{sec:small-examples}.}, the rational functions $\Y_v(t)$ are pairwise coprime Laurent polynomials in $(\x_V,\x_E)$.
\end{conjecture}

\def\degs{\delta}
\def\sizes{\theta}
\begin{proposition}\label{prop:bidirected}
  Suppose that Conjecture~\ref{conj:bidirected} holds for a bidirected digraph $G$ with at least three vertices. Then the following is true:
  \begin{enumerate}[\normalfont (i)]
  \item\label{item:bidirected_R} The values of the universal $R$-system with coefficients associated with $G$ are given by
\begin{equation}\label{eq:bidirected_R}
    R_v(t+1)=\frac{\Y_{v}(t-1)}{\Y_{v}(t)}
\end{equation}
    for all $t\geq -1$. In particular, it has the singularity confinement property, and $\Y(t)$ is a strong $\tau$-sequence for this $R$-system.
  \item\label{item:bidirected_entropy} Suppose that $\Gnot$ is \emph{regular}, i.e., each vertex of $\Gnot$ has the same degree $d$ in $\Gnot$. Let
    \[h:=n(d-2)+2,\quad\text{and}\quad \l:=\frac12\left(h+\sqrt{h^2-4}\right).\]
    Then the $R$-system associated with $G$ has algebraic entropy equal to $\log(\l)$. In particular, it is zero if and only if $\Gnot$ is a cycle (i.e., $d=2$). The Laurent polynomial $\Y_v(t)$ is the sum of  $\kappa^{\sizes(t)}$ Laurent monomials each of degree $\degs(t)$ in $(\x_V,\x_E)$, where $\kappa=\Arb(G,u)$ is the complexity of $G$ and $\degs(t),\sizes(t)$ are sequences of integers defined by
    \begin{equation*}
      \begin{split}
    \degs(-2)&=1,\quad \degs(-1)=0,\quad \degs(t+1)=h\degs(t)-\degs(t-1)+(n-1),\quad \text{for $t\geq -1$};\\
    \sizes(-2)&=0,\quad \sizes(-1)=0,\quad \sizes(t+1)=h\sizes(t)-\sizes(t-1)+1,\quad \text{for $t\geq -1$}.
      \end{split}
    \end{equation*}
    Just as before, $n$ denotes the number of vertices of $G$.

  \item\label{item:bidirected_P} For each $v\in V$ and $t\geq -1$, the polynomials $\Y_u(t)$ satisfy
    \begin{equation}\label{eq:P_recurrence_bidirected}
\resizebox{1\textwidth}{!} 
{
$\displaystyle
\frac{\displaystyle \sum_{(u,v)\in E}\wtt u v\Y_u(t+1)\prod_{\substack{(u',v)\in E\\u'\neq u}}\Y_{u'}(t)}
  {\Y_v(t+1)}=
  \frac{\displaystyle \sum_{(v,w)\in E}\wtt v w\Y_w(t-1)\prod_{\substack{(v,w')\in E\\w'\neq w}}\Y_{w'}(t)}
  {\Y_v(t-1)}.$
  }
    \end{equation}
  \end{enumerate}
\end{proposition}

We denote both sides of~\eqref{eq:P_recurrence_bidirected} by $Z_v(t)$ (in particular, we let $Z_v(-2)$ be the left hand side of~\eqref{eq:P_recurrence_bidirected}). Note that $\frac{Z_v(t)}{\Y_v(t)}$ is \emph{not} a conserved quantity even when the degree of $v$ is equal to $2$: if we denote by $a$ and $b$ the only neighbors of $v$ in $G$ then we have
\begin{equation*}
  \begin{split}
\frac{Z_v(t)}{\Y_v(t)}&=\frac{\wtt{a} v\Y_a(t+1)\Y_b(t)+\wtt b v \Y_b(t+1)\Y_a(t)}{\Y_v(t+1)\Y_v(t)},\\
\frac{Z_v(t+1)}{\Y_v(t+1)}&=\frac{\wtt v a\Y_a(t)\Y_b(t+1)+\wtt v b \Y_b(t)\Y_a(t+1)}{\Y_v(t+1)\Y_v(t)}.
  \end{split}
\end{equation*}
The two expressions become equal if the edge weights satisfy $\wt(a,v)=\wt(v,b)$ and $\wt(b,v)=\wt(v,a)$.

\begin{conjecture}\label{conj:Z_Laurent}
For each $v\in V$ and $t\geq -2$, $Z_v(t)$ is a Laurent polynomial in $(\x_V,\x_E)$.
\end{conjecture}

We also note that Bellon-Viallet~\cite{BV} conjectured that the algebraic entropy is always the logarithm of an algebraic integer. We see that it is certainly true for the case of regular graphs as follows from Proposition~\ref{prop:bidirected}~\eqref{item:bidirected_entropy} (assuming Conjecture~\ref{conj:bidirected}).

\begin{proof}[Proof of Proposition~\ref{prop:bidirected}]
  Just as for toric digraphs, Part~\eqref{item:bidirected_R} is a restatement of the toggle relations~\eqref{eq:toggle}. The recurrences for $\degs(t)$ and $\sizes(t)$ are clear from~\eqref{eq:big_sum_bidirected}. The fact that these sequences grow exponentially in the case $d>2$ follows since $\l>1$ becomes the dominant eigenvalue for the corresponding linear recurrence. When $d=2$ and $\l=1$, we have
  \begin{equation}\label{eq:sizes_degs}
\sizes(t)={t+2\choose 2},\quad \degs(t)=\frac12((n-1)t^2+(3n-5)t+(2n-4)),
  \end{equation}
  which is easily proven by induction. And thus the statement about the algebraic entropy follows in this case as well.

  Finally,~\eqref{item:bidirected_P} is shown in a similar way to its counterpart for toric digraphs so we omit the proof.
\end{proof}

We also give an analogue of Proposition~\ref{prop:toric_divide}:
\begin{proposition}\label{prop:bidirected_divide}
Let $t\geq -2$ and suppose that $\Y_u(t')$ is a Laurent polynomial for all $u\in V$ and $-2\leq t'\leq t$. Assume that any two such Laurent polynomials are coprime. Suppose in addition that $Z_u(t')$ is a Laurent polynomial for all $u\in V, -2\leq t'\leq t$. Then $\Y_v(t+1)$ is also a Laurent polynomial for all $v\in V$.
\end{proposition}
\begin{proof}
  Fix a vertex $v\in V$ and let
  \[\bigsum:=\sum_{\arb\in\Arb(G,v)}\WTT(\arb).\]
  We need to show that $\bigsum$ is divisible by $\Y_u(t-1)$ for each vertex $u\neq v$. Let $u$ be such a vertex. Then the only arborescences $\arb\in\Arb(G,v)$ such that  $\Y_u(t-1)$ does not appear as a factor in $\WTT(\arb)$ are the ones where $u$ has indegree zero. Let us say that two such arborescences $\arb,\arb'\in\Arb(G,v)$ are \emph{equivalent} if we have $\arb(w)=\arb'(w)$ for all $w\in V\setminus\{u,v\}$, and let $C$ be an equivalence class of such arborescences. Then the number of elements in $C$ equals the number of neighbors of $u$ in $\Gnot$ and 
  $\sum_{\arb\in C} \WTT(\arb)$ is divisible by
  \[ \sum_{(u,w)\in E}\wtt u w\Y_w(t-1)\prod_{\substack{(u,w')\in E\\w'\neq w}}\Y_{w'}(t).\]
  Thus by~\eqref{eq:P_recurrence_bidirected}, it is divisible by $\Y_u(t-1)$, and therefore so is $\bigsum$. We are done with the proof.
\end{proof}

\begin{example}
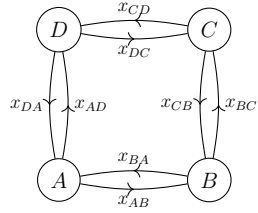
\begin{figure}
  \def\nodesc{0.7}
  \def\labelsc{1}
  \def\tikzscx{2}
  \def\tikzscy{2}
  \def\arrsc{0.6}
  \def\sclbx{0.7}
  \def\sclbxbig{1.4}
  \def\bnd{10}
  \def\nodepos{0.5}
   \begin{tikzpicture}[xscale=\tikzscx,yscale=\tikzscy]
    \node[scale=\nodesc,draw,circle] (A) at (0,0) {$A$};
    \node[scale=\nodesc,draw,circle] (B) at (1,0) {$B$};
    \node[scale=\nodesc,draw,circle] (C) at (1,1) {$C$};
    \node[scale=\nodesc,draw,circle] (D) at (0,1) {$D$};
    \draw[postaction={decorate}] (A) to[bend right=\bnd] node[pos=\nodepos,below,scale=\arrsc] {$x_{AB}$} (B) ;
    \draw[postaction={decorate}] (B) to[bend right=\bnd] node[pos=\nodepos,right,scale=\arrsc] {$x_{BC}$} (C);
    \draw[postaction={decorate}] (C) to[bend right=\bnd] node[pos=\nodepos,above,scale=\arrsc] {$x_{CD}$} (D);
    \draw[postaction={decorate}] (D) to[bend right=\bnd] node[pos=\nodepos,left,scale=\arrsc] {$x_{DA}$} (A);
    \draw[postaction={decorate}] (D) to[bend right=\bnd] node[pos=\nodepos,below,scale=\arrsc] {$x_{DC}$} (C) ;
    \draw[postaction={decorate}] (A) to[bend right=\bnd] node[pos=\nodepos,right,scale=\arrsc] {$x_{AD}$} (D);
    \draw[postaction={decorate}] (B) to[bend right=\bnd] node[pos=\nodepos,above,scale=\arrsc] {$x_{BA}$} (A);
    \draw[postaction={decorate}] (C) to[bend right=\bnd] node[pos=\nodepos,left,scale=\arrsc] {$x_{CB}$} (B);
    
  \end{tikzpicture}

  \caption{\label{fig:bidirected_4} A bidirected $4$-cycle $G$ shown together with the edge weights.}
\end{figure}
Let $G$ be a bidirected $4$-cycle shown in Figure~\ref{fig:bidirected_4}. Denote the vertices of $G$ by $A,B,C,D$, the vertex variables by $\x_V:=(a,b,c,d)$, and the edge variables by

\[\x_E:=(x_{AB},x_{BC},x_{CD},x_{DA},x_{AD},x_{DC},x_{CB},x_{BA}).\]
Let us compute the first few $\Y$-polynomials:

\[\Y_A(-2)=a;\quad \Y_B(-2)=b;\quad \Y_C(-2)=c;\quad\Y_D(-2)=d;\]
\[\Y_A(-1)=\Y_B(-1)=\Y_C(-1)=\Y_D(-1)=1;\]
Using~\eqref{eq:big_sum_bidirected}, we obtain
\[\Y_A(0)=\frac{x_{BA}x_{CB}x_{DA}ab + x_{BA}x_{CB}x_{DC}bc + x_{BA}x_{CD}x_{DA}ad + x_{BC}x_{CD}x_{DA}cd}{bcd},\]

  \[\Y_B(0)=\frac{x_{AB}x_{CB}x_{DA}ab + x_{AB}x_{CB}x_{DC}bc + x_{AB}x_{CD}x_{DA}ad + x_{AD}x_{CB}x_{DC}cd}{acd},\]
  \[\Y_C(0)=\frac{x_{AB}x_{BC}x_{DA}ab + x_{AB}x_{BC}x_{DC}bc + x_{AD}x_{BA}x_{DC}ad + x_{AD}x_{BC}x_{DC}cd}{abd},\]
  \[\Y_D(0)=\frac{x_{AD}x_{BA}x_{CB}ab + x_{AB}x_{BC}x_{CD}bc + x_{AD}x_{BA}x_{CD}ad + x_{AD}x_{BC}x_{CD}cd}{abc}.\]
We can now calculate the $Z$-polynomials. For example, for $v=A$, we get
  \[Z_A(-2)=x_{BA}d+x_{DA}b;\quad Z_A(-1)=\frac{x_{AB}b+x_{AD}d}{a},\]
  where $Z_A(-2)$ (resp., $Z_A(-1)$) was computed using the left (resp., right) hand side of~\eqref{eq:P_recurrence_bidirected}. We can now verify that the left hand side of~\eqref{eq:P_recurrence_bidirected} yields the same result for $t=-1$:
  \[Z_A(-1)=\frac{x_{BA}\Y_B(0)+x_{DA}\Y_D(0)}{\Y_A(0)}.\]
  We see that the cancellation indeed happens and we get $\frac{x_{AB}b+x_{AD}d}{a}$, confirming both~\eqref{eq:P_recurrence_bidirected} and Conjecture~\ref{conj:Z_Laurent}. Let us now attempt to calculate $\Y_A(1)$. We first write down a general formula for $\Y_A(t+1)$. Using~\eqref{eq:big_sum_bidirected}, after dividing through by $\Y_A(t-1)\Y_C(t)$, we get that $\Y_A(t+1)$ is equal to
\begin{equation*}
\resizebox{1\textwidth}{!} 
{
$\displaystyle\frac{\left(\begin{aligned} &x_{BA}x_{CB}x_{DC}\Y_B(t-1)\Y_C(t-1)\Y_A(t)\Y_D(t)+x_{DA}x_{BA}x_{CB}\Y_A(t-1)\Y_B(t-1)\Y_C(t)\Y_D(t)+\\
        &+x_{BA}x_{DA}x_{CD}\Y_A(t-1)\Y_D(t-1)\Y_B(t)\Y_C(t)+x_{BC}x_{CD}x_{DA}\Y_C(t-1)\Y_D(t-1)\Y_B(t)\Y_A(t)
      \end{aligned}\right)}{\Y_B(t-1)\Y_C(t-1)\Y_D(t-1)}$
}
\end{equation*}

Thus for $t=0$ we obtain
  \[\begin{aligned}\Y_A(1)= &x_{BA}x_{CB}x_{DC}\Y_A(0)\Y_D(0)+x_{DA}x_{BA}x_{CB}\Y_C(0)\Y_D(0)+\\
        &x_{BA}x_{DA}x_{CD}\Y_B(0)\Y_C(0)+x_{BC}x_{CD}x_{DA}\Y_B(0)\Y_A(0).
      \end{aligned}\]
    Let us explain how the proof of Proposition~\ref{prop:bidirected_divide} works for $\Y_A(t+1)$. We need to divide by $\Y_B(t-1)\Y_C(t-1)\Y_D(t-1)$. Out of four terms in the numerator, two already are divisible by $\Y_B(t-1)$. The sum of the other two is equal to
    \[x_{DA}x_{CD}\Y_D(t-1)\Y_B(t)(x_{BA}\Y_A(t-1)\Y_C(t)+x_{BC}\Y_C(t-1)\Y_A(t))\]
    which by~\eqref{eq:P_recurrence_bidirected} equals
    \[x_{DA}x_{CD}\Y_D(t-1)\Y_B(t)Z_B(t)\Y_B(t-1).\]
This shows that the numerator $\bigsum$ of the expression for $\Y_A(t+1)$ is divisible by $\Y_B(t-1)$ if $Z_B(t)$ is a Laurent polynomial. Similarly, one shows that it is divisible by $\Y_C(t-1)$ and $\Y_D(t-1)$, but to conclude that it is divisible by the product, we need to assume that $\Y_B(t-1)$, $\Y_C(t-1)$, and $\Y_D(t-1)$ are coprime.

Let us also mention that $\Y_A(0)$ and $\Y_A(1)$ contain $4=\kappa^1$ and $64=\kappa^3$ Laurent monomials respectively, where $\kappa=4$ is the complexity of $\Gnot$, i.e., the number of its spanning trees. Each Laurent monomial is of degree
\[2=\delta(0)=\frac12(0+0+4),\quad\text{resp.,}\quad 7=\delta(1)=\frac12(3+7+4),\]
in agreement with~\eqref{eq:sizes_degs}.
\end{example}

One particularly interesting case is that of the complete bidirected digraph $K_n$ with $n$ vertices, since any $R$-system is obtained from the $R$-system associated with $K_n$ after specializing some edge weights to $0$ and some edge weights to $1$. Thus understanding how the $\Y$-polynomials for $K_n$ factor into irreducibles after such a specialization is equivalent to understanding the behavior of the specialized $R$-system. Note that $K_n$ is regular of degree $n-1$ and thus we obtain the following upper bound on the algebraic entropy of any $R$-system:

\begin{corollary}
  Suppose that Conjecture~\ref{conj:bidirected} is true for $K_n$ and assume that the limit $d(G)$ given by~\eqref{eq:entropy_def} exists for a digraph $G$ with $n$ vertices. Then this limit is finite and satisfies
  \[0\leq d(G)\leq \log\left(\frac{h+\sqrt{h^2-4}}2\right),\quad \text{where}\quad h=n^2-3n+2.\]
\end{corollary}

\subsection{Antimorphic digraphs}

We note that there is a certain class of digraphs that includes both toric and bidirected digraphs and, according to our computations, the $R$-systems associated with such graphs exhibit singularity confinement with a similar pattern.

\begin{definition}
For a digraph $G=(V,E)$, a bijection $\eta: V \rightarrow V$ is called an {\it {antimorphism}} if  for any $v,u \in V$, we have an edge $(v,u) \in E$ if and only if we have an edge $(u,\eta(v)) \in E$.
\end{definition}

Note that we allow $\eta(v)=v$. We call a digraph {\it {antimorphic}} if it possesses an antimorphism $\eta$. 

\begin{example}
 Toric digraphs are antimorphic with $\eta(v) = v + e$ for all $v \in V$.
\end{example}

\begin{example}
 Bidirected digraphs are antimorphic with $\eta(v)=v$ for all $v \in V$.
\end{example}

\begin{figure}

\begin{tabular}{cc}
\def\scl{0.3}
\def\rad{2}
\def\textscl{1}
\begin{tikzpicture}
  \foreach \i in {1,2,...,8}{
\node[draw, circle, scale=\scl, fill=black]  (\i) at ({112.5-45*\i}:\rad) { };
\node[scale=\textscl,anchor={180+112.5-45*\i}]  (L\i) at (\i.center) {$v_{\i}$};
}
\draw[postaction={decorate}] (8)--(1);
\draw[postaction={decorate}] (1)--(2);
\draw[postaction={decorate}] (2)--(3);
\draw[postaction={decorate}] (3)--(4);
\draw[postaction={decorate}] (4)--(5);
\draw[postaction={decorate}] (5)--(6);
\draw[postaction={decorate}] (6)--(7);
\draw[postaction={decorate}] (7)--(8);

\draw[postaction={decorate}] (8)--(2);
\draw[postaction={decorate}] (1)--(3);
\draw[postaction={decorate}] (2)--(4);
\draw[postaction={decorate}] (3)--(5);
\draw[postaction={decorate}] (4)--(6);
\draw[postaction={decorate}] (5)--(7);
\draw[postaction={decorate}] (6)--(8);
\draw[postaction={decorate}] (7)--(1);

\draw[postaction={decorate}] (8)--(3);
\draw[postaction={decorate}] (1)--(4);
\draw[postaction={decorate}] (2)--(5);
\draw[postaction={decorate}] (3)--(6);
\draw[postaction={decorate}] (4)--(7);
\draw[postaction={decorate}] (5)--(8);
\draw[postaction={decorate}] (6)--(1);
\draw[postaction={decorate}] (7)--(2);
\end{tikzpicture}

&

\def\scl{0.3}
\def\bnd{15}
\begin{tikzpicture}[yscale=1.5,xscale=1.2]
  \node[draw, circle, scale=\scl, fill=black,label=above:{$v_1$}] (1) at (0,3) { };
  \node[draw, circle, scale=\scl, fill=black,label=left:{$v_2$}] (2) at (-1,2) { };
  \node[draw, circle, scale=\scl, fill=black,label=right:{$v_3$}] (3) at (0,2) { };
  \node[draw, circle, scale=\scl, fill=black,label=right:{$v_4$}] (4) at (1,2) { };
  \node[draw, circle, scale=\scl, fill=black,label=left:{$v_5$}] (5) at (-1,1) { };
  \node[draw, circle, scale=\scl, fill=black,label=right:{$v_6$}] (6) at (0,1) { };
  \node[draw, circle, scale=\scl, fill=black,label=right:{$v_7$}] (7) at (1,1) { };
  \node[draw, circle, scale=\scl, fill=black,label=below:{$v_8$}] (8) at (0,0) { };
  \draw[postaction={decorate}] (1) to[bend right=\bnd] (2);
  \draw[postaction={decorate}] (2) to[bend right=\bnd] (1);
  \draw[postaction={decorate}] (1) to[bend right=\bnd] (3);
  \draw[postaction={decorate}] (3) to[bend right=\bnd] (1);
  \draw[postaction={decorate}] (1) to[bend right=\bnd] (4);
  \draw[postaction={decorate}] (4) to[bend right=\bnd] (1);

  \draw[postaction={decorate}] (8) to[bend right=\bnd] (5);
  \draw[postaction={decorate}] (5) to[bend right=\bnd] (8);
  \draw[postaction={decorate}] (8) to[bend right=\bnd] (6);
  \draw[postaction={decorate}] (6) to[bend right=\bnd] (8);
  \draw[postaction={decorate}] (8) to[bend right=\bnd] (7);
  \draw[postaction={decorate}] (7) to[bend right=\bnd] (8);

  \draw[postaction={decorate}] (2) -- (5);
  \draw[postaction={decorate}] (5) -- (3);
  \def\positn{0.65}
  \draw[postaction={decorate}] (3) -- (6);
  \def\positn{0.5}
  \draw[postaction={decorate}] (6) -- (4);
  \draw[postaction={decorate}] (4) -- (7);
  \def\positn{0.6}
  \draw[postaction={decorate}] (7) -- (2);
  \def\positn{0.5}
\end{tikzpicture}
\end{tabular}
    \caption{Two antimorphic digraphs.}
    \label{fig:toggle8}
\end{figure}
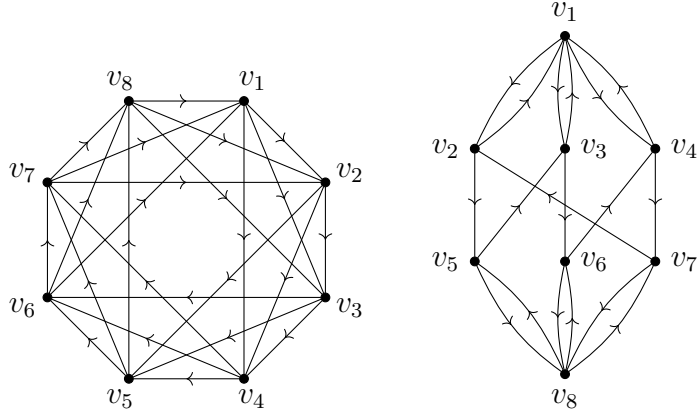

\begin{example}
 In Figure \ref{fig:toggle8}, we can see two more examples of antimorphic digraphs. For the digraph on the left, the antimorphism is given by $\eta(v_i) = v_{i+4}$ where the indices are taken modulo $8$. For the digraph on the right, we have $\eta(v_1)=v_1$, $\eta(v_8) = v_8$, $\eta(v_2)=v_3$, $\eta(v_3)=v_4$, $\eta(v_4)=v_2$, $\eta(v_5)=v_6$,
 $\eta(v_6)=v_7$, $\eta(v_7)=v_5$.
\end{example}

\begin{conjecture}
  $R$-systems associated with antimorphic digraphs possess strong $\tau$-sequences $Y(t)$ with the substitution given by
  \[R_v(t+1)=\frac{Y_{\eta^t(v)}(t-1)}{Y_{\eta^t(v)}(t)}.\]
\end{conjecture}

\bibliographystyle{alpha}
\bibliography{toggle_bib}{}

\end{document}